\documentclass[final,12pt]{amsart}
    \title[Torsion pairs]{Torsion pairs in a triangulated category generated by a spherical object}

\author{Raquel Coelho Sim\~oes}
\address{Centro de An\'alise Funcional, Estruturas Lineares e Aplica\c{c}\~oes, Faculdade de Ci\^encias da Universidade de Lisboa, Campo Grande, Edif\'icio C6, Piso 2, 1749-016, Lisboa, Portugal}
\email{rcoelhosimoes@campus.ul.pt}

\author{David Pauksztello}
\address{School of Mathematics, The University of Manchester, Oxford Road, Manchester, M13 9PL, United Kingdom.}
\email{david.pauksztello@manchester.ac.uk}

\usepackage{amssymb}
\usepackage{latexsym}
\usepackage{amsmath}
\usepackage{mathrsfs}
\usepackage[all]{xy}
\usepackage{enumerate}	
\usepackage{enumitem}
\usepackage{paralist}
\usepackage{graphicx}

\usepackage{verbatim}

\usepackage[usenames,dvipsnames]{xcolor}
\usepackage{hyperref}
\hypersetup{
    colorlinks,
    linkcolor={red!50!black},
    citecolor={blue!50!black},
    urlcolor={blue!80!black}
}

\newcommand{\harxiv}[1]{ \href{http://arxiv.org/abs/#1}{\texttt{arXiv:#1}}}
\newcommand{\hyref}[2]{ \hyperref[#2]{#1~\ref*{#2}} }

\usepackage[notref, notcite]{showkeys} % for draft versions only

\theoremstyle{plain}
\newtheorem{theorem}{Theorem}[section]
\newtheorem{lemma}[theorem]{Lemma}
\newtheorem{corollary}[theorem]{Corollary}
\newtheorem{proposition}[theorem]{Proposition}
\newtheorem{introtheorem}{Theorem}

\theoremstyle{definition}
\newtheorem{remark}[theorem]{Remark}

\newtheorem{definition}[theorem]{Definition} % numbered definition.
\newtheorem{definitions}[theorem]{Definitions} % numbered definitions.
\newtheorem{setup}[theorem]{Setup}
\newtheorem*{notation}{Notation}

\newcommand{\Case}[1]{\medskip\noindent\emph{Case #1:}}
\newcommand{\Subcase}[1]{\medskip\noindent\emph{Subcase #1:}}

\DeclareMathAlphabet{\mathpzc}{OT1}{pzc}{m}{it}

\newcommand{\cC}{\mathscr{C}}
\newcommand{\cT}{\mathscr{T}}

\newcommand{\sA}{\mathsf{A}}
\newcommand{\sC}{\mathsf{C}}
\newcommand{\sF}{\mathsf{F}}
\newcommand{\sS}{\mathsf{S}}
\newcommand{\sT}{\mathsf{T}}
\newcommand{\sX}{\mathsf{X}}
\newcommand{\sY}{\mathsf{Y}}

\newcommand{\Db}{\mathsf{D}^b}

\newcommand{\bN}{\mathbb{N}}
\newcommand{\bZ}{\mathbb{Z}}

\newcommand{\mcP}{\mathcal{P}}

\newcommand{\tA}{\mathtt{A}}
\newcommand{\tE}{\mathtt{E}}
\newcommand{\tX}{\mathtt{X}}

\newcommand{\Tw}{\mathsf{T}_w}
\newcommand{\Cw}{\mathsf{C}_w}
\newcommand{\Cm}{\mathsf{C}_m}
\newcommand{\Ainf}{A_{\infty}}

\renewcommand{\geq}{\geqslant}
\renewcommand{\leq}{\leqslant}

\DeclareMathOperator{\coker}{\mathsf{coker}}

\newcommand{\cone}[1]{\mathsf{cone}(#1)}
\newcommand{\cocone}[1]{\mathsf{cocone}(#1)}
\DeclareMathOperator{\Hom}{\mathsf{Hom}}
\DeclareMathOperator{\Ext}{\mathsf{Ext}}
\DeclareMathOperator{\id}{{\mathsf{id}}}
\newcommand{\ind}[1]{\mathsf{ind}(#1)}
\renewcommand{\mod}[1]{\mathsf{mod}(#1)}
\newcommand{\inj}[1]{\mathsf{inj}(#1)}
\DeclareMathOperator{\add}{\mathsf{add}}
\DeclareMathOperator{\modulo}{\mathrm{mod}}
\DeclareMathOperator{\ext}{\mathrm{ext}}

\newcommand{\kk}{{\mathbf{k}}}

\newcommand{\SSS}{\mathbb{S}}

\newcommand{\thick}[2]{\mathsf{thick}_{#1}(#2)}

\newcommand{\extn}[1]{\langle #1 \rangle}

\newcommand{\too}{\longrightarrow}
\newcommand{\rightlabel}[1]{\stackrel{#1}{\longrightarrow}}
\newcommand{\isoto}{\rightlabel{\sim}}
\newcommand*{\longhookrightarrow}{\ensuremath{\lhook\joinrel\longrightarrow}}
\newcommand*{\longtwoheadrightarrow}{\ensuremath{\relbar\joinrel\twoheadrightarrow}}
\newcommand{\intoo}{\longhookrightarrow}
\newcommand{\ontoo}{\longtwoheadrightarrow}

\newcommand{\tri}[3]{#1\rightarrow #2\rightarrow #3\rightarrow \Sigma #1}
\newcommand{\trilabel}[4]{#1\stackrel{#4}{\longrightarrow} #2\longrightarrow #3\longrightarrow \Sigma #1}
\newcommand{\trilabels}[6]{#1\stackrel{#4}{\longrightarrow} #2\stackrel{#5}{\longrightarrow} #3\stackrel{#6}{\longrightarrow} \Sigma #1}

\newcommand{\rayfrom}[1]{\mathsf{ray}_{\! +}(#1)}
\newcommand{\rayto}[1]{\mathsf{ray}_{\! -}(#1)}
\newcommand{\exray}[2]{\mathsf{exray}_{#1}(#2)}

\newcommand{\coray}[1]{\mathsf{coray}(#1)}            
\newcommand{\corayfrom}[1]{\mathsf{coray}_{\! +}(#1)}
\newcommand{\corayto}[1]{\mathsf{coray}_{\! -}(#1)}
\newcommand{\excoray}[2]{\mathsf{excoray}_{#1}(#2)}

\newcommand{\homfrom}[1]{\mathsf{H}^+(#1)}
\newcommand{\homto}[1]{\mathsf{H}^-(#1)}

\newcommand{\extfrom}[1]{\mathsf{E}^+(#1)}
\newcommand{\extto}[1]{\mathsf{E}^-(#1)}

\newcommand{\arc}[1]{\mathtt{#1}}

\entrymodifiers={+!!<0pt,\fontdimen22\textfont2>}

\newcommand{\sqmat}[4]{{\big[\genfrac{.}{.}{0pt}{1}{#1}{#3} \,
                        \genfrac{.}{.}{0pt}{1}{#2}{#4}\big] }}
\newcommand{\colmat}[2]{{\big[\genfrac{.}{.}{0pt}{1}{#1}{#2}\big] }}
\newcommand{\rowmat}[2]{\big[#1 \ #2\big]}
\newcommand{\rowvec}[3]{\big[#1 \ #2 \ #3\big]}
\newcommand{\rowvecc}[4]{\big[#1 \ #2 \ #3 \ #4\big]}
\newcommand{\colvec}[3]{\Bigg[\begin{smallmatrix}#1\\#2\\#3\end{smallmatrix}\Bigg]}

\newcommand{\arr}{\ar[ur] \ar[dr]}
\newcommand{\arup}{\ar@/^/[u]}
\newcommand{\ardn}{\ar@/^/[d]}

\newcommand{\smxy}[1]{{\text{\tiny$#1$}}}

\newcommand{\coloneqq}{\mathrel{\mathop:}=}
\newcommand{\eqqcolon}{=\mathrel{\mathop:}}

\newcommand{\rfountain}[2]{\mathbb{RF}(#1;#2)}
\newcommand{\lfountain}[2]{\mathbb{LF}(#1;#2)}

\newcommand{\dd}[1]{\mathbf{d}(#1)}

\usepackage{tikz}
	\usetikzlibrary{decorations.pathmorphing}
        \usetikzlibrary{topaths}
        \usetikzlibrary{automata,arrows}
        \usetikzlibrary{shapes.geometric} 

\begin{document}

\begin{abstract}
We extend Ng's characterisation of torsion pairs in the $2$-Calabi-Yau triangulated category generated by a $2$-spherical object to the characterisation of torsion pairs in the $w$-Calabi-Yau triangulated category, $\Tw$, generated by a $w$-spherical object for any $w \in \bZ$. Inspired by the combinatorics of $\Tw$ for $w\leq -1$, we also characterise the torsion pairs in certain negative Calabi-Yau  orbit categories of the bounded derived category of the path algebra of Dynkin type $A$.
\end{abstract}



\maketitle

{\small
\setcounter{tocdepth}{1}
\tableofcontents
}

%============================================================================
% Introduction
\addtocontents{toc}{\protect{\setcounter{tocdepth}{-1}}}  % No toc entry for Introduction
\section*{Introduction} 
\addtocontents{toc}{\protect{\setcounter{tocdepth}{1}}}   % but enable toc entries for other sections
%============================================================================

Calabi-Yau (CY) triangulated categories are triangulated categories that satisfy an important duality. They are becoming increasingly important throughout  mathematics and physics, for example as $3$-CY categories arising from Calabi-Yau threefolds in algebraic geometry and string theory, to $3$-CY categories arising in representation theory coming from quivers with potential. Of particular importance in representation theory are ($2$-)cluster categories, which provide categorifications of important aspects of the theory of cluster algebras. There are higher analogues, so-called $w$-cluster categories for $w \geq 2$, which are $w$-CY. These give rise to an important family of categories of positive CY dimension which satisfy many interesting and important homological and combinatorial properties. 

Throughout this article $\kk$ will be an algebraically closed field. Let $\sT$ be a $\kk$-linear triangulated category and $w \in \bZ$. An object $s \in \sT$ is \emph{$w$-spherical} if it is a $w$-Calabi-Yau object and its graded endomorphism algebra  is given by
\[
\Hom^{\bullet}(s,s) = \kk[x]/(x^2), \text{ where } x \text{ sits in cohomological degree } -w.
\]
In particular, $s$ has the `same cohomology' as the $w$-sphere.
We refer the reader to Section~\ref{sec:spherical} for a more precise definition.

Let $\Tw$ be a $\kk$-linear triangulated category that is idempotent complete and generated by a $w$-spherical object.
The $\Tw$ constitute a family of categories which are $w$-CY whose structure is sufficiently simple to allow concrete computation.
As such, they provide a `natural laboratory' in which to explore the properties of CY triangulated categories, as witnessed by the intense recent interest in these categories; see \cite{CS13,FY, HJ, Jorgensen-AR-theory, KYZ}. Indeed, for $w\geq 2$, $\Tw$ occurs naturally as a $w$-cluster category of type $\Ainf$.

Owing to their importance and ubiquity, much work has been carried out on understanding triangulated categories of positive CY-dimension. However, very little work has been carried out on understanding the properties of triangulated categories of negative CY-dimension, although there is the beginning of a theory emerging in \cite{CS10, CS11, CS13, HJY}. In \cite{HJY}, it was shown that for $w \geq 1$, the category $\Tw$ has one family of bounded t-structures and no bounded co-t-structures, whilst for $w \leq 0$ the opposite is true. This shows that there are important homological differences between triangulated categories of positive and negative CY dimension.

Both t-structures and co-t-structures are examples of torsion pairs in triangulated categories \cite{IY}. Torsion pairs have long been studied in representation theory in the context of tilting theory to provide important structural information about module categories of finite-dimensional algebras and a means of comparing different categories. In the context of cluster-tilting theory, they can be seen as a generalisation of cluster-tilting objects; furthermore, they admit a mutation theory \cite{ZZ}. Thus characterising and understanding torsion pairs is central to understanding the structure of triangulated categories.

In the context of cluster theory, geometric/combinatorial models are a useful tool, first arising in \cite{CCS, Schiffler}. Combinatorial models for the $\Tw$ were obtained by Holm and J\o rgensen in \cite{HJ1,HJ} and employed by Ng in \cite{Ng} to characterise torsion pairs in $\sT_2$. Building on Ng's ideas, characterisations of torsion pairs have since been given in various settings, see \cite{Baur-Buan-Marsh,HJR-A,HJR-D,HJR-tubes}, and used for detailed studies of their mutation theories \cite{Gratz,ZZ}.

The combinatorial models involve setting up a correspondence between indecomposable objects of the category and certain `admissible' arcs or diagonals of some geometric object.
For $\Tw$ with $w \neq 1$, the combinatorial model consists of `$(w-1)$-admissible' arcs of the $\infty$-gon; see Section~\ref{sec:contra} for precise details.
It is a well-known consequence of the $2$-Calabi-Yau property of cluster categories that the crossing of arcs corresponds to the existence of a non-trivial extension between the corresponding indecomposable objects. Given two crossing arcs, the admissible Ptolemy arcs are defined to be the admissible arcs connecting the endpoints of the two crossing arcs.

We extend Ng's characterisation of torsion pairs for $\sT_2$ to the entire family:

\begin{introtheorem} \label{thm:A}
Let $\sX$ be a full additive subcategory of $\Tw$ for $w \neq 1$ and $\tX$ be the corresponding set of arcs in the appropriate combinatorial model of $\Tw$. Then $(\sX,\sX^{\perp})$ is a torsion pair in $\Tw$ if and only if
\begin{compactenum}
\item for $w \geq 2$, any so-called `right fountain' in $\tX$ is a so-called `left fountain' and $\tX$ is closed under taking admissible Ptolemy arcs.
\item for $w \leq 0$, any left fountain in $\tX$ is a right fountain and $\tX$ is closed under taking admissible Ptolemy arcs and `modified Ptolemy' arcs. 
\end{compactenum}
\end{introtheorem}

See Sections~\ref{sec:Ptolemy} and \ref{sec:Ptolemy-0} for precise statements.
The statement for $w \geq 2$ is somewhat expected, although it is not a completely straightforward generalisation of Ng's characterisation for $w=2$ in \cite{Ng} because crossings of arcs instead correspond to the existence of some higher extension instead of simply extensions, which requires a substantially different approach from \cite{Ng}. The case $w=1$ is degenerate and does not admit such a combinatorial model; it is treated in the short Section ~\ref{sec:extensions-in-T1}. However, surprisingly, for $w \leq 0$ there is a pleasant combinatorial description using this combinatorial model.

As observed in \cite{CS13}, it turns out that the combinatorial model for $\Tw$ when $w < 0$ induces a combinatorial model on another important $w$-CY category, namely the following orbit category: $\Cw(A_n) \coloneqq \Db(\kk A_n)/\Sigma^{1-w}\tau$ for $w\leq -1$. When $w=-1$ the maximal rigid objects of this category are classified in \cite{CS11} using a different combinatorial model. With the combinatorial model of \cite{CS11}, the characterisation of torsion pairs in $\sC_{-1}(A_n)$ proved intractable. However, with the induced combinatorial model, the characterisation is tractable and gives us our second main result.

\begin{introtheorem} \label{thm:B}
Let $\sX$ be a full additive subcategory of $\Cw(A_n)$ for $w\leq -1$ and $\tX$ be the corresponding set of arcs in the combinatorial model for $\Cw(A_n)$. Then $(\sX, \sX^{\perp})$ is a torsion pair in $\Cw(A_n)$ if and only if $\tX$ is closed under taking admissible Ptolemy arcs and `modified' Ptolemy arcs.
\end{introtheorem}

It is our viewpoint that for $w \leq -1$ the categories $\Cw(A_n)$ are naturally $w$-CY, i.e. natural examples of triangulated categories having negative CY dimension. However, even in the case $w=-1$ there is some debate on the CY dimension of these categories. For example, in \cite{Dugas}, Dugas takes the CY dimension to be defined as the \emph{least positive} integer $d$ such that $\Sigma^d$ is (isomorphic to) the Serre functor. According to this definition, the CY dimension of $\sC_{-1}(A_n)$ is $2n-1$; \cite[Theorem 6.1]{Dugas}. Note, however, that in $\sC_{-1}(A_n)$, the inverse suspension $\Sigma^{-1}$ is also isomorphic to the Serre functor of $\sC_{-1}(A_n)$.

In contrast, by \cite[Proposition 2.8]{HJY}, $\Tw$ is unambiguously $w$-CY.
It was argued in \cite{CS13} that for $\Cw(A_n)$, with $w \leq -1$, the `correct' CY dimension should be $w$, owing to similarities in the combinatorics of so-called $w$-Hom-configurations in the categories $\Cw(A_n)$ and $\Tw$. We believe the similarities in the combinatorics of torsion pairs in Theorems~\ref{thm:A} and \ref{thm:B} provide further support for this viewpoint. Moreover, we believe that this means triangulated categories of negative CY dimension are more widespread than previously believed, and warrant further, systematic, study. This article should be considered as a step in this direction.

\subsection*{Acknowledgments.} 
This paper was begun while both authors were at the Leibniz Universit\"at Hannover. 
The first author gratefully acknowledges the financial support of the Riemann Center for Geometry and Physics during her stay in Hannover. 
RCS would also like to thank Funda\c{c}\~ao para a Ci\^encia e Tecnologia, for financial support through the grant SFRH/BPD/90538/2012. 
DP gratefully acknowledges the support of the EPSRC through the grant EP/K022490/1. 
We would like to thank the referee for a careful reading and useful comments and suggestions.

%=============================================================================
%Section
\section{Torsion pairs, extension closure and functorial finiteness} \label{sec:torsion}
%=============================================================================

A triangulated category $\sT$ is called \emph{Krull-Schmidt} if every object $t$ admits a direct sum decomposition $t=t_1 \oplus \cdots \oplus t_n$ into indecomposable objects, which is unique up to reordering and isomorphism. 
We shall denote the collection of (isomorphism classes of) indecomposable objects by $\ind{\sT}$. Throughout this paper all categories will be Krull-Schmidt and all subcategories will be full and additive.

A \emph{torsion pair} in $\sT$ consists of a pair of full subcategories $(\sX,\sY)$, which are closed under direct summands, and satisfy $\Hom_{\sT}(\sX,\sY)=0$ and $\sX * \sY = \sT$, 
where 
\begin{equation*} \label{star-product}
\sX * \sY \coloneqq \{t \in \sT \mid \exists \ \tri{x}{t}{y} \text{ with } x\in \sX \text{ and } y \in \sY\}.
\end{equation*}
A torsion pair is called a \emph{t-structure} when $\Sigma \sX \subseteq \sX$ ($\Leftrightarrow \Sigma^{-1} \sY \subseteq \sY$); see \cite{BBD}. It is called a \emph{co-t-structure} (or \emph{weight structure}) when $\Sigma^{-1} \sX \subseteq \sX$ ($\Leftrightarrow \Sigma \sY \subseteq \sY$); see \cite{Bondarko, P}.
If $\sT$ is Krull-Schmidt, a torsion pair $(\sX,\sY)$ is called \emph{split} if for any $t \in \ind{\sT}$ we have either $t \in \sX$ or $t \in \sY$. 

A subcategory $\sX$ of $\sT$ is \emph{closed under extensions} or \emph{extension-closed} if given any distinguished triangle $\tri{x'}{x}{x''}$ in $\sT$ with $x',x'' \in \sX$ then $x\in \sX$. The object $x$ will be called the \emph{middle term} of the extension. We denote by $\extn{\sX}$ the smallest extension-closed subcategory of $\sT$ containing $\sX$.

Let $\sC$ be any category and $\sA$ be a subcategory. 
A morphism $f\colon a \to c$ is called a \emph{right $\sA$-approximation} of $c$ if the induced map $\Hom_{\sC}(a',f)\colon \Hom_{\sC}(a',a) \to \Hom_{\sC}(a',c)$ is surjective for each object $a'$ of $\sA$. 
In the case that any object of $\sC$ admits a right $\sA$-approximation we say that $\sA$ is a \emph{contravariantly finite subcategory of $\sC$}. There are dual notions of \emph{left $\sA$-approximation} and \emph{covariantly finite}. If $\sA$ is both contra- and covariantly finite, $\sA$ is called \emph{functorially finite}.
Right (resp. left) $\sA$-approximations are often called  $\sA$-precovers (resp. $\sA$-preenvelopes). 

These concepts are linked by the following proposition.

\begin{proposition}[{\cite[Proposition 2.3]{IY}}]
\label{prop:Iyama-Yoshino}
Let $\kk$ be an algebraically closed field and $\sT$ be a $\kk$-linear, Krull-Schmidt and Hom-finite triangulated category. The following conditions are equivalent:
\begin{compactenum}
\item $(\sX,\sY)$ is a torsion pair;
\item $\sX$ is an extension-closed contravariantly finite subcategory of $\sT$ and $\sY = \sX^{\perp}$;
\item $\sY$ is an extension-closed covariantly finite subcategory of $\sT$ and $\sX = {}^{\perp}\sY$.
\end{compactenum}
\end{proposition}

%=============================================================================
% SECTION
\section{Triangulated categories generated by $w$-spherical objects} \label{sec:spherical}
%=============================================================================

Let $\kk$ be an algebraically closed field. Let $\sT$ be a $\kk$-linear triangulated category and $w \in \bZ$.  An object $s$ of $\sT$ is called \emph{$w$-spherical} \cite{Seidel-Thomas} if it satisfies the following axioms:
\begin{compactitem}
\item[(S1)] \label{spherelike} it is a \emph{$w$-spherelike} object \cite{HKP}, i.e.
\[
\Hom_{\sT}(s, \Sigma^i s) = 
\left\{
\begin{array}{ll}
\kk          & \text{if } i=0,w; \\
0            & \text{otherwise, and}
\end{array}
\right.
\]
\item[(S2)] \label{CY} it is a \emph{$w$-Calabi-Yau} object ($w$-CY, for short), i.e. there is a functorial isomorphism $\Hom_{\sT}(s,t) \simeq D\Hom_{\sT}(t,\Sigma^w s)$, where $t \in \sT$ and $D(-)\coloneqq \Hom_{\kk}(-,\kk)$ is the usual vector space duality.
\end{compactitem}

An object $s$ \emph{generates} $\sT$ if $\thick{\sT}{s} = \sT$, i.e. the smallest triangulated subcategory containing $s$ that is also closed under direct summands is $\sT$.

Let $\Tw$ be a $\kk$-linear triangulated category that is idempotent complete and generated by a $w$-spherical object. By \cite[Theorem 2.1]{KYZ}, $\Tw$ is unique up to triangle equivalence. Thus, we shall refer to $\Tw$ as \emph{the triangulated category generated by a $w$-spherical object}. The categories $\Tw$ satisfy many nice properties:
\begin{compactitem}
\item $\Tw$ is Hom-finite and Krull-Schmidt;
\item $\Tw$ has a Serre functor $\SSS$, i.e. a functor $\SSS\colon \Tw \to \Tw$ satisfying a functorial isomorphism $\Hom_{\Tw}(x,y) \simeq D\Hom_{\Tw}(y,\SSS x)$ for all $x,y\in \Tw$. Moreover, $\SSS = \Sigma \tau$, where $\tau$ is the Auslander--Reiten translate in $\Tw$.
\item $\Tw$ is  $w$-CY, i.e. $\SSS \simeq \Sigma^w$ and all objects $s,t \in \Tw$ satisfy (S2). 
\end{compactitem}

\subsection{The AR quiver of $\Tw$}

 For background on Auslander--Reiten (AR) theory we direct the reader to \cite{ASS}, \cite{ARS} and, in the triangulated setting \cite{Happel}.

The structure of the AR quiver of $\Tw$ was described in \cite{HJY} by using a model of $\Tw$ as a thick subcategory of the derived category of a certain differential graded algebra. 
The indecomposable objects of $\Tw$ and the form of the AR quiver of $\Tw$  was determined for $w \geq 2$ in \cite[Theorem 8.13]{Jorgensen-AR-theory} and for general $w$ in \cite[Section 3.3]{FY}. We summarise this below using the notation from \cite{HJY}.

\begin{proposition}
The indecomposable objects of $\Tw$ are precisely the (co)suspensions of a family of objects $X_r$ for $r\geq 0$.
If $w\neq 1$ then the AR quiver of $\Tw$ consists of $|w-1|$ copies of $\bZ A_{\infty}$. If $w=1$ then the AR quiver consists of a homogeneous tube and all its (co)suspensions. See Figure~\ref{fig:AR-quiver1}.
\end{proposition}

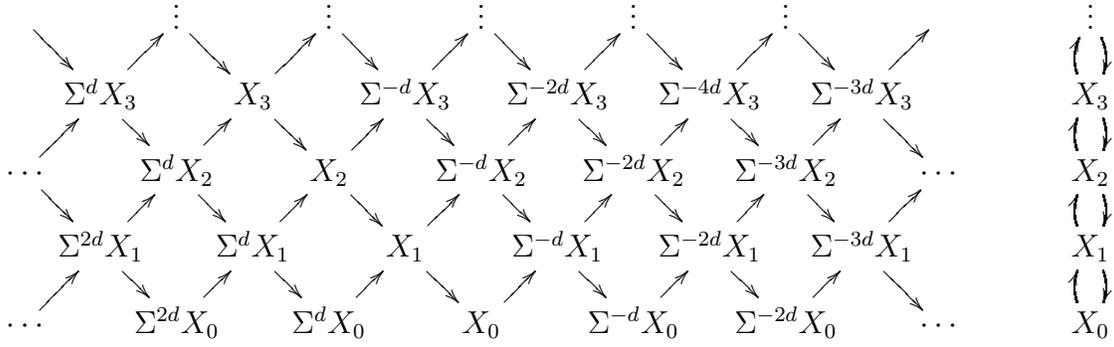
\begin{figure}
\centerline{
\xymatrix@!=0.4pc{
\ar[dr]    &                      & \vdots \ar[dr]           &                   & \vdots \ar[dr]                  &          & \vdots \ar[dr] &                      & \vdots \ar[dr]       &                        & \vdots \ar[dr]         &                      &  
&& \vdots \ardn \\
                        & \Sigma^d X_3 \arr &                        & X_3 \arr &                      & \Sigma^{-d} X_3 \arr &                 & \Sigma^{-2d} X_3 \arr &                        & \Sigma^{-4d} X_3 \arr  &                         & \Sigma^{-3d} X_3 \arr &    
&& X_3 \ardn \arup               \\
\cdots \arr        &                      & \Sigma^d X_2  \arr         &                   & X_2  \arr            &          & \Sigma^{-d} X_2 \arr &                      & \Sigma^{-2d} X_2  \arr     &                        & \Sigma^{-3d} X_2  \arr    &                      & \cdots 
&&  X_2 \arup \ardn  \\
                        & \Sigma^{2d} X_1 \arr &                        & \Sigma^d X_1 \arr &                      & X_1 \arr &                 & \Sigma^{-d} X_1 \arr &                        & \Sigma^{-2d} X_1 \arr  &                         & \Sigma^{-3d} X_1 \arr &    
&& X_1 \arup \ardn              \\
\cdots \ar[ur] &                      & \Sigma^{2d} X_0 \ar[ur] &                   & \Sigma^d X_0 \ar[ur] &          & X_0 \ar[ur]     &                     & \Sigma^{-d} X_0 \ar[ur] &                       & \Sigma^{-2d} X_0 \ar[ur] &                      & \cdots 
&& X_0 \arup
}}
\caption{\label{fig:AR-quiver1} {\it Left:} A component of type $\bZ \Ainf$ of the AR quiver in the case $w\neq 1$. Here $d=w-1$. The other components are obtained by applying $\Sigma^i$ to this one, for $i=1,\ldots, d-1$. 
{\it Right:} A component of the AR quiver in the case $w=1$. It is a homogeneous tube; all other components are obtained by applying $\Sigma^i$ to it, for $i\in \bZ$.}
\end{figure}

\subsection{Hom-hammocks in $\Tw$ for $w\neq 1$}

The notion of a Hom-hammock, introduced in \cite{Brenner}, is well known in the setting of Auslander--Reiten theory.
To describe the Hom-hammocks in $\Tw$ conveniently, we need to introduce some notation regarding rays and corays, which is borrowed from \cite{BPP}.

Consider the object $\Sigma^i X_j$ and make the following definitions (see Figure~\ref{fig:rayscorays} for an illustration):
\begin{align*}
\rayfrom{\Sigma^i X_j}      & \coloneqq \{\Sigma^{i-nd} X_{j+n} \mid n\geq 0\}, & \corayfrom{\Sigma^i X_j}    & \coloneqq \{\Sigma^i X_k \mid 0 \leq k \leq j\}; \\
\rayto{\Sigma^i X_j}        & \coloneqq \{\Sigma^{i+nd} X_{j-n} \mid 0 \leq n \leq j\}, & \corayto{\Sigma^i X_j}      & \coloneqq \{\Sigma^i X_k \mid k \geq j\}.  
\end{align*}

Given a set $\sS \subseteq \ind{\Tw}$ we make the obvious definitions of rays and corays determined by $\sS$, for example, 
$
\rayfrom{\sS} \coloneqq \bigcup_{s \in \sS}\rayfrom{s}.
$

For an object $a \in \ind{\Tw}$ define $L(a) \in \rayto{a}$ to be the unique object lying on the mouth of the component. Analogously, define $R(a) \in \corayfrom{a}$ to be the unique object lying on the mouth. Thus, if $a$ itself lies on the mouth, then $a=L(a) = R(a)$.

Given two indecomposable objects $a,b \in \ind{\Tw}$ that lie on the same ray or coray in the AR quiver of $\Tw$, then the finite set consisting of these two objects and all indecomposables lying between them on the (co)ray is denoted by $\overline{ab}$.
In an abuse of notation,  we identify $\rayfrom{a}\cap\corayto{b}$ with its indecomposable additive generator.

Following the usage prevalent in algebraic geometry, for objects $a,b \in \Tw$ we set $\hom_{\Tw}(a,b) \coloneqq \dim_{\kk} \Hom_{\Tw}(a,b)$. 
For $a \in \ind{\Tw}$, define the \emph{forward Hom-hammock} and the \emph{backward Hom-hammock} of $a$ as, respectively, 
\[
\homfrom{a} \coloneqq \rayfrom{\overline{aR(a)}} 
\ \text{and} \
\homto{a} \coloneqq \corayto{\overline{L(a)a}}.
\]

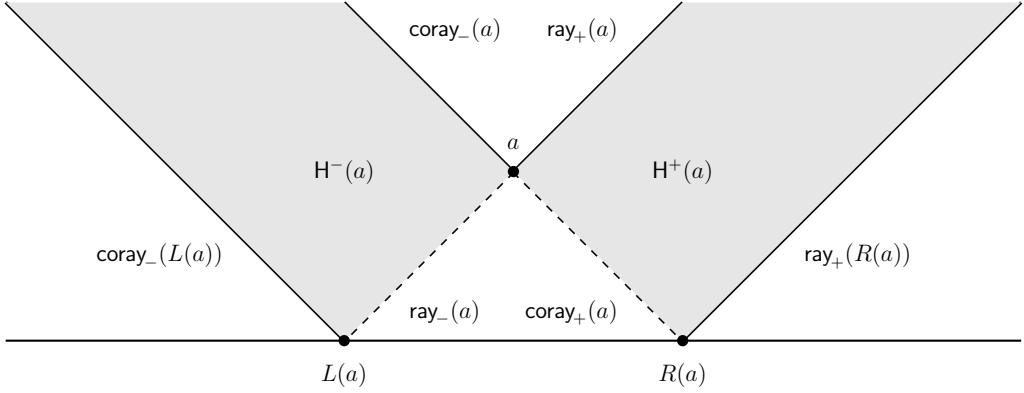
\begin{figure}
\begin{center}
\begin{tikzpicture}[thick,scale=0.75, every node/.style={scale=0.75}]

%coray-L(x)
\node (1) at (-5,1.5) {};
\node[left] at (1) {$\corayto{L(a)}$};

%coray- (x)
\node (2) at (-2,5.5) {};
\node[right] at (2) {$\corayto{a}$};

%ray+(x)
\node (3) at (2,5.5) {};
\node[left] at (3) {$\rayfrom{a}$};

%ray+R(x)
\node (4) at (5,1.5) {};
\node[right] at (4) {$\rayfrom{R(a)}$};

%coray+x
\node (5) at (2,0.5) {};
\node[left] at (5) {$\corayfrom{a}$};

%ray-(x)
\node (6) at (-2,0.5) {};
\node[right] at (6) {$\rayto{a}$};

%lines:
\draw[very thick] (0,3) -- (3,6);
\draw[very thick] (3,0) -- (9,6);
\draw[very thick] (-9,6) -- (-3,0);
\draw[very thick] (-3,6) -- (0,3);
\draw[very thick,dashed] (0,3) -- (3,0);
\draw[very thick,dashed] (0,3) -- (-3,0);

%shaded areas
\fill[fill=gray!20] (0,3) -- (-3,0) -- (-9,6) -- (-3,6) -- cycle;
\fill[fill=gray!20] (0,3) -- (3,0) -- (9,6) -- (3,6) -- cycle;

\draw[thick] (-9,0) -- (9,0); 

%x:
\fill (0,3) circle (1mm);
\node (x) at (0,3.25) {};
\node[above] at (x) {$a$};

%L(x):
\fill (-3,0) circle (1mm);
\node (L) at (-3,-0.25) {};
\node[below] at (L) {$L(a)$};

%R(x):
\fill (3,0) circle (1mm);
\node (R) at (3,-0.25) {};
\node[below] at (R) {$R(a)$};

%H-(x)
\node (-) at (-3,3) {};
\node at (-) {$\homto{a}$};

%H+(x)
\node (+) at (3,3) {};
\node at (+) {$\homfrom{a}$};

\end{tikzpicture}
\end{center}
\caption{The forward and backward Hom-hammocks of $a$.} \label{fig:rayscorays}
\end{figure}

\begin{proposition}[{\cite[Propositions 3.2 and 3.3]{HJY}}] \label{prop:hom-hammocks}
Let $a,b \in \ind{\Tw}$ for $w\neq 1$.
\begin{compactenum}[(i)]
\item If $w\neq 0$, then
\[
\hom_{\Tw}(a,b) = 
\left\{
\begin{array}{ll}
1 & \text{if } b \in \homfrom{a} \cup \homto{\SSS a}, \\
0 & \text{otherwise.}
\end{array}
\right.
\]
\item If $w= 0$, then
\[
\hom_{\Tw}(a,b) = 
\left\{
\begin{array}{ll}
1 & \text{if } b \in \homfrom{a} \cup \homto{\SSS a}\setminus \{a\}, \\
2 & \text{if } b=a, \\
0 & \text{otherwise.}
\end{array}
\right.
\]

\end{compactenum}
\end{proposition}

\subsection{Factorisation properties}

Later, it will be important to know how morphisms between indecomposable objects of $\Tw$ factor. This is dealt with in the following proposition, which generalises the statements of \cite[Propositions 2.1 and 2.2]{HJ}.

\begin{proposition} \label{prop:factoring}
Suppose $a,b$ and $c$ are in $\ind{\Tw}$ for $w \in \bZ\setminus\{0,1\}$.
\begin{compactenum}[(i)]
\item \label{item:fac-same-comp} Suppose that $c\in \homfrom{a} \cap \homfrom{b}$ and $b\in \homfrom{a}$. If $g\colon a \to b$ is a nonzero morphism, then each morphism $f\colon a \to c$ factors through $g$ as $a \rightlabel{g} b \too c$.
\item \label{item:fac-different-comp} Suppose that $c \in \homto{\SSS a} \cap \homto{\SSS b}$ and $b \in \homfrom{a}$. If $g\colon a \to b$ is a nonzero morphism, then each morphism $f\colon a \to c$ factors through $g$ as $a \rightlabel{g} b \too c$.
\end{compactenum}
Dually, 
\begin{compactenum}[($i^\prime$)]
\item  Suppose that $c\in \homto{a} \cap \homto{b}$ and $b\in \homto{a}$. If $g\colon b \to a$ is a nonzero morphism, then each morphism $f\colon c \to a$ factors through $g$ as $c \too b \rightlabel{g} a$.
\item  Suppose that $c \in \homfrom{\SSS^{-1} a} \cap \homfrom{\SSS^{-1} b}$ and $b \in \homto{a}$. If $g\colon b \to a$ is a nonzero morphism, then each morphism $f\colon c \to a$ factors through $g$ as $c \too b \rightlabel{g} a$.
\end{compactenum}
\end{proposition}

\begin{proof}
For $w \geq 2$ these statements are \cite[Propositions 2.1 and 2.2]{HJ} and their duals. The proofs in \cite{HJ} use only one-dimensionality of the Hom-spaces and Serre duality. Thus, the same arguments apply to $\Tw$ when $w \leq -1$.
\end{proof}

When $w=0$, each $a\in \ind{\sT_0}$ has a two-dimensional endomorphism space and we tweak the result for this case.

\begin{proposition} \label{prop:factoring-T0}
Suppose $a$, $b$ and $c$ are indecomposable objects in $\sT_0$.
\begin{compactenum}[(i)]
\item If $a$, $b$ and $c$ are pairwise non-isomorphic, then the statements in Proposition~\ref{prop:factoring} hold without modification.
\item Suppose $a \neq b$ and $b \in \homfrom{a} \cup \homto{a}$. Then for any non-isomorphism $f \in \Hom_{\sT_0}(a,a)$ there exist nonzero maps $g \colon a \to b$ and $h \colon b \to a$ such that $f = hg$.
\end{compactenum}
\end{proposition}

\begin{proof}
For statement (i) use one-dimensionality of the Hom-spaces and \cite{HJ} as above.

Let $f \colon a \to a$ be a nonzero non-isomorphism and consider the AR triangle $\tri{\tau a}{e_1 \oplus e_2}{a}$.
Let $b \in \homfrom{a} \cup \homto{a} \setminus \{a,e_1,e_2\}$. By the almost split property of the second map in the AR triangle, there is a commutative diagram:
\[
\xymatrix{
                   &                                                                                         & a \ar[d]^-{f} \ar[dl]_-{\exists\, \colmat{k_1}{k_2}} & \\
\tau a \ar[r] & e_1 \oplus e_2 \ar[r]_-{\rowmat{\epsilon_1}{\epsilon_2}} & a \ar[r]                                                    & \Sigma \tau a .
}
\]
By part (i) the maps $k_i \colon a \to e_i$ factor through $b$ as $a \rightlabel{g} b \rightlabel{h_i} e_i$. Note that the same map $g \colon a \to b$ can be chosen for each factorisation by the one-dimensionality of the Hom-spaces. Thus,
\[
f = \rowmat{\epsilon_1}{\epsilon_2} \colmat{k_1}{k_2} = \rowmat{\epsilon_1}{\epsilon_2} \colmat{h_1 g}{h_2 g} = \epsilon_1 h_1 g + \epsilon_2 h_2 g = (\epsilon_1 h_1 + \epsilon_2 h_2) g \eqqcolon hg,
\]
giving the required factorisation. If $b \in \{e_1,e_2\}$, then we argue dually using the other AR triangle $\tri{a}{e'_1 \oplus e'_2}{\tau^{-1} a}$.
\end{proof}

\begin{remark} \label{rem:basis-T0}
Let $a,b\in \ind{\sT_0}$, $a \neq b$ and $b \in \homfrom{a}$. Let $g \in \Hom_{\sT_0}(a,b)$ and $h \in \Hom_{\sT_0}(b,a)$ be nonzero maps. Then $\{\id_a,hg\}$ forms a basis of $\Hom_{\sT_0}(a,a)$.
\end{remark}

%==============================================================================
%  SECTION 
\section{Extensions in $\Tw$ with indecomposable outer terms for $w \neq 1$}\label{sec:extensions}
%==============================================================================

In this section, we describe how to compute the middle terms of extensions in $\Tw$ for which the outer terms are indecomposable.

\subsection{A necessary condition}

In this subsection we assume only that $\sT$ is a Krull-Schmidt triangulated category. 
Given a triangle $\tri{a}{e}{b}$ we give necessary conditions that the object $e$ must satisfy with respect to $a$ and $b$.
 The material is well-known to experts, but we give brief proofs for the convenience of the reader.

In a distinguished triangle $\trilabels{x}{y}{z}{f}{g}{}$, the object $z$ is called the \emph{cone} of $f$ and written $\cone{f}$, and the object $x$ is called the \emph{cocone} of $g$ and written $\cocone{g}$. 
The following lemma is straightforward.

\begin{lemma} \label{lem:cone}
Let $\trilabels{a}{e_1}{b}{f}{g}{h}$ be a distinguished triangle in a Krull-Schmidt triangulated category $\sT$. We have the following isomorphism of triangles:
\[
\xymatrix@!R=10pt{
a \ar[r]^-{\colmat{f}{0}} \ar@{-}[d]^-{\simeq} & e_1 \oplus e_2 \ar[r] \ar@{-}[d]^-{\simeq}  & \cone{\colmat{f}{0}} \ar[r] \ar@{-}[d]^-{\simeq} & \Sigma a \, , \ar@{-}[d]^-{\simeq} \\
a \ar[r]_-{\colmat{f}{0}}                      & e_1 \oplus e_2 \ar[r]_-{\sqmat{g}{0}{0}{1}} & b \oplus e_2 \ar[r]_-{\rowmat{h}{0}}             & \Sigma a
} 
\]
Dually when taking the cocone of a map of the form $\rowmat{g}{0}$.
\end{lemma}

An analogue of the following lemma is contained in the proof of \cite[Proposition~8.3]{BMRRT}.

\begin{lemma} \label{lem:homext-vanishing}
Let $\sT$ be a Krull-Schmidt triangulated category and suppose that
\[
\trilabels{a}{\bigoplus_{i=1}^n e_i}{b}{\colvec{f_1}{\vdots}{f_n}}{\rowvec{g_1}{\cdots}{g_n}}{}
\]
is a non-split distinguished triangle for some $a,b\in \ind{\sT}$. Then:
\begin{compactenum}
\item The $f_i$ and $g_i$ are each nonzero. 
\end{compactenum}
If additionally,  $\hom(x,y) \leq 1$ for each $x,y\in \ind{\sT}$ and $\Ext^1(x,x)=0$ for all $x\in\ind{\sT}$, then we also have
\begin{compactenum}
\item[(2)] $\Ext^1(e_i,a)=0$ and $\Ext^1(b,e_i)=0$. 
\item[(3)] The multiplicity of each indecomposable summand of the middle term is at most one.
\end{compactenum}
\end{lemma}

\begin{proof}
Statement $(1)$ follows immediately from Lemma~\ref{lem:cone}. 
For $(2)$, apply the functors $\Hom(-,a)$ and $\Hom(b,-)$ to the distinguished triangle and use one-dimensionality of the Hom-spaces and  the vanishing self-extension property.

Without loss of generality, to show $(3)$ it is enough to show that a non-split triangle of the form 
$\trilabel{a}{e^m}{b}{f}$
with $b \in \ind{\sT}$,
is forced to satisfy $m \leq 1$. Suppose for a contradiction that $m > 1$. Write $0 \neq f \coloneqq \rowvec{f_1}{\cdots}{f_m}^t$. Since $\hom(a,e) = 1$, we have $f_i = - \lambda_i f_1$ with $\lambda_i \in \kk$ for $2 \leq i \leq m$. Let $D$ be the matrix with $1$s along the leading diagonal and $\lambda_i$ for $2 \leq i \leq m$ down the first column. Then $D$ defines an isomorphism, making the first square in the following diagram of triangles commute:
\[
\xymatrix{
a \ar[r]^-{f} \ar@{=}[d] & e^m \ar[r] \ar[d]^-{D}              & b \ar[r] \ar[d]^-{\simeq} & \Sigma a \ar@{=}[d] \\
a \ar[r]_-{f'}           & e^m \ar[r] \ar[r]                   & b' \oplus e^{m-1} \ar[r]   & \Sigma a,}
\]
where $f' \coloneqq \rowvecc{f_1}{0}{\cdots}{0}^t$. Lemma~\ref{lem:cone} tells us that $\cone{f'} = b' \oplus e^{m-1}$, where $b' = \cone{f_1}$, and the Five Lemma for triangles gives $b \simeq b' \oplus e^{m-1}$. If $b' \neq 0$, then we get the desired contradiction to the indecomposability of $b$. If $b' = 0$, then $f_1$ is an isomorphism, meaning that we started with a split triangle, also a contradiction.
\end{proof}

\subsection{Ext-hammocks}
As with Hom-spaces, we write $\ext^1_{\Tw}(b,a) \coloneqq \dim_{\kk} \Ext^1_{\Tw}(b,a)$ for $a,b\in \ind{\Tw}$.
The Ext-hammocks for $a$ can be obtained by combining Proposition~\ref{prop:hom-hammocks} with Serre duality. The \emph{forward} and \emph{backward Ext-hammocks} of $a$ are, respectively,
\[
\extfrom{a} \coloneqq \homfrom{\tau^{-1}a} 
\ \text{and} \
\extto{a} \coloneqq \homto{\Sigma a}.
\]

\begin{proposition} \label{prop:ext-hammocks}
Suppose that $w \in \bZ\setminus \{1\}$ and $a, b \in \ind{\Tw}$.
\begin{compactenum}[(i)]
\item If $w\neq 0$, then
\[
\ext^1_{\Tw}(b,a) = 
\left\{
\begin{array}{ll}
1 & \text{if } b \in \extfrom{a} \cup \extto{a}, \\
0 & \text{otherwise.}
\end{array}
\right.
\]
\item If $w=0$, then
\[
\ext^1_{\sT_0}(b,a) = 
\left\{
\begin{array}{ll}
1 & \text{if } b \in \extfrom{a} \cup \extto{a} \setminus \{\Sigma a \}, \\
2 & \text{if } b= \Sigma a, \\
0 & \text{otherwise.}
\end{array}
\right.
\]
\end{compactenum}
\end{proposition}

Consider the object $X_r \in \ind{\Tw}$ for $r\geq 0$. The Ext-hammocks of $X_r$ are given by
\[
\extfrom{X_r} = \bigcup_{i=0}^{r} \rayfrom{\Sigma^{-d} X_{r-i}} 
\quad \text{and} \quad
\extto{X_r}  = \bigcup_{i=0}^r \corayto{\Sigma^{id+1} X_{r-i}}.
\]
These are indicated graphically in Figure~\ref{fig:ext}.

\subsection{Cohomology of the middle terms}
In this section we compute the cohomology of the middle terms of extensions in $\Tw$ for $w\neq 1$ whose outer terms are indecomposable. Since the action of $\Sigma$ and $\tau$ is transitive on the AR quiver of $\Tw$, without loss of generality we may restrict our attention to the objects $X_r$ for $r\geq 0$. 

We first deal with the Ext-hammock $\extfrom{X_r}$. Note that the non-trivial extensions occurring in this Ext-hammock have the form
\begin{equation} \label{eq:ray-triangle}
\trilabels{X_r}{E}{\Sigma^{-sd} X_{r+s-i}}{}{}{f} 
\quad \text{for }
s \geq 1, \text{ and } 1\leq i \leq r+1.
\end{equation}

\begin{lemma}\label{lem:ray-cohomology}
Let $w \in \bZ\setminus\{0,1\}$. Consider a triangle of the form \eqref{eq:ray-triangle} above. Then:
\[
H^n (E) = 
\left\{
\begin{array}{ll}
\kk   & \text{for } n=d,2d,\ldots,sd, \\
\kk^2 & \text{for } n=0,-d,\ldots, -(r-i)d, \\
\kk   & \text{for } n=-(r-i+1)d,-(r-i+2)d,\ldots,-rd, \\
0     & \text{otherwise,}
\end{array}
\right.
\]
where when $i=r+1$, we take the second condition to be empty.
\end{lemma}

\begin{proof}
Suppose $|d|>1$. Applying the functor $H^n(-)$ to the distinguished triangle~\eqref{eq:ray-triangle} and using the fact that, by the proofs of \cite[Propositions 3.2--3.4]{HJY},
\[
H^n(X_t) =
\left\{
\begin{array}{ll}
\kk & \text{if } n=0,-d,-2d,\ldots,-td, \\
0   & \text{otherwise.}
\end{array}
\right.
\]
we read off the cohomology from the following (short) exact sequences when $0\leq i \leq r$:
\begin{align}
&H^{(s-j)d}(E)   \isoto  H^{-jd}(X_{r+s-i}) &\text{for } 0\leq j < s; \notag \\
H^{-jd}(X_r) \intoo &H^{-jd}(E) \ontoo H^{-(j+s)d}(X_{r+s-i}) &\text{for } 0\leq j \leq r-i;  \label{eq:ray-ses} \\
H^{-jd}(X_r) \isoto &H^{-jd}(E) & \text{for } r-i+1 \leq j \leq r. \notag
\end{align}
When $i=r+1$, the sequence \eqref{eq:ray-ses} degenerates into the isomorphism on the third line.

Now suppose $d=1$. Then the short exact sequences \eqref{eq:ray-ses} above are connected into the long exact sequence
\begin{align*}
H^{-(r-i)}(X_r) \intoo &H^{-(r-i)}(E) \too H^{-(r+s-i)}(X_{r+s-i}) \rightlabel{H^{-(r-i)}(f)} H^{-(r-i-1)}(X_r) \too \cdots \\
\cdots \too 
&H^{-s-1}(X_{r+s-i}) \rightlabel{H^{-1}(f)} H^0(X_r) \too H^0(E) \ontoo H^{-s}(X_{r+s-i}).
\end{align*}
For $i=r$ and $i=r+1$ there is nothing to prove: in the first case, there is already only one short exact sequence, and in the second, \eqref{eq:ray-ses} degenerates into an isomorphism. Thus, we need only consider the cases $1\leq i < r$.
There are nonzero maps $g\colon \Sigma^{-s} X_{r+s-i} \to \Sigma^{r+2-i} X_{i-1}$ and $h\colon \Sigma^{r+2-i} X_{i-1} \to \Sigma X_r$ by Proposition~\ref{prop:hom-hammocks}. Thus the map $f$ in triangle~\eqref{eq:ray-triangle} factors as $f = hg$ by 
Proposition~\ref{prop:factoring}\eqref{item:fac-different-comp}. 
It follows that $H(f) = H(h)H(g)$. Now for $0\leq j \leq r-i$ we have 
\[
H^{-j}(\Sigma^{r+2-i}X_{i-1}) = H^{r+2-i-j}(X_{i-1})=0
\]
because $r+2-i-j > 0$ and $H^n(X_t) = 0$ for any $n >0$ and $t\geq 0$ in the case $d>0$. Thus it follows that the connecting maps $H^{-j}(f)$ for $0\leq j \leq r-i$ in the long exact sequence above are zero, which thus decomposes back into the short exact sequences \eqref{eq:ray-ses} allowing us to again read off the cohomology of $E$.
\end{proof}

We now deal with the Ext-hammock $\extto{X_r}$. Note that the non-trivial extensions occurring in this Ext-hammock have the form
\begin{equation} \label{eq:coray-triangle}
\trilabels{X_r}{E}{\Sigma^{1+id} X_{r+s-i}}{}{}{f} \quad \text{for } s\geq 0 \text{ and } 0\leq i \leq r.
\end{equation}

\begin{lemma} \label{lem:coray-cohomology}
Let $w \in \bZ \setminus \{0,1\}$. Consider a triangle of the form \eqref{eq:coray-triangle} above. Then:
\[
H^n (E) = 
\left\{
\begin{array}{ll}
\kk   & \text{for } n=0,d,\ldots,(i-1)d, \\
\kk   & \text{for } n=(r+i)d-1,(r+i+1)d-1,\ldots,(r+s+i)d-1,  \\
0     & \text{otherwise,}
\end{array}
\right.
\]
where when $i=0$ we assume the first condition to be empty.
\end{lemma}

\begin{proof}
Apply the functor $H^n(-)$ to the triangle \eqref{eq:coray-triangle} and note that the long exact cohomology sequence decomposes into exact sequences
\[
H^{-jd-1}(E)   \intoo H^{(i-j)d}(X_{r+s-i})  \rightlabel{H^{-jd-1}(f)}  H^{-jd}(X_r)  \ontoo  H^{-jd}(E).     
\]
Since $H^{-jd-1}(f)=0$ for $0 \leq j < i$ and $i+r \leq j \leq i+r+s$, we have $H^{-jd}(X_r) \simeq H^{-jd}(E)$ for $0 \leq j < i$ and $H^{-jd-1}(E) \simeq H^{(i-j)d}(X_{r+s-i})$ for $i+r \leq j \leq i+r+s$.

The map $f\colon \Sigma^{1+id} X_{r+s-i} \to \Sigma X_r$ factors as $\Sigma^{1+id} X_{r+s-i} \rightlabel{h} \Sigma^{1+id} X_{r-i} \rightlabel{g} \Sigma X_r$ by Proposition~\ref{prop:factoring}. The map $g$ is induced from an inclusion map of the underlying DG modules; see \cite[Section 2]{HJY} for precise details. As such $H^{-jd-1}(g)\colon H^{(i-j)d}(X_{r-i}) \too H^{-jd}(X_r)$ is nonzero and thus an isomorphism (by one-dimensionality) for $i \leq j \leq i+r$, and zero otherwise. 
Similarly the induced map $H^{(i-j)d}(h)\colon H^{jd-1}(X_{r+s-i}) \to H^{-jd-1}(X_{r-i})$ is an isomorphism for $i\leq j \leq i+r$, and zero otherwise.
Since $H(f) = H(g)H(h)$, it follows that $H^{-jd-1}(f)$ is an isomorphism for $i \leq j \leq i+r$. 
Now one can read off the cohomology of $E$ from the sequences above. 
\end{proof}

\begin{lemma} \label{lem:w=0twoextensions}
In the case $w=0$, the statements of Lemmas~\ref{lem:ray-cohomology} and \ref{lem:coray-cohomology} also hold, with the modification that there are, up to equivalence, two extensions,
\[
\tri{X_r}{E}{\Sigma X_r},
\]
one whose middle term has cohomology as in Lemma~\ref{lem:ray-cohomology}, and one whose middle term has trivial cohomology.
\end{lemma}

\begin{proof}
In the case that $w=0$, $d=-1$ and so the AR quiver of $\sT_0$ consists of only one $\bZ A_{\infty}$ component. However, the extensions with indecomposable outer terms are formulated exactly as in Lemmas~\ref{lem:ray-cohomology} and \ref{lem:coray-cohomology}. The only difference occurs because the Ext-hammocks $\extfrom{X_r}$ and $\extto{X_r}$ have non-trivial intersection $\extfrom{X_r} \cap \extto{X_r} = \{\Sigma X_r\}$.  The two-dimensional Ext-space,  $\Ext^1(\Sigma X_r,X_r) = \Hom(\Sigma X_r,\Sigma X_r)$ (see Proposition~\ref{prop:hom-hammocks}), has a basis $\{\id_a,f\}$ where $f$ can be chosen to be a non-isomorphism factoring through any indecomposable object in $\homfrom{\Sigma X_r}$; see Remark~\ref{rem:basis-T0}. The corresponding extensions are:
\[
\trilabels{X_r}{E}{\Sigma X_r}{}{}{f} \quad \text{and} \quad
\trilabels{X_r}{0}{\Sigma X_r}{}{}{\id}.
\]
The first triangle is (equivalent to) the AR triangle, thus its cohomology is known. However, one can also argue exactly as in the case $d=1$ in the proof of Lemma~\ref{lem:ray-cohomology}. It is clear that the middle term of the second triangle has trivial cohomology.
\end{proof}

\subsection{Graphical calculus}\label{sec:graphical}

The main technical result of this section is the following computation of the middle terms of extensions whose outer terms are indecomposable. It is analogous to the graphical calculus in \cite[Corollary 8.5]{BMRRT}. The strategy of our proof is inspired by \cite[Section 8]{Jorgensen-AR-theory}.

\begin{theorem}
\label{thm:graphical-calculus}
Let $a,b \in \ind{\Tw}$ for $w\neq 0,1$. Suppose $\Ext^1_{\sT}(b,a)\neq 0$. 
Let $\tri{a}{e}{b}$ be the unique non-split extension of $b$ by $a$. Then $e$ decomposes as $e=e_1 \oplus e_2$, where $e_i$ is either indecomposable or zero for $i=1,2$. Moreover, the $e_i$ can be computed by the following graphical calculus:
\begin{compactenum}[(i)]
\item \label{item:same-component}
If $b \in \extfrom{a}$ then
$e_1  = \rayfrom{a} \cap \corayto{b}$ and $e_2  = \corayfrom{a} \cap \rayto{b}$.
\item \label{item:different-component}
If $b\in \extto{a}$ then
$e_1  = \corayto{L(\sS a)} \cap \rayto{b}$ and $e_2  = \corayfrom{a} \cap \rayfrom{R(\sS^{-1}b)}$.
\end{compactenum}
If any of the intersections in parts \eqref{item:same-component} and \eqref{item:different-component} are empty, then we interpret the corresponding object as being the zero object. See Figure~\ref{fig:ext} for an illustration.
\end{theorem}

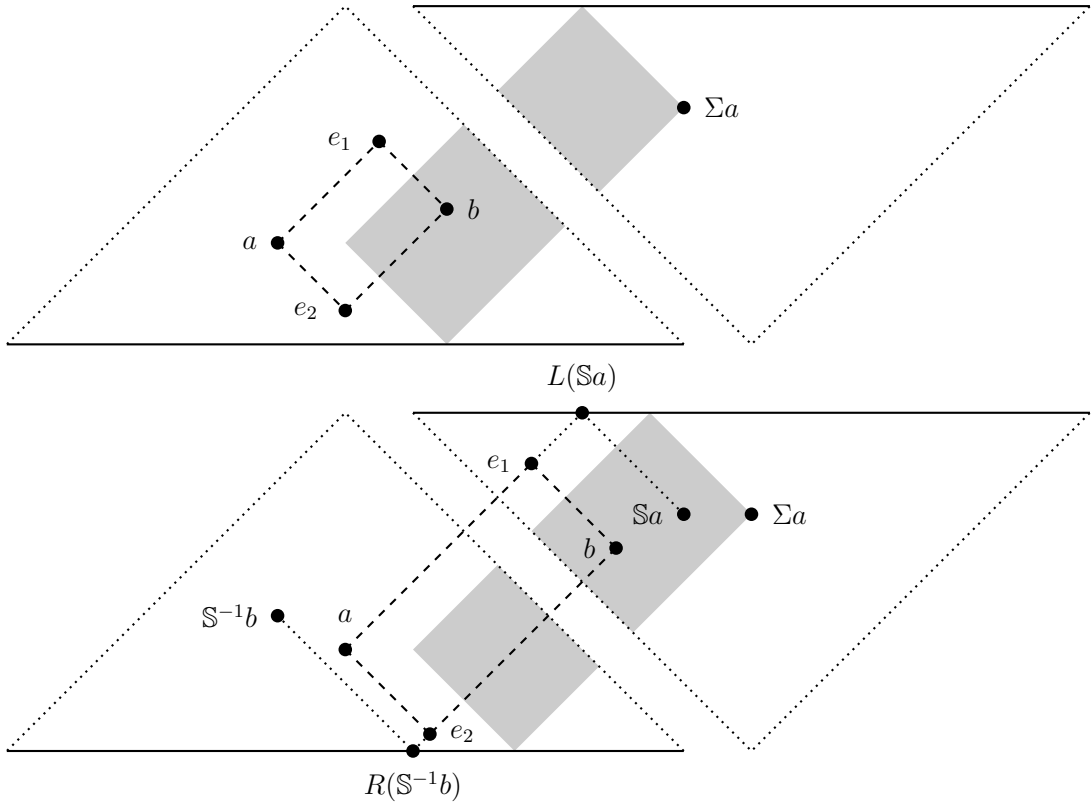
\begin{figure}
\begin{center}
\begin{tikzpicture}[thick,scale=0.9, every node/.style={scale=0.9}]

\fill[gray!40] (5,1.5) -- (6.5,0) -- (8.25,1.75) -- (6.75,3.25) -- cycle;
\fill[gray!40] (7.25, 3.75) -- (8.75,2.25) -- (10,3.5) -- (8.5,5) -- cycle;

%The components
\draw (0,0) -- (10,0); 
\draw[dotted] (10,0) -- (5,5) -- (0,0);
\draw[dotted] (6,5) -- (11,0) -- (16,5); 
\draw (16,5) -- (6,5);

\fill (4,1.5) circle (1mm); %x-r
\node (x-r) at (4,1.5) {};
\node [left] at (x-r.west) {$a$};

%Ray coming out of (x-r)
\draw[dashed] (4,1.5) -- (5.5,3); 
\draw[dashed] (4,1.5) -- (5,0.5); 

\fill (10,3.5) circle (1mm); %\Sigma x-r
\node (sigma) at (10,3.5) {}; 
\node [right] at (sigma.east) {$\Sigma a$};

\fill (6.5,2) circle (1mm); % b
\node (b) at (6.5,2) {};
\node [right] at (b.east) {$b$};

\draw[dashed] (5.5,3) -- (6.5,2); 
\draw[dashed] (5,0.5) -- (6.5,2); 

\fill (5.5,3) circle (1mm);
\node (e1) at (5.25,3) {};
\node[left] at (e1) {$e_1$};

\fill (5,0.5) circle (1mm);
\node (e2) at (4.75,0.5) {};
\node[left] at (e2) {$e_2$};

\end{tikzpicture}
\end{center}
\begin{center}
\begin{tikzpicture}[thick,scale=0.9, every node/.style={scale=0.9}]

\fill[gray!40] (6,1.5) --(7.5,0) -- (8.75,1.25) -- (7.25,2.75) -- cycle;
\fill[gray!40] (7.75,3.25) -- (9.25,1.75) -- (11,3.5) -- (9.5,5) -- cycle;

%The components
\draw (0,0) -- (10,0);
\draw[dotted] (10,0) -- (5,5) -- (0,0);
\draw[dotted] (6,5) -- (11,0) -- (16,5);
\draw (16,5) -- (6,5);

\fill (5,1.5) circle (1mm); %x-r
\node (x-r) at (5,1.75) {};
\node [above] at (x-r) {$a$};

%Ray coming out of (x-r)
\draw[dashed] (5,1.5) -- (6.75,3.25) -- (7.25,3.75) -- (7.75,4.25); 

%Coray coming out of (x-r)
\draw[dashed] (5,1.5) -- (6.25,0.25);

\fill (11,3.5) circle (1mm); %\Sigma x-r
\fill (10,3.5) circle (1mm); %\SSS x-r
\node (sigma) at (11,3.5) {}; 
\node [right] at (sigma.east) {$\Sigma a$};
\node (serre) at (10,3.5) {};
\node [left] at (serre.west) {$\SSS a$};

\fill (8.5,5) circle (1mm); %L(\SSS x-r)$
\node (L) at (8.5,5) {};
\node [above] at (L.north) {$L(\SSS a)$};

%dashed line from \SSS x-r to L(\SSS x-r)
\draw [dotted] (10,3.5) -- (8.5,5);

%dashed line from x-r to L(\SSS x-r)
\draw [dotted] (8.5,5) -- (7.75,4.25); 

\fill (9,3) circle (1mm); % b
\node (b) at (9,3) {};
\node [left] at (b.west) {$b$};

\fill (4,2) circle (1mm); % \SSSinv b
\node (Sb) at (4,2) {};
\node [left] at (Sb.west) {$\SSS^{-1} b$};

\fill (6,0) circle (1mm); %R(\SSSinv b)$
\node (R) at (6,0) {};
\node [below] at (R.south) {$R(\SSS^{-1}b)$};

%dashed line from \SSS^{-1}b to R(\SSS^{-1}b)
\draw [dotted] (4,2) -- (6,0);

%dashed line from R(\SSS^{-1} b) to b
\draw [dotted] (6,0) -- (6.25,0.25); 

\draw[dashed] (9,3) -- (6.25,0.25); 
\draw[dashed] (7.75,4.25) -- (9,3); 

\fill (7.75,4.25) circle (1mm);
\node (e1) at (7.6,4.25) {};
\node[left] at (e1) {$e_1$};

\fill (6.25,0.25) circle (1mm);
\node (e2) at (6.4,0.25) {};
\node[right] at (e2) {$e_2$};

\end{tikzpicture}
\end{center}
\caption{Top: Middle terms of extensions whose outer terms lie in the same component. Bottom: Middle terms of extensions whose outer terms lie in different components. Top and bottom: Shaded regions are the Ext-hammocks $\extfrom{a}$ and $\extto{a}$.} \label{fig:ext}
\end{figure}

\begin{proof}
Without loss of generality, we may assume that $a= X_r$ for some $r \geq 0$. Firstly consider triangle~\eqref{eq:ray-triangle}:
\[
\tri{X_r}{E}{\Sigma^{-sd} X_{r+s-i}}
\]
Suppose $E=\bigoplus_{i=1}^n E_i^{m(i)}$ with $E_i \in \ind{\Tw}$ and $m(i) \geq 0$. Note that when $w \neq 0,1$, Lemma~\ref{lem:homext-vanishing}$(3)$ means that the $m(i) \leq 1$. Moreover, the only indecomposable objects satisfying the necessary conditions of Lemma~\ref{lem:homext-vanishing} are $\{X_{r-i}, \Sigma^{-sd} X_{r+s}\}$. Call these objects \emph{candidates}. Note that when $i=r+1$, the candidate is simply $\Sigma^{-sd}X_{r+s}$.
We now use Lemma~\ref{lem:ray-cohomology} to identify whether these two indecomposable summands appear in $E$ with multiplicity $0$ or $1$.
Write $E=E_1\oplus E_2$ with $E_1$ indecomposable.

Suppose that $d<-1$. In this case the cohomology of the indecomposable objects $X_t$ is concentrated in non-negative degrees. By Lemma~\ref{lem:ray-cohomology} the lowest degree in which $E$ has non-trivial cohomology is $sd$. Thus, $\Sigma^{-sd} X_{r+s}$ must be a direct summand of $E$.
Set $E_1 = \Sigma^{-sd} X_{r+s}$, and note that $E_1$ has one-dimensional cohomology in degrees $sd, (s-1)d, \ldots, -rd$. This leaves $E_2$ with one-dimensional cohomology in degrees $0, -d, \ldots,-(r-i)d$. The only candidate object with cohomology in these degrees is $X_{r-i}$, giving the unique non-split triangle as
\[
\tri{X_r}{\Sigma^{-sd} X_{r+s} \oplus X_{r-i}}{\Sigma^{-sd} X_{r+s-i}}.
\] 
Inspecting the AR quiver gives: 
\begin{align*}
\rayfrom{X_r} \cap \corayto{\Sigma^{-sd} X_{r+s-i}} & = \{\Sigma^{-sd} X_{r+s}\}, \\
\corayfrom{X_r} \cap \rayto{\Sigma^{-sd} X_{r+s-i}} & = \{ X_{r-i} \}.
\end{align*}
The case $d>0$ is analogous, taking into account that the indecomposables $X_t$ now have non-trivial cohomology only in non-positive degrees, giving statement (i). An analogous argument applied to the triangle \eqref{eq:coray-triangle} using Lemma~\ref{lem:coray-cohomology} gives (ii).
\end{proof}

\begin{proposition} \label{prop:graphical-calculus}
If $w=0$ and $a,b \in \ind{\sT_0}$ with $b \neq \Sigma a$, then the statement of Theorem~\ref{thm:graphical-calculus} also holds. If $b = \Sigma a$, then there are two non-split triangles 
\[
\trilabels{a}{e_1 \oplus e_2}{\Sigma a}{}{}{f} \ \text{and} \ \trilabels{a}{0}{\Sigma a}{}{}{\id},
\]
where $f$ is a non-isomorphism. The first is computed as in Theorem~\ref{thm:graphical-calculus}(i), the second corresponds to Theorem~\ref{thm:graphical-calculus}(ii).
\end{proposition}

\begin{proof}
Here we have $d=-1$ and we cannot apply Lemma~\ref{lem:homext-vanishing}$(2)$ because each indecomposable object has two-dimensional endomorphism spaces and, provided it does not lie on the boundary of the AR quiver, self-extensions. However, we can apply Lemma~\ref{lem:homext-vanishing}$(1)$ and brute force using Lemmas~\ref{lem:ray-cohomology} and \ref{lem:coray-cohomology}.

Consider the triangle \eqref{eq:ray-triangle} corresponding to the Ext-hammock  $\extfrom{X_r}$. Again write $E=E_1\oplus E_2$ with $E_1$ indecomposable. First observe that Lemma~\ref{lem:ray-cohomology} implies that one summand of $E$ is $\Sigma^s X_t$ for some $t\geq 0$. Lemma~\ref{lem:homext-vanishing}$(1)$ and Proposition~\ref{prop:hom-hammocks} mean that $\Sigma^s X_{r+s}$ is the only possibility, which we take to be $E_1$. This means that $E_2$ has cohomology in degrees $0,1,\ldots,r-i$, and so there is precisely one summand of $E_2$ that consists of an unsuspended $X_t$ for some $0 \leq t \leq r-i$. Note that if $t > r-i$ then $X_t$ has cohomology in too many degrees. Inspecting the AR quiver and using Proposition~\ref{prop:hom-hammocks} again, we see that if $t<r-i$, there is no map $X_t \to \Sigma^{-s}X_{r+s-i}$, and thus by Lemma~\ref{lem:homext-vanishing}$(1)$ such $X_t$ cannot be a summand of $E_2$. This leaves only $X_{r-i}$ itself, giving $E=\Sigma^{-s}X_{r+s} \oplus X_{r-i}$ again, as claimed.

The argument for the triangle \eqref{eq:coray-triangle} is analogous, however, one must deal with the case $\tri{X_r}{E}{\Sigma X_r}$ separately. As remarked in Lemma~\ref{lem:w=0twoextensions}, the two triangles are the standard triangle $\trilabels{X_r}{0}{\Sigma X_r}{}{}{\id}$ and the AR triangle. The AR triangle puts us in case (i) of the theorem, and the second triangle puts us in case (ii) of the theorem with empty intersections, and therefore zero middle term.
\end{proof}

\begin{remark}
In the case that $w = 0$ and $a$ is an indecomposable object not lying on the boundary of the AR quiver we have $\Ext^1_{\sT_0}(a,a) = \kk$. The middle term of the self-extension $\tri{a}{e_1 \oplus e_2}{a}$ is computed using Theorem~\ref{thm:graphical-calculus}(ii): one should regard the second occurrence of $a$ as being $\SSS a$ and lying in a `different AR component'.
\end{remark}

%========================================================================
% SECTION
\section{Extensions in $\Tw$ with decomposable outer terms for $w \neq 1$}\label{sec:non-indecouterterms}
%========================================================================

When computing the extension closure of a set of objects, it is useful to be able to reduce to computing only the middle terms of extensions whose outer terms are indecomposable. The aim of this section is to establish this for the categories $\Tw$ by proving the following theorem.

\begin{theorem} \label{thm:extensionsnonindecomposable}
Let $w\in \bZ \setminus \{1\}$. Let $\{a_i\}_{i=1}^n$ and $\{ b_j \}_{j=1}^m$ be sets of (not necessarily pairwise non-isomorphic) indecomposable objects of $\Tw$. Any extension of the form 
\[
\tri{\bigoplus_{i=1}^n a_i}{e}{\bigoplus_{j=1}^m b_j}\]
can be computed iteratively from extensions whose outer terms are indecomposable and built from $\{a_i\}_{i=1}^n$ and $\{b_j\}_{j=1}^m$. 
\end{theorem}

\subsection{Factorisation-free extensions}

We start with a short general subsection. Here $\sT$ will be a $\kk$-linear, Hom-finite, Krull-Schmidt triangulated category.

\begin{definition}
A collection of morphisms $\{h_i \colon b_i \rightarrow \Sigma a\}_{i=1}^m$ in $\sT$ will be called \textit{factorisation-free} if for each $i \neq j$ there is no map $\beta \colon b_i \rightarrow b_j$ such that $h_i = h_j \beta$.  
\end{definition}

\begin{lemma} \label{lem:reduction}
Let $a, b_1, b_2 \in \ind{\sT}$. 
Suppose $h_1 \colon b_1 \rightarrow \Sigma a$ and $h_2 \colon b_2 \rightarrow \Sigma a$ are non-zero maps such that there exists $\beta \colon b_1 \rightarrow b_2$ with $h_1 = h_2 \beta$. Then, there is an isomorphism of triangles:
\[
\xymatrix@!R=10pt{
a \ar[r] \ar@{=}[d]^-{\simeq} & e \ar[r] \ar@{-}[d]^-{\simeq}  & b_1 \oplus b_2 \ar[r]^{\rowmat{h_1}{h_2}} \ar@{-}[d]^-{\simeq} & \Sigma a \ar@{=}[d]^-{\simeq} \\
a \ar[r]                    & b_1 \oplus f \ar[r] & b_1 \oplus b_2 \ar[r]_-{\rowmat{0}{h_2}}             & \Sigma a,
}
\]
where $f$ is the cocone of $h_2$, i.e. $\trilabels{a}{f}{b_2}{}{}{h_2}$.
\end{lemma}

\begin{proof}
This follows from the Five Lemma for triangulated categories by considering the isomorphism $F \colon b_1 \oplus b_2 \rightarrow b_1 \oplus b_2$ given by $F = \sqmat{1}{0}{\beta}{1}$ and using Lemma~\ref{lem:cone}. 
\end{proof} 

\begin{corollary}\label{cor:m=1directsum}
Any extension $\tri{a}{e}{\bigoplus_{i=1}^m b_i}$ in $\sT$ is isomorphic to a direct sum of the triangles $\tri{a}{E}{B}$ and $\tri{0}{B'}{B'}$, in which $B$ is the sum of objects $b_i$ such that $\{h_i \colon b_i \rightarrow \Sigma a\}$ is factorisation-free and $B^\prime$ is the sum of the remaining summands. 
\end{corollary}

\begin{remark} \label{rem:multiplesummands}
In the case of $\Tw$, Corollary~\ref{cor:m=1directsum} implies that we need only consider extensions of the form $\tri{a}{e}{\bigoplus_{i=1}^m b_i}$ in which the $b_i$ are pairwise non-isomorphic: any morphism $b_i \to \Sigma a$ factors through an isomorphism $b_i \to b_i$.
\end{remark}

We now show that factorisation-freeness in $\Tw$ is essentially the same as Hom-orthogonality.

\begin{lemma} \label{lem:orthogonal}
Suppose $b_1, b_2 \in \extfrom{a} \cup \extto{a}$ are indecomposable objects such that $b_i \neq \Sigma a$ for $i=1,2$. The pair of morphisms $\{h_i \colon b_i \rightarrow a\}_{i=1}^2$ is factorisation-free if and only if $b_1$ and $b_2$ are Hom-orthogonal.  
\end{lemma}

\begin{proof}
One direction is clear. For the other, suppose $\Hom (b_1, b_2) \neq 0$. We claim that $\{h_1, h_2\}$ is not factorisation-free. 

\Case{$b_1 \in \homfrom{\tau^{-1} a}$}
In this case, either $b_2 \in \homfrom{b_1}$ or $b_2, \Sigma a \in \homto{\SSS b_1}$. Either way, we have that $h_1$ factors through $h_2$, by Propositions \ref{prop:factoring} and \ref{prop:factoring-T0} and by one-dimensionality of $\Hom (b_2, \Sigma a)$.  

\Case{$b_1 \in \homto{\Sigma a}$}
Since $\Hom (b_1, b_2) \neq 0$, we have two subcases.

\Subcase{$b_2 \in \homfrom{b_1}$} 
If $w \neq 0, 2$, then $b_2$ must lie in $\homto{\Sigma a}$ since $\Hom (b_2, \Sigma a) \neq 0$. Hence, by Proposition \ref{prop:factoring}, $h_1$ factors through $h_2$.
If $w = 0$ or $2$, then $b_2$ can lie either in $\homto{\Sigma a}$, in which case $h_1$ factors through $h_2$, or $\homfrom{\tau^{-1} a}$. In the latter case, note that $b_1, \Sigma a \in \homto{\SSS b_2}$, and so, by Propositions \ref{prop:factoring} and \ref{prop:factoring-T0}, $h_2$ factors through $h_1$. 

\Subcase{$b_2 \in \homto{\SSS b_1}$} 
If $w \neq 0, 2$, then $b_2 \not\in \homfrom{\tau^{-1} a} \cup \homto{\Sigma a}$, and so $\Hom (b_2, \Sigma a) = 0$, a contradiction. So, $w$ must be $0$ or $2$. But then we have $b_1, \Sigma a \in \homfrom{b_2}$, and so again, by Propositions \ref{prop:factoring} and \ref{prop:factoring-T0}, $h_2$ factors through $h_1$.  
\end{proof}

\subsection{Extensions for which the first term is indecomposable}

For this subsection we invoke the following setup.

\begin{setup} \label{setup:decomposable}
Let $w \in \bZ \setminus \{1\}$ and fix $a \in \ind{\Tw}$. Suppose $b_i \in \ind{\Tw}$ are such that $b_i \in \extfrom{a} \cup \extto{a}$ for $i=1,\ldots,m$. Consider the triangles
\[
\trilabels{a}{e_i^\prime \oplus e_i^{\prime \prime}}{b_i}{}{}{h_i}
\]
from Theorem~\ref{thm:graphical-calculus} and Proposition~\ref{prop:graphical-calculus}. We shall always assume that $e^{\prime \prime}_i \in \corayfrom{a}$.

We now compute the middle terms of triangles of the form 
\[
\trilabels{a}{e}{\bigoplus_{i=1}^m b_i}{}{}{\rowvec{h_1}{\cdots}{h_m}},
\]
where $\{h_i \colon b_i \rightarrow \Sigma a \}_{i=1}^m$ is factorisation-free, and each map $h_i \colon b_i \to \Sigma a$ is nonzero.
If $w = 0$, we further assume that $b_i \neq \Sigma a$. In particular, we always have $\ext(b_i,a) = 1$.

In light of Lemma~\ref{lem:orthogonal}, the assumption that $\{ h_i \colon b_i \to \Sigma a\}_{i=1}^m$ is factorisation-free boils down to assuming that the $b_i$ are pairwise Hom-orthogonal.
\end{setup}

To simplify our arguments we need the following technical definition.

\begin{definition}
Let $b \in \extfrom{a} \cup \extto{a}$. We define the \textit{extended ray and extended coray of $b$ with respect to $a$} as follows:
\begin{align*}
\exray{a}{b} & := \begin{cases}
\rayfrom{L(b)} \cup \corayto{\SSS L(b)} &\mbox{if } b \in \extfrom{a}, \\
\corayto{R(b)} \cup \rayfrom{\SSS^{-1} R(b)} &\mbox{if } b \in \extto{b},\mbox{ and } \end{cases} \\
\excoray{a}{b} & := \begin{cases}
\corayto{R(b)} \cup \rayfrom{\SSS^{-1} R(b)} &\mbox{if } b \in \extfrom{a}, \\
\rayfrom{L(b)} \cup \corayto{\SSS L(b)} &\mbox{if } b \in \extto{a}. \end{cases}
\end{align*}
Note that every element in $\extfrom{a} \cup \extfrom{a}$ lies in $\exray{a}{b}$ for some $b \in \overline{(\tau^{-1} a) R(\tau^{-1}a)}$. 
\end{definition}

We can define a total order on extended rays of elements in $\extfrom{a} \cup \extto{a}$ as follows. Given $x, y \in \extfrom{a} \cup \extto{a}$, we say that $\exray{a}{x} \leq \exray{a}{y}$ if and only if there is $k \geq 0$ for which one of the following conditions holds:
\begin{compactenum}
\item $x, y \in \extfrom{a}$ and $L(x) = \tau^k L(y)$;
\item $x, y \in \extto{a}$ and $R(x) = \tau^k R(y)$;
\item $x \in \extfrom{a}, y \in \extto{a}$ and $\SSS L(x) = \tau^k R(y)$, or
\item $x \in \extto{a}, y \in \extfrom{a}$ and $R(x) = \tau ^k \SSS L(y)$.
\end{compactenum} 

\begin{lemma}\label{lem:extensionsaisandeis}
Suppose  $b_1, b_2 \in \ind{\Tw}$ are Hom-orthogonal objects in $\extfrom{a} \cup \extto{a}$ such that $\exray{a}{b_1} < \exray{a}{b_2}$. Then
\begin{compactenum}
\item $\Ext^1 (b_2, e^\prime_1) = 0$ but $\Ext^1 (b_2, e^{\prime \prime}_1) \neq 0$;
\item $\Ext^1 (b_1, e^{\prime \prime}_2) = 0$ but $\Ext^1 (b_1,e^\prime_2) \neq 0$.
\end{compactenum}
\end{lemma}

\begin{proof}
This follows by a case analysis in examining the various positions $b_1$ and $b_2$ can have inside $\extfrom{a} \cup \extto{a}$ and then comparing the Ext-hammocks of $e_i^\prime$ and $e_i^{\prime \prime}$ with the positions of $b_i$, for $i = 1, 2$.
\end{proof} 

\begin{remark}\label{rmk:middletermeisandbis}
By Theorem \ref{thm:graphical-calculus} and Proposition \ref{prop:graphical-calculus}, the non-split triangles corresponding to the non-vanishing extensions $\Ext^1(b_2, e^{\prime \prime}_1)$ and $\Ext^1 (b_1, e^\prime_2)$ above are 
\[
\tri{e^{\prime \prime}_1}{x \oplus e^{\prime \prime}_2}{b_2} \text{ and } \tri{e^\prime_2}{x \oplus e^\prime_1}{b_1},
\]
respectively, where $x = \excoray{a}{b_2} \cap \exray{a}{b_1}$. 
\end{remark}

\begin{lemma}\label{lem:extbism>=3}
Suppose  $b_1, b_2 \in \ind{\Tw}$ are Hom-orthogonal objects in $\extfrom{a} \cup \extto{a}$ such that $\exray{a}{b_1} < \exray{a}{b_2}$. We have:
\begin{compactenum}  
\item If $w \neq 0$ then $\Ext^1 (b_1, b_2) = 0 = \Ext^1 (b_2, b_1)$.
\item If $w = 0$ then $\Ext^1 (b_1, b_2) = 0 = \Ext^1 (b_2, b_1)$ provided there is $b \in \extfrom{a} \cup \extto{a}$ for which $b_1$, $b_2$ and $b$ are pairwise Hom-orthogonal and $\exray{a}{b_1} < \exray{a}{b} < \exray{a}{b_2}$.
\end{compactenum}
\end{lemma}

\begin{proof}
Statement (1) can be checked by a case analysis as in Lemma \ref{lem:extensionsaisandeis}.

Now let $w = 0$. We shall only show that if $\Ext^1 (b_1, b_2) \neq 0$ then there is no $b$ satisfying the conditions of statement (2), since the argument is dual for $\Ext^1 (b_2, b_1) \neq 0$. Suppose $\Ext^1 (b_1, b_2) \neq 0$. We only consider the case when $b_2 \in \extto{a}$; the case when $b_2 \in \extfrom{a}$ is similar.

Since $b_1$ and $b_2$ are Hom-orthogonal and $b_1 \in (\extfrom{b_2} \cup \extto{b_2}) \cap (\extfrom{a} \cup \extto{a})$, we have that $b_1$ must lie in the region $X \cup Y \cup Z$ indicated in Figure \ref{fig:w=0exrays}.
If $b_1 \in Y \cup Z$, then $\exray{a}{b_2} < \exray{a}{b_1}$, contradicting the hypothesis. Hence, $b_1 \in X$. If there is an indecomposable object $b$ in $\extfrom{a} \cup \extto{a}$ such that $\exray{a}{b_1} < \exray{a}{b} < \exray{a}{b_2}$, then $b$ must lie in the shaded area of lower sketch in Figure \ref{fig:w=0exrays}. But then we have either $\Hom (b_1, b) \neq 0$ or $\Hom (b_2, b) \neq 0$, meaning that $b_1$, $b_2$ and $b$ are not pairwise Hom-orthogonal.
\end{proof}

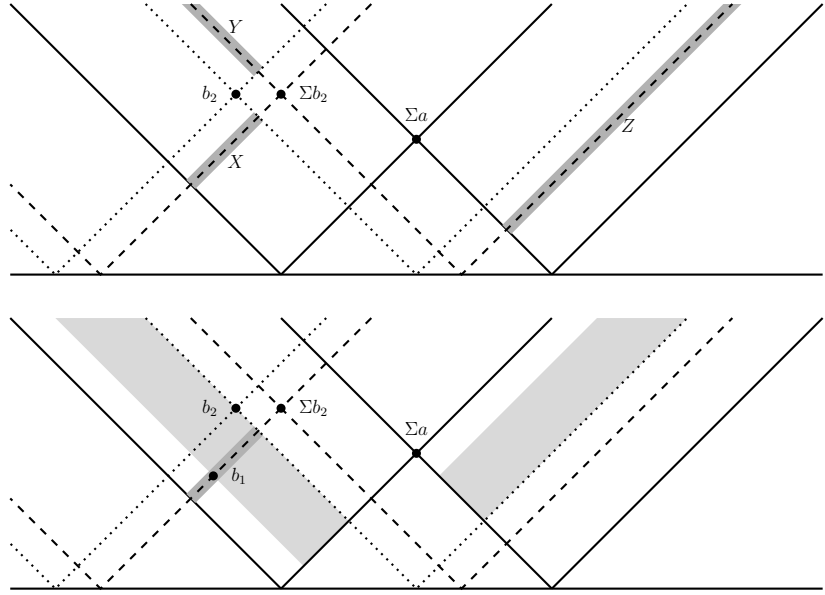
\begin{figure}
\begin{center}
\begin{tikzpicture}[thick,scale=0.6, every node/.style={scale=0.6}]

%lines:
\draw[thick] (0,3) -- (3,6);
\draw[thick] (3,0) -- (9,6);
\draw[thick] (-9,6) -- (-3,0);
\draw[thick] (-3,6) -- (0,3);
\draw[thick] (0,3) -- (3,0);
\draw[thick] (0,3) -- (-3,0);

\draw[thick] (-9,0) -- (9,0); 

% X region
\fill[gray!60] (-5.1,2.1) -- (-4.9,1.9) -- (-3.4,3.4) -- (-3.6,3.6) -- cycle;
\node (X) at (-4,2.5) {$X$};

% Y region
\fill[gray!60] (-5.2,6) -- (-4.8,6) -- (-3.4,4.6) -- (-3.6,4.4) -- cycle;
\node (Y) at (-4,5.5) {$Y$};

% Z region
\fill[gray!60] (1.9, 1.1) -- (2.1,0.9) -- (7.2,6) -- (6.8,6) -- cycle;
\node (Z) at (5,3) {};
\node[above left] at (Z) {$Z$};

% right \Sigma b hammock
\draw[thick, dashed] (-1,6) -- (-3,4) -- (1,0) -- (7,6);

% left \Sigma b hammock
\draw[thick, dashed] (-5,6) -- (-3,4) -- (-7,0) -- (-9,2);

% right b hammock
\draw[thick, dotted] (-2,6) -- (-4,4) -- (0,0) -- (6,6);

% left \Sigma b hammock
\draw[thick, dotted] (-6,6) -- (-4,4) -- (-8,0) -- (-9,1);

% \Sigma a:
\fill (0,3) circle (1mm);
\node (a) at (0,3.25) {};
\node[above] at (a) {$\Sigma a$};

% \Sigma b
\fill (-3,4) circle (1mm);
\node (Sb) at (-2.75,4) {};
\node[right] at (Sb) {$\Sigma b_2$};

% b
\fill (-4,4) circle (1mm);
\node (b) at (-4.25,4) {};
\node[left] at (b) {$b_2$};

\end{tikzpicture}
\end{center}
\bigskip
\begin{center}
\begin{tikzpicture}[thick,scale=0.6, every node/.style={scale=0.6}]

% shaded region
\fill[gray!30] (-8,6) -- (-2.5,0.5) -- (-1.5,1.5) -- (-6,6);
\fill[gray!30] (0.5,2.5) -- (1.5,1.5) -- (6,6) -- (4,6);

%lines:
\draw[thick] (0,3) -- (3,6);
\draw[thick] (3,0) -- (9,6);
\draw[thick] (-9,6) -- (-3,0);
\draw[thick] (-3,6) -- (0,3);
\draw[thick] (0,3) -- (3,0);
\draw[thick] (0,3) -- (-3,0);

\draw[thick] (-9,0) -- (9,0); 

% X region
\fill[gray!60] (-5.1,2.1) -- (-4.9,1.9) -- (-3.4,3.4) -- (-3.6,3.6) -- cycle;

% right \Sigma b hammock
\draw[thick, dashed] (-1,6) -- (-3,4) -- (1,0) -- (7,6);

% left \Sigma b hammock
\draw[thick, dashed] (-5,6) -- (-3,4) -- (-7,0) -- (-9,2);

% right b hammock
\draw[thick, dotted] (-2,6) -- (-4,4) -- (0,0) -- (6,6);

% left \Sigma b hammock
\draw[thick, dotted] (-6,6) -- (-4,4) -- (-8,0) -- (-9,1);

% \Sigma a:
\fill (0,3) circle (1mm);
\node (a) at (0,3.25) {};
\node[above] at (a) {$\Sigma a$};

% \Sigma b
\fill (-3,4) circle (1mm);
\node (Sb) at (-2.75,4) {};
\node[right] at (Sb) {$\Sigma b_2$};

% b
\fill (-4,4) circle (1mm);
\node (b) at (-4.25,4) {};
\node[left] at (b) {$b_2$};

% b'
\fill (-4.5,2.5) circle (1mm);
\node (b') at (-4.25,2.5) {};
\node[right] at (b') {$b_1$};

\end{tikzpicture}
\end{center}
\caption{The regions used in the proof of Lemma~\ref{lem:extbism>=3}. The Hom-hammocks for $\Sigma a$, $\Sigma b_2$ and $b_2$ are shown in solid, broken and dotted lines, respectively.} \label{fig:w=0exrays}
\end{figure}

We are now ready to compute the middle term of the extension in Setup~\ref{setup:decomposable}.

\begin{proposition} \label{prop:m=1incomparable}
Suppose we are in Setup~\ref{setup:decomposable} and we have ordered the objects $b_i$ such that $\exray{a}{b_i} < \exray{a}{b_{i+1}}$, for $1 \leq i < m$. In the non-split triangle 
\[
\tri{a}{e}{\bigoplus_{i=1}^m b_i}, \text{ we have } e \simeq e^{\prime}_1 \oplus x_1 \oplus \cdots \oplus x_{m-1} \oplus e^{\prime \prime}_m,
\]
where $x_i = \exray{a}{b_i} \cap \excoray{a}{b_{i+1}}$ for $i = 1, \ldots, m-1$.
\end{proposition}

\begin{proof}
We proceed by induction on $m$. For $m=1$ this is Theorem~\ref{thm:graphical-calculus} and Proposition~\ref{prop:graphical-calculus}. 
Let $m \geq 2$ and assume, by induction, we have computed
$\tri{a}{e^{\prime}_1 \oplus x_1 \oplus \cdots \oplus x_{m-2} \oplus e^{\prime \prime}_{m-1}}{\bigoplus_{i=1}^{m-1} b_i}$ and  
$\tri{a}{e^{\prime}_2 \oplus x_2 \oplus \cdots \oplus x_{m-1} \oplus e^{\prime \prime}_m}{\bigoplus_{i=2}^m b_i}$.
Consider a diagram coming from the octahedral axiom in which the first column is the split triangle:
\[
\xymatrix{
                                     & \Sigma^{-1} (\bigoplus_{i=1}^{m-1} b_i) \ar@{=}[r] \ar[d] & \Sigma^{-1} (\bigoplus_{i=1}^{m-1} b_i) \ar[d] & \\
\Sigma^{-1} e \ar[r] \ar@{=}[d] & \Sigma^{-1} (\bigoplus_{i=1}^m b_i) \ar[r] \ar[d]  & a \ar[r] \ar[d]  & e \ar@{=}[d] \\
\Sigma^{-1} e \ar[r]      & \Sigma^{-1} b_m \ar[r]^{\alpha} \ar[d]^0          & f \ar[r] \ar[d]             & e  \\
                                     & \bigoplus_{i=1}^{m-1} b_i \ar@{=}[r]                        & \bigoplus_{i=1}^{m-1} b_i
}
\] 
where $f \simeq e^{\prime}_1 \oplus x_1 \oplus \cdots \oplus x_{m-2} \oplus e^{\prime \prime}_{m-1}$. If $\alpha = 0$, then $e \simeq b_m \oplus e^{\prime}_1 \oplus x_1 \oplus \cdots \oplus x_{m-2} \oplus e^{\prime \prime}_{m-1}$. Suppose now that $\alpha \neq 0$. By applying $\Hom(\Sigma^{-1} b_m, -)$ to the family of triangles $\{\tri{e_{i+1}^\prime}{e^{\prime}_i \oplus x_i}{b_i}\}_{1 \leq i \leq m-2}$, and using Lemmas \ref{lem:extensionsaisandeis} and \ref{lem:extbism>=3}, we have $\Ext^1 (b_m, x_i) = 0$, for $i= 1, \ldots, m-2$. On the other hand, $\ext^1 (b_m, e^{\prime \prime}_{m-1}) = 1$ and the corresponding non-split triangle is $\tri{e^{\prime \prime}_{m-1}}{e^{\prime \prime}_m \oplus x_{m-1}}{b_m}$. Hence, by Lemma \ref{lem:cone}(i), we have $e \simeq e^{\prime}_1 \oplus x_1 \oplus \cdots \oplus x_{m-1} \oplus e^{\prime \prime}_m$. 

Now consider the following application of the octahedral axiom:
\[
\xymatrix{
                                     & \Sigma^{-1} (\bigoplus_{i=2}^{m} b_i) \ar@{=}[r] \ar[d] & \Sigma^{-1} (\bigoplus_{i=2}^{m} b_i) \ar[d] & \\
\Sigma^{-1} e \ar[r] \ar@{=}[d] & \Sigma^{-1} (\bigoplus_{i=1}^m b_i) \ar[r] \ar[d]  & a \ar[r] \ar[d]  & e \ar@{=}[d] \\
\Sigma^{-1} e \ar[r]      & \Sigma^{-1} b_1 \ar[r]^{\alpha^\prime} \ar[d]^0          & f^\prime \ar[r] \ar[d]             & e  \\
                                     & \bigoplus_{i=2}^{m-2} b_i \ar@{=}[r]                        & \bigoplus_{i=2}^{m-2} b_i
}
\] 
where $f^\prime = e^{\prime}_2 \oplus x_2 \oplus \cdots \oplus x_{m-1} \oplus e^{\prime \prime}_m$. 
Using Lemmas \ref{lem:extensionsaisandeis} and \ref{lem:extbism>=3}, we have that $e \simeq b_1 \oplus e^{\prime}_2 \oplus x_2 \oplus \cdots \oplus x_{m-1} \oplus e^{\prime \prime}_m$ if $\alpha^\prime = 0$, and $e \simeq e^{\prime}_1 \oplus x_1 \oplus \cdots \oplus x_{m-1} \oplus e^{\prime \prime}_m$ otherwise. Since the cone, $e$, must coincide in both diagrams, we must have $\alpha \neq 0, \alpha^\prime \neq 0$ and $e \simeq  e^{\prime}_1 \oplus x_1 \oplus \cdots \oplus x_{m-1} \oplus e^{\prime \prime}_m$. 
\end{proof}

\subsection{Extensions involving $\Sigma a$}

In the previous subsection we excluded $b_i = \Sigma a$ for each $i$ from our analysis. The following proposition deals with this case.

\begin{proposition} \label{prop:Sigma-a}
Let $w=0$ and $b_1, \ldots, b_k \in \extfrom{a} \cup \extto{a}$ with $b_i \neq \Sigma a$, for $i = 1, \ldots, k$. Consider the triangle 
\[
\trilabels{a}{e}{(\bigoplus_{i=1}^k b_i) \oplus \Sigma a}{}{}{\rowmat{H}{h}},
\]
where $H = \rowvec{h_1}{\cdots}{h_k}$. Then either $e \simeq \bigoplus_{i=1}^k b_i$ or $e \simeq \cocone{h} \oplus \Sigma a$. 
\end{proposition}

\begin{proof}
Suppose $h$ is an isomorphism. Then clearly every map $h_i$ factors through $h$. By Lemma \ref{lem:reduction}, and because the cocone of $h$ is zero, we have an isomorphism of triangles:
\[
\xymatrix@!R=10pt{
a \ar[r] \ar@{=}[d] & e \ar[r] \ar@{-}[d]^-{\simeq}  & (\bigoplus_{i=1}^k b_i) \oplus \Sigma a \ar[r]^-{\rowmat{H}{h}} \ar@{-}[d]^-{\simeq}_{F} & \Sigma a \ar@{=}[d] \\
a \ar[r]    & \bigoplus_{i=1}^k b_i \ar[r] & (\bigoplus_{i=1}^k b_i) \oplus \Sigma a \ar[r]_-{\rowmat{0}{h}}             & \Sigma a,
}
\]
where $F$ is the matrix with 1s along the leading diagonal and $h^{-1} h_i$ as the entry in position $(k+1, i)$, for $i = 1, \ldots, k$.

Now suppose $h$ is a non-isomorphism. By Proposition \ref{prop:factoring-T0}(ii), $h$ factors through any of the $h_i$'s. Write, for instance, $h = h_1 \beta$, with $\beta: \Sigma a \rightarrow b$ nonzero. Again, by Lemma \ref{lem:reduction}, we have an isomorphism of triangles:
\[
\xymatrix@!R=10pt{
a \ar[r] \ar@{=}[d] & e \ar[r] \ar@{-}[d]^-{\simeq}  & (\bigoplus_{i=1}^k b_i) \oplus \Sigma a \ar[r]^-{\rowmat{H}{h}} \ar@{-}[d]^-{\simeq}_{F} & \Sigma a \ar@{=}[d] \\
a \ar[r]    & f \oplus \Sigma a \ar[r] & (\bigoplus_{i=1}^k b_i) \oplus \Sigma a \ar[r]_-{\rowmat{H}{0}}             & \Sigma a,
}
\]
where $F$ is the matrix with 1s along the leading diagonal and $\beta$ as the entry in position $(1,k+1)$ and $f = \cocone{h}$.
\end{proof}

\subsection{Proof of Theorem~\ref{thm:extensionsnonindecomposable}}

We are now ready to prove the main result of this section.

\begin{lemma} \label{lem:n=1}
Any extension of the form $\tri{a}{e}{\bigoplus_{i=1}^m b_i}$ in $\sT_w$, $w \in \bZ \setminus \{1\}$, can be computed iteratively from extensions whose outer terms are indecomposable and involving only $a, b_1, b_2, \ldots, b_m$ and objects constructed from these objects. 
\end{lemma}

\begin{proof}
Immediate from Corollary~\ref{cor:m=1directsum} and Propositions~\ref{prop:m=1incomparable} and \ref{prop:Sigma-a}.
\end{proof}

\begin{proof}[Proof of Theorem \ref{thm:extensionsnonindecomposable}]
We proceed by induction on $n$. For $n=1$, this is Lemma~\ref{lem:n=1}. Suppose $n > 1$ and consider the octahedral axiom diagram in which the first column is split.
\[
\xymatrix{
                                     & a_1 \ar@{=}[r] \ar[d] & a_1 \ar[d] & \\
\Sigma^{-1} (\bigoplus_{i=1}^m b_i) \ar[r] \ar@{=}[d] & \bigoplus_{i=1}^n a_i \ar[r] \ar[d]  & e \ar[r] \ar[d]  & \bigoplus_{i=1}^m b_i \ar@{=}[d] \\
\Sigma^{-1} (\bigoplus_{i=1}^m b_i) \ar[r]      & \bigoplus_{i=2}^n a_i \ar[r] \ar[d]^0          & \bigoplus_{i=1}^k x_i \ar[r] \ar[d]             & \bigoplus_{i=1}^m b_i  \\
                                     & \Sigma a_1 \ar@{=}[r]                        & \Sigma a_1
}
\]
By induction, the triangle $\bigoplus_{i=2}^n a_i \rightarrow \bigoplus_{i=1}^k x_i \rightarrow \bigoplus_{i=1}^{m} b_i \rightarrow \Sigma (\bigoplus_{i=2}^n a_i)$ is constructed from extensions whose outer terms are indecomposable, which are built from the $a_i$'s and the $b_j$'s. Therefore, by Lemma~\ref{lem:n=1}, it follows that so is $\tri{a_1}{e}{\bigoplus_{i=1}^k x_i}$, as required.
\end{proof}

%=========================================================================
% Section
\section{The combinatorial model} \label{sec:comb}
%=========================================================================

Here we recall the combinatorial model for $\Tw$ from \cite{HJ} in the case $w\geq 2$ and its natural extension to $w\leq 0$ in terms of `arcs/diagonals of the $\infty$-gon'. Namely, we regard each pair of integers $(t,u)$ as an arc connecting the integers $t$ and $u$. 

For $w\in \bZ\setminus\{1\}$ set $d=w-1$. A pair of integers $(t,u)$ is called a \emph{$d$-admissible arc} if
\begin{compactenum}[(i)]
\item for $w \geq 2$, one has $u-t \geq w$ and $u-t \equiv 1 \modulo d$;
\item for $w=0$, a $(-1)$-admissible arc $(t,u)$ is one with $u-t \leq 0$; and
\item for $w \leq -1$, one has $u-t \leq w$ and $u-t \equiv 1 \modulo d$;
\end{compactenum}

\begin{notation}
Whenever the value of $w$ is not specified, the arc incident with the distinct integers $t$ and $u$ is denoted by $\{t,u\}$. In other words, assuming $t > u$, $\{t,u\} = (t,u)$ when $w \leq -1$ and $\{t,u\} = (u,t)$ when $w \geq 2$. 
\end{notation}

The \emph{length} of the arc $(t,u)$ is $|u-t|$.

When $d$ is clear from context we refer to $d$-admissible arcs simply as \emph{admissible arcs}. Figure~\ref{fig:AR-quiver2} shows how admissible arcs correspond to the indecomposable objects of $\Tw$ when $w \neq 1$. We note that the action of the suspension, AR translate and Serre functor in this model are given by
\[
\Sigma(t,u) = (t-1,u-1) \quad
\tau(t,u) = (t-d,u-d) \quad
\SSS(t,u) = (t-w,u-w).
\]

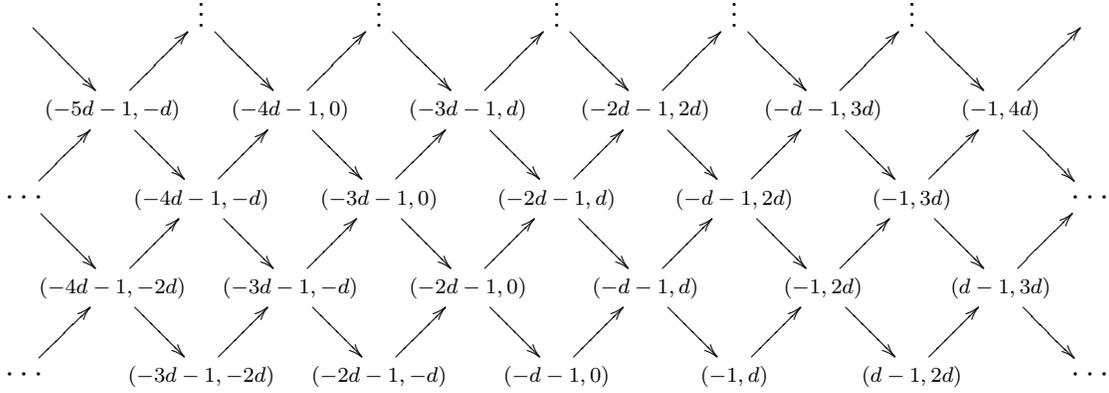
\begin{figure}
\centerline{
\xymatrix@!=0.8pc{
\ar[dr]    &                      & \vdots \ar[dr]           &                   & \vdots \ar[dr]                  &          & \vdots \ar[dr] &                      & \vdots \ar[dr]       &                        & \vdots \ar[dr]         &                      &  \\
                        & \smxy{(-5d-1,-d)} \arr &                        & \smxy{(-4d-1,0)} \arr &                      & \smxy{(-3d-1,d)} \arr &                 & \smxy{(-2d-1,2d)} \arr &                        & \smxy{(-d-1,3d)} \arr  &                         & \smxy{(-1,4d)} \arr &                  \\
\cdots \arr        &                      & \smxy{(-4d-1,-d)}  \arr         &                   & \smxy{(-3d-1,0)}  \arr            &          & \smxy{(-2d-1,d)} \arr &                      & \smxy{(-d-1,2d)}  \arr     &                        & \smxy{(-1,3d)}  \arr    &                      & \cdots \\
                        & \smxy{(-4d-1,-2d)} \arr &                        & \smxy{(-3d-1,-d)} \arr &                      & \smxy{(-2d-1,0)} \arr &                 & \smxy{(-d-1,d)} \arr &                        & \smxy{(-1,2d)} \arr  &                         & \smxy{(d-1,3d)} \arr &                  \\
\cdots \ar[ur] &                      & \smxy{(-3d-1,-2d)} \ar[ur] &                   & \smxy{(-2d-1,-d)} \ar[ur] &          & \smxy{(-d-1,0)} \ar[ur]     &                     & \smxy{(-1,d)} \ar[ur] &                       & \smxy{(d-1,2d)} \ar[ur] &                      & \cdots
}}
\caption{\label{fig:AR-quiver2} A component of type $\bZ \Ainf$ with the endpoints of the $d$-admissible arcs described.}
\end{figure}

Let $\tA$ be a collection of $d$-admissible arcs. Then an integer $t$ is called a \emph{left fountain} of $\tA$ if $\tA$ contains infinitely many arcs of the form $\{s,t\}$ with $s \leq t$. Dually, one defines a \emph{right fountain} of $\tA$; a \emph{fountain} of $\tA$ is both a left fountain and a right fountain.

By Proposition~\ref{prop:Iyama-Yoshino}, to obtain a characterisation of torsion pairs in $\Tw$ we need to characterise extension-closed contravariantly finite subcategories $\sX$ of $\Tw$. 
The computations of Sections~\ref{sec:extensions} and \ref{sec:non-indecouterterms} will be used to characterise the extension-closed subcategories. Below, we state a characterisation of contravariantly finite subcategories in $\Tw$.

%===============================================================================
%  SECTION
\section{Contravariant-finiteness} \label{sec:contra}
%===============================================================================

The characterisation of contravariantly finite subcategories in $\Tw$ is as follows:

\begin{proposition}\label{prop:contravariant-finiteness}
Suppose $w \in \bZ \setminus\{1\}$. A full subcategory $\sX$ of $\Tw$ is contravariantly finite if and only if, for each $i \in \bZ$, an infinitude of objects in $\sX \cap \rayfrom{\Sigma^i X_0}$ forces an infinitude of objects in $\sX \cap \corayto{\SSS \Sigma^i X_0}$. 
\end{proposition}

This can be neatly described in terms of fountains.

\begin{corollary}\label{cor:contravariant-finiteness}
Suppose $w \in \bZ \setminus\{1\}$. Let $\sX$ be a full subcategory of $\Tw$ and $\tX$ be the set of admissible arcs corresponding to $\ind{\sX}$. Then
\begin{compactenum}
\item If $w \geq 2$, $\sX$ is contravariantly finite in $\Tw$ if and only if every right fountain in $\tX$ is also a left fountain.
\item If $w \leq 0$, $\sX$ is contravariantly finite in $\Tw$ if and only if every left fountain in $\tX$ is also a right fountain.
\end{compactenum}
\end{corollary}

\begin{proof}
For $w \geq 2$, a right fountain at a given integer corresponds to having infinitely many objects from a ray, $\rayfrom{\Sigma^i X_0}$, in the subcategory $\sX$. The corresponding left fountain consists of infinitely many objects from $\corayto{\SSS \Sigma^i X_0}$. For $w \leq 0$ the situation is dual: a left fountain at a given integer corresponds to having infinitely many objects from a ray, $\rayfrom{\Sigma^i X_0}$ in $\sX$, and the corresponding right fountain consists of infinitely many objects from $\corayto{\SSS \Sigma^i X_0}$. The result then follows from Proposition \ref{prop:contravariant-finiteness}.
\end{proof}

When $w=2$ Propositon~\ref{prop:contravariant-finiteness} is \cite[Theorem 2.2]{Ng}; for $w > 2$ it is \cite[Proposition 2.4]{HJ}.
Note that the statements of \cite[Theorem 2.2]{Ng} and \cite[Proposition 2.4]{HJ} are in terms of fountains, however the authors prove what is stated in Proposition~\ref{prop:contravariant-finiteness}. The same arguments as in \cite{HJ} and \cite{Ng} also work, with minor modifications, for $w \leq 0$, so we refrain from giving full details. 

A crucial part of the arguments in \cite{HJ} and \cite{Ng} is 
showing that the composition of two so-called `backward maps' is zero.
The term `backward map' was first used in the proof of \cite[Theorem 2.2]{Ng} but not explicitly defined.

\begin{definition}
Assume $w\in \bZ\setminus \{1\}$. A \emph{backward map} in $\Tw$ is a non-zero morphism between indecomposable objects $f\colon a \to b$ with $b \in \homto{\SSS a}$. In $\sT_0$ we additionally require that $b$ is not isomorphic to $a$.  
\end{definition}

We require a new argument for the case $w=0$. While thinking about the case $w=0$, we noticed an unexplained step in the `backward maps' part of the proof in \cite[Theorem 2.2]{Ng}. For the convenience of the reader, we prove the required zero composition of `backward maps' in $\Tw$ for any $w \neq 1$.

\begin{lemma} \label{lem:backward}
Let $w \in \bZ\setminus \{1\}$. The composition of two backward maps in $\Tw$ is zero.
\end{lemma}

\begin{proof}
We deal with three cases: $w = 0$, $w=2$ and $w \neq 0,2$. Suppose $f \colon a \to b$ and $g \colon b \to c$ are backward maps. We start with the easier general case.

\Case{$w \neq 0,2$} Denote the components of the AR quiver of $\Tw$ by $\sC_0,\ldots,\sC_{|d|-1}$, where, as usual, $d=w-1$. When $w \neq 0,2$ we have $|d| > 1$, i.e. the AR quiver has more than one component.
By applying a power of the suspension if necessary, we may assume that $a \in \sC_0$. Since $f$ and $g$ are backward maps, it follows that $b \in \sC_1$ and $c \in \sC_2$, where the subscript is interpreted modulo $|d|$. When $|d| > 2$ we have $\Hom_{\Tw}(a,c)=0$ since $\homfrom{a} \cup \homto{\SSS a} \subseteq \sC_0 \cup \sC_1$ and $c \in \sC_2$. If $|d| = 2$ a direct computation shows that $\homto{\SSS b} \cap \homfrom{a} = \emptyset$, whence $\Hom_{\Tw}(a,c)=0$. Thus backward maps compose to zero whenever $w \neq 0,2$.

\Case{$w=0$} By Proposition~\ref{prop:factoring-T0} the map $gf$ factors as $a \rightlabel{f} b \rightlabel{g_1} a \rightlabel{g_2} c$, where $g_1 f$ is a non-isomorphism and $g_2$ is unique up to scalars. Applying Proposition~\ref{prop:factoring-T0} again gives that $g_2$ factors through $R(a)$ as $g_2 \colon a \rightlabel{h} R(a) \rightlabel{h'} c$. Since $a$ is not isomorphic to $b$, we have $R(a) \notin \homfrom{b} \cup \homto{b}$, whence the composite $g_2 g_1 = 0$, giving $gf=0$, as claimed.

\Case{$w=2$} Assume for a contradiction that $gf \neq 0$. By Proposition~\ref{prop:factoring}, the map $g$ factors through $a$ as $g \colon b \rightlabel{h} a \rightlabel{h'} c$. 
By \cite[Proposition 2.1 (i)]{HJ}, the map $h$ is a composition of irreducible maps $h=h_n \cdots h_1$ say. Thus, the map $gf$ factors as $a \rightlabel{f} b \rightlabel{h} a \rightlabel{h'} c$ and we have a nonzero map $hf\colon a \to a$. By Proposition~\ref{prop:hom-hammocks}, $\hom_{\sT_2}(a,a) = 1$, whence $hf = \lambda \id_a$ with $\lambda \neq 0$. It follows that $\frac{1}{\lambda} (h_{n-1} \cdots h_1 f)$ is a right inverse for $h_n$, i.e. $h_n$ is a split epimorphism. This contradicts the irreducibility of $h_n$, whence the original composition $gf$ must have been zero.
\end{proof}

To show how the composition of two backward maps being zero is used in the proof of Proposition~\ref{prop:contravariant-finiteness}, we sketch this part of the argument for $w=0$.

\begin{proof}[Proof of Proposition~\ref{prop:contravariant-finiteness}]
Firstly, we show the forward implication for $w=0$. Suppose, without loss of generality, that there are infinitely many objects $s_i \in \rayfrom{X_0}$ occurring as objects in $\sX$. We want to show that this implies there are infinitely many objects $c \in \corayto{X_0}$ occurring in $\sX$, recalling that $\SSS X_0 = X_0$. The diagrams in \cite[Proof of Theorem 2.2]{Ng} may be useful to help understand our arguments.

Since $\sX$ is contravariantly finite, there is a right $\sX$-approximation $x \to c$, where $x = x_1 \oplus \cdots \oplus x_n$. We may assume that the map $x \to c$ is nonzero from each summand. Since $c \in \homto{\SSS s_i} = \homto{s_i}$ for each $i \in \bN$, the map $s_i \to c$ factors as $s_i \to x \to c$. In particular, there is a summand, $x_i$ say, of $x$ such that the map $s_i \rightlabel{f} x_i \rightlabel{g} c$ is nonzero. Now inspecting the Hom-hammocks shows that $x_i$ either lies on $\corayto{X_0}$ or lies in the region of the AR quiver of $\sT_0$ bounded by the following:
\[
\corayto{X_0}, \ \corayto{R(s_i)}, \ \text{and,} \ \rayfrom{L(c)}, \ \rayfrom{X_0}.
\]
If $x_i \in \corayto{X_0}$ there is nothing to show, so suppose that $x_i \notin \corayto{X_0}$.

Now by Lemma~\ref{lem:backward}, unless $x_i \simeq s_i$, we have $gf=0$; a contradiction. Now, since $x$ contains only finitely many indecomposable summands, only finitely many of the $s_i$ may occur as summands of $x$. Taking an $s_j$ that is not a summand of $x$ thus yields the required contradiction. This shows that each of the $x_i$ must lie on $\corayto{X_0}$ above $c$. Repeating this argument indefinitely for $c \in \corayto{X_0}$ further and further from the mouth then gives infinitely many objects of $\coray{X_0}$ in $\sX$, as claimed.

The reverse implication works in exactly the same way as in \cite[Theorem 2.2]{Ng}.
\end{proof}

%===============================================================================
% Section
\section{Torsion pairs and the Ptolemy condition in $\Tw$ for $w\neq 0,1$}\label{sec:Ptolemy}
%===============================================================================

Throughout this section $w \in \bZ \setminus\{0,1\}$. We will give a combinatorial description of the extension closure of a subcategory of $\Tw$ using the combinatorial model presented in Section \ref{sec:comb}. This description is in terms of Ptolemy diagrams of different classes, which include those defined in \cite{Ng}.

For $a \in \ind{\Tw}$ denote the corresponding admissible arc by $\arc{a}$. If $\arc{a}=(t,u)$ then the \emph{starting point} is $s(\arc{a})=t$ and the \emph{ending point} is $t(\arc{a})=u$. 

Recall that when $w \leq -1$, the first coordinate of an admissible arc is strictly bigger than the second coordinate, and when $w \geq 2$ the opposite holds. 

\begin{definitions}
Let $\arc{a} = \{t, u\}, \arc{b} = \{v, w\}$, with $t > u$ and $v > w$, be two admissible arcs of $\Tw$. 
\begin{compactenum}
\item The arcs $\arc{a}$ and $\arc{b}$ are said to \textit{cross} if either $u < w < t < v$ or $w < u < v < t$.
\item If $\arc{a}$ and $\arc{b}$ are crossing arcs, then the \textit{Ptolemy arcs of class I} associated to $\arc{a}$ and $\arc{b}$ are the remaining four arcs connecting the vertices incident with $\arc{a}$ or $\arc{b}$, i.e. the set of Ptolemy arcs is $\{\{x,y\} \mid x, y \in \{t, u, v, w\}, x \neq y, \{x,y\} \neq \arc{a}, \arc{b} \}$.
\item The \textit{distance} between $\arc{a}$ and $\arc{b}$ is defined as
\[
d(\arc{a},\arc{b}) \coloneqq \min\{|t - v|,|u - w|,|t - w|,|u-v|\}.
\]
\item The arcs $\arc{a}$ and $\arc{b}$ are \textit{neighbouring} if they do not cross and $d(\arc{a},\arc{b}) = 1$. 
\item If $\arc{a}$ and $\arc{b}$ are neighbouring arcs and $d(\arc{a},\arc{b})$ is given by the distance between vertices $x$ and $x-1$, then the corresponding \textit{Ptolemy arc of class II} is the arc connecting the vertices incident with $\arc{a}$ or $\arc{b}$ which are not $x$ and $x-1$.   
 \end{compactenum}
Figures \ref{fig:ptolemy-arcs} and \ref{fig:modified-ptolemy} illustrate these concepts. For brevity we shall sometimes refer to Ptolemy arcs of both classes simply as Ptolemy arcs.
\end{definitions}

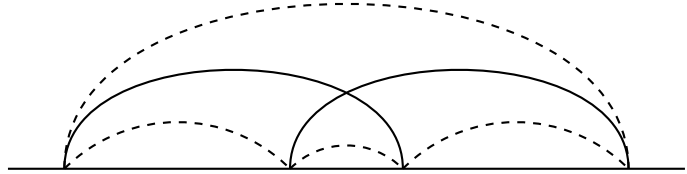
\begin{figure}[!ht]
\begin{center}
\begin{tikzpicture}[thick,scale=0.75, every node/.style={scale=0.75},out=90,in=90]

\draw (0,1) -- (12,1);

%arcs a and b:
\path (1,1) edge (7,1);
\path (5,1) edge (11,1);

%corresponding Ptolemy arcs:
\path[out=45,in=135,dashed] (1,1) edge (5,1);
\path[dashed] (1,1) edge (11,1);
\path[out=45,in=135,dashed] (5,1) edge (7,1);
\path[out=45,in=135,dashed] (7,1) edge (11,1);

\end{tikzpicture}
\end{center}
\caption{The Ptolemy arcs of class I.}
\label{fig:ptolemy-arcs}
\end{figure}

\begin{figure}[!ht]
\begin{center}
\begin{tikzpicture}[thick,scale=0.75, every node/.style={scale=0.75}]

\draw (3,4) -- (10,4);

%neighbours which are not innerarcs:
\path[out=85,in=95] (3.5,4) edge (6,4);
\path[out=85,in=95] (7,4) edge (9.5,4);

%the modified Ptolemy arc:
\path[out=90,in=90,dashed] (3.5,4) edge (9.5,4);

%distance 1:
\draw[<->] (6,3.8)-- (7,3.8);

\node (arrow) at (6.5,3.8) {};
\node [below] at (arrow.south) {1};

%neighbours which are innerarcs:
\draw (1,1) -- (6,1);

\draw (7,1) -- (12,1);

%neighbour to the left:
\path[out=85,in=95] (1.5,1) edge (5.5,1);
\path[out=85,in=95] (2.5,1) edge (4,1);
\path[out=45,in=135,dashed] (4,1) edge (5.5,1);

\draw[<->] (1.5,0.8)-- (2.5,0.8);
 %node(1)[above] {1}

\node (arrow1) at (2,0.8) {};
\node [below] at (arrow1.south) {1};

%neighbour to the right:
\path[out=85,in=95] (7.5,1) edge (11.5,1);
\path[out=85,in=95] (9,1) edge (10.5,1);
\path[out=45,in=135,dashed] (7.5,1) edge (9,1);

\draw[<->] (10.5,0.8)-- (11.5,0.8);
\node (arrow2) at (11,0.8) {};
\node [below] at (arrow2.south) {1};

\end{tikzpicture}
\end{center}
\caption{The Ptolemy arcs of class II.}
\label{fig:modified-ptolemy}
\end{figure}
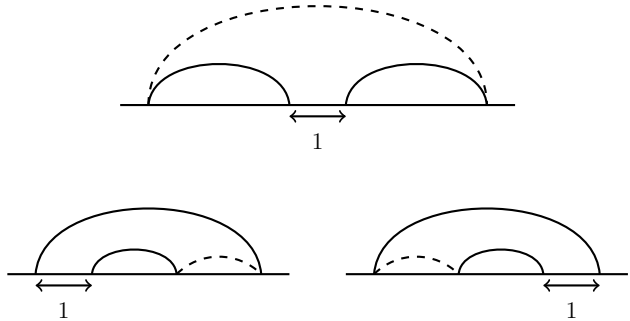

Recall that for a full subcategory $\sX$ of $\Tw$, $\extn{\sX}$ denotes the smallest extension-closed subcategory of $\Tw$ containing $\sX$. In this section, we prove the following main result.

\begin{theorem}\label{thm:extensionclosurewneq01}
Let $\sX$ be a full additive subcategory of $\Tw$ and let $\tX$ be the arcs corresponding to the objects of $\ind{\sX}$. Then the objects of $\ind{\extn{\sX}}$ correspond to the arcs of
\begin{compactenum}
\item for $w \geq 2$, the closure of $\tX$ under admissible Ptolemy arcs of class I,
\item for $w \leq -1$, the closure of $\tX$ under admissible Ptolemy arcs of classes I and II. 
\end{compactenum}  
\end{theorem}

Putting this together with Propositions~\ref{prop:Iyama-Yoshino} and \ref{cor:contravariant-finiteness} gives us part of Theorem~\ref{thm:A}:

\begin{corollary}
Let $w\in \bZ\setminus\{0,1\}$, $\sX$ be a full additive subcategory of $\Tw$ and $\tX$ be the corresponding set of arcs. Then $(\sX,\sX^{\perp})$ is a torsion pair in $\Tw$ if and only if
\begin{compactenum}
\item for $w \geq 2$, any right fountain in $\tX$ is also a left fountain and $\tX$ is closed under taking admissible Ptolemy arcs of class I.
\item for $w \leq -1$, any left fountain in $\tX$ is also a right fountain and $\tX$ is closed under taking admissible Ptolemy arcs of classes I and II.
\end{compactenum}
\end{corollary}

\subsection{Combinatorial description of the Ext-hammocks}
We first need to introduce some notation regarding \textit{partial fountains}, which is borrowed from \cite{CS13}. 

\begin{notation}
Let $v>t>u$ be integers such that $\{t,u\}$ and $\{t,v\}$ are $d$-admissible arcs. Define the \emph{partial right fountain at $t$ starting at $v$} and the \emph{partial left fountain at $t$ starting at $u$} by
\begin{align*}
\rfountain{t}{v} & \coloneqq \{\{t,x\} \ \text{$d$-admissible} \mid x \geq v\}; \\
\lfountain{t}{u} & \coloneqq \{\{t,y\} \ \text{$d$-admissible} \mid y \leq u\}.
\end{align*}
Let $V \subseteq \bZ$ such that $\{t,v\}$ is a $d$-admissible arc for each $v \in V$. Write
\[
\rfountain{V}{t} \coloneqq \bigcup_{v \in V} \rfountain{v}{t}
\ \text{and} \
\lfountain{V}{t} \coloneqq \bigcup_{v \in V} \lfountain{v}{t}.
\]
\end{notation}

Below is a description of the Ext-hammocks in terms of partial fountains. 

\begin{lemma}\label{partialfountainslemma}
Let $a, b \in \ind{\Tw}$, and $V_{\arc{a}} = \{s(\arc{a}) + id \mid i = 1, \ldots, k \}$, where $k \geq 1$ is such that $t(\arc{a}) - s(\arc{a}) = kd +1$.
\begin{compactenum}
\item If $w \geq 2$, then $\Ext^1_{\Tw}(b,a) \neq 0$ if and only if $\arc{b} \in \lfountain{V_{\arc{a}}}{s(\arc{a})-1} \cup \rfountain{V_{\arc{a}}}{t(\arc{a})+d}$. In particular, 
\[
b \in \extfrom{a} \iff \arc{b} \in \rfountain{V_{\arc{a}}}{t(\arc{a})+d}
\ \text{and} \
b \in \extto{a} \iff \arc{b} \in \lfountain{V_{\arc{a}}}{s(\arc{a})-1}.
\]
\item If $w \leq 0$, then $\Ext^1_{\Tw}(b,a) \ne 0$ if and only if $\arc{b} \in \rfountain{V_{\arc{a}}}{s(\arc{a})-1} \cup \lfountain{V_{\arc{a}}}{t(\arc{a})+d}$.
In particular,
\[
b \in \extfrom{a} \iff \arc{b} \in \lfountain{V_{\arc{a}}}{t(\arc{a})+d}
\ \text{and} \
b \in \extto{a} \iff \arc{b} \in \rfountain{V_{\arc{a}}}{s(\arc{a})-1}.
\]
\end{compactenum}
\end{lemma}

\begin{proof}
The case when $w \leq -1$ is \cite[Remark 2.2]{CS13}. In the case when $w = 0$, the first coordinate of an admissible arc is greater than or equal to the second coordinate, and so the result is the same as in $w \leq -1$. When $w \geq 2$, the admissible arcs are ordered from left to right, and so we swap $\rfountain{-}{-}$ and $\lfountain{-}{-}$.
\end{proof}

The following two propositions are direct consequences of the previous lemma.

\begin{proposition}\label{Extnonzerowpositive}
Let $V_{\arc{a}}$ be as in Lemma~\ref{partialfountainslemma}. If $w \geq 2$, and $a, b \in \ind{\Tw}$, then $\Ext^1_{\Tw} (b,a) \ne 0$ if and only if $\arc{a}$ and $\arc{b}$ cross and $\arc{b}$ is incident with a vertex in $V_{\arc{a}}$.
\end{proposition}

\begin{proposition}\label{cor:Extnonzerownegative}
Let $V_{\arc{a}}$ be as in Lemma \ref{partialfountainslemma}. If $w \leq -1$, and $a, b \in \ind{\Tw}$, then $\Ext^1_{\Tw} (b,a) \ne 0$ if and only if we are in one of the following situations:
\begin{compactenum}
\item $\arc{a}$ and $\arc{b}$ cross and $\arc{b}$ is incident with a vertex in $V_{\arc{a}} \setminus \{t(\arc{a})-1\}$;
\item $\arc{a}$ and $\arc{b}$ are neighbouring arcs such that $\arc{b}$ is incident with $s(\arc{a})-1$ or $t(\arc{a})-1$;
\item $\arc{b} = (s(\arc{a})-1, t(\arc{a})-1)$, i.e. $b = \Sigma a$. 
\end{compactenum}
\end{proposition}

\subsection{The middle terms of extensions correspond to admissible Ptolemy arcs}

We now define some arcs associated with $a, b \in \ind{\Tw}$ for which $\Ext^1_{\Tw}(b,a) \neq 0$. A priori these arcs need not be admissible, but when they are, they will correspond to the indecomposable summands of the middle term of the extension.

\begin{definition} \label{def:es}
Let $a, b \in \ind{\Tw}$ and suppose $\Ext^1_{\Tw} (b,a) \neq 0$.
\begin{compactenum}
\item If $b \in \extfrom{a}$ then $\arc{e_1} := (s(\arc{a}),t(\arc{b}))$ and $\arc{e_2} := (s(\arc{b}),t(\arc{a}))$.
\item If $b \in \extto{a}$ then $\arc{e_1} := (s(\arc{b}),s(\arc{a}))$ and $\arc{e_2} := (t(\arc{b}),t(\arc{a}))$.
\end{compactenum}
\end{definition}

\begin{proposition}\label{prop:middletermsarcs}
Let $a, b \in \ind{\Tw}$ be such that $\Ext^1_{\Tw} (b,a) \neq 0$, and $e_1$ and $e_2$ be as in Theorem \ref{thm:graphical-calculus}.
\begin{compactenum}
\item Suppose $b \in \extfrom{a}$. The arc $\arc{e_1}$ is always admissible. The arc $\arc{e_2}$ is admissible if and only if $s(\arc{b}) \neq t(\arc{a})-1$.
\item Suppose $b \in \extto{a}$. 
For $w \geq 2$, the arc $\arc{e_1}$ is admissible if and only if $s(\arc{b}) < s(\arc{a})-1$.
For $w \leq -1$, $\arc{e_1}$ is admissible if and only if $s(\arc{b}) > s(\arc{a})$. 
For $w \in \bZ\setminus\{0,1\}$, the arc $\arc{e_2}$ is admissible if and only if $t(\arc{b}) \ne t(\arc{a})-1$.  
\end{compactenum}
Moreover, $e_i$ is nonzero if and only if $\arc{e_i}$, which is the corresponding arc, is admissible. 
\end{proposition}
\begin{proof}
We will first see when $\arc{e_1}$ and $\arc{e_2}$ are admissible. Let $k$ and $V_a$ be as in Lemma \ref{partialfountainslemma}, and $k^\prime \geq 1$ be such that $t(\arc{b}) -s(\arc{b}) = k^\prime d +1$.

Suppose $b \in \extfrom{a}$. Then $s(\arc{b}) \in V_{\arc{a}}$, and so we have:
\begin{compactitem}
\item $t(\arc{e_1})-s(\arc{e_1}) = (k^{\prime} + i)d + 1$, for some $i \in \{1, \ldots, k\}$; and
\item $t(\arc{e_2})-s(\arc{e_2}) = (k - i)d + 1$. 
\end{compactitem}
Hence, $\arc{e_1}$ is always admissible, and $\arc{e_2}$ is admissible if and only if $k-i \geq 1$, i.e. $s(\arc{b}) \ne t(\arc{a}) -1$.

Suppose now that  $b \in \extto{a}$. Then $t(\arc{b}) \in V_{\arc{a}}$, and so we have:
\begin{compactitem}
\item $t(\arc{e_1})-s(\arc{e_1}) = (k^{\prime} -i)d + 1$, for some $i \in \{1, \ldots, k\}$; and
\item $t(\arc{e_2})-s(\arc{e_2}) = (k - i)d + 1$. 
\end{compactitem}
Hence, $\arc{e_1}$ is admissible if and only if $k^{\prime} -i \geq 1$, i.e. $s(\arc{b}) < s(\arc{a})-1$ if $w \geq 2$ and $s(\arc{b}) > s(\arc{a})$ if $w \leq -1$. On the other hand, $\arc{e_2}$ is admissible if and only if $k-i \geq 1$, i.e. $t(\arc{b}) \ne t(\arc{a}) -1$.

The last statement of the proposition follows from applying the following facts:
\begin{compactitem}
\item For $x \in \ind{\Tw}$, the arc corresponding to $L(\SSS x)$ is $(s(\arc{x})-w, s(\arc{x}))$ and the arc corresponding to $R(\SSS^{-1} x)$ is $(t(\arc{x}), t(\arc{x})+w)$.
\item Two indecomposable objects lie in the same ray (resp. coray) if and only if the first (resp. second) coordinate of the corresponding arcs coincides. \qedhere
\end{compactitem}
\end{proof}

\begin{corollary}
If $w \leq -1$ and $\arc{a}$ and $\arc{b}$ are neighbouring admissible arcs, then the corresponding Ptolemy arcs of class II are always admissible. 
\end{corollary}

\begin{definition} \label{def:E(a,b)}
Let $a, b \in \ind{\Tw}$. Write $a * b := \add{a} * \add{b}$, where the star product is defined on page~\pageref{star-product}. We define $E(a,b) := \ind{a * b \cup b * a}\setminus \{a,b\}$. Denote the set of arcs corresponding to the objects in $E(a,b)$ by $\tE (\arc{a},\arc{b})$.
\end{definition}

\begin{remark}\label{rem:addE(a,b)}
By Remark~\ref{rem:multiplesummands}, $E(a,b)$ consists of the indecomposable objects which occur as summands in the middle terms of non-split extensions $\tri{a}{e}{b}$ and $\tri{b}{f}{a}$. 
\end{remark}

\begin{remark} \label{rem:exceptionalcase}
Note that the only case in Proposition \ref{cor:Extnonzerownegative} where $\arc{a}$ and $\arc{b}$ are neither crossing nor neighbouring arcs is when $w = -1$, $a$ lies on the mouth of the AR quiver and $b = \Sigma a$. In this case, $a \notin \extfrom{\Sigma a} \cup \extto{\Sigma a}$, and so $\Ext^1_{\sT_{-1}} (a,b) = 0$. On the other hand, $\ext^1_{\sT_{-1}}(b,a) = 1$ and the extension of $b (= \Sigma a)$ by $a$ is $\trilabels{a}{0}{\Sigma a}{}{}{\id}$. Therefore $E(a,b) = \emptyset$.
\end{remark}

We have the following corollary to Proposition~\ref{prop:middletermsarcs}.

\begin{corollary} \label{cor:extensions-vs-Ptolemy}
Let $a,b\in \ind{\Tw}$ and $\arc{a},\arc{b}$ the corresponding admissible arcs. Then
\[
\tE(\arc{a},\arc{b}) \subseteq \{\text{admissible Ptolemy arcs incident with the endpoints of $\arc{a}$ and $\arc{b}$}\}.
\]
\end{corollary}

\begin{proof}
By Remark \ref{rem:addE(a,b)}, we have $E(a,b) = \{e_1, e_2, f_1, f_2\}$, where the $e_i$'s and $f_i$'s are such that the extensions are $\tri{a}{e_1 \oplus e_2}{b}$ and $\tri{b}{f_1 \oplus f_2}{a}$. Note that some of these objects may be zero. We will only check that $e_1$ and $e_2$ correspond to admissible Ptolemy arcs when nonzero, as the proof is analogous for $f_1$ and $f_2$. 

Suppose $e_i \neq 0$. By Remark~\ref{rem:exceptionalcase}, $\arc{a}$ and $\arc{b}$ are either crossing or neighbouring arcs. By Proposition \ref{prop:middletermsarcs}, $\arc{e_i}$ is admissible. It remains to check that $\arc{e_i}$ is a Ptolemy arc. By definition, the Ptolemy arcs and the arc $\arc{e_i}$ connect endpoints of $\arc{a}$ and $\arc{b}$. Ptolemy arcs of class I cover all the possibilities for this connection, so the only non-trivial case that we need to consider is when $\arc{a}$ and $\arc{b}$ are neighbours. The problem here lies in the fact that the Ptolemy arc of class II might not be the only admissible arc connecting the endpoints of $\arc{a}$ and $\arc{b}$. Namely, when $w = -1$, the arc $(x,x-1)$ connecting the two closest endpoints of $\arc{a}$ and $\arc{b}$ is also admissible, and it is not in general a Ptolemy arc of class II.

Since $\Ext^1(b,a) \neq 0$, we have $\arc{b}$ incident with either $s(\arc{a})-1$ or $t(\arc{a})-1$ (see Proposition \ref{cor:Extnonzerownegative}(2)). Hence, the arc in question is:
\begin{compactitem}
\item $(s(\arc{a}),s(\arc{b}))$, if $b$ is incident with $s(\arc{a})-1$, 
\item $(t(\arc{a}),s(\arc{b}))$ or $(t(\arc{a}),t(\arc{b}))$, if $b$ is incident with $t(\arc{a})-1$.
\end{compactitem}
Note that, by definition, $\arc{e_i}$ is not any of these arcs, and therefore $\arc{e_i}$ is the Ptolemy arc of class II. 
\end{proof}

Let $\sX$ be a full subcategory of $\Tw$. By the above and Theorem~\ref{thm:extensionsnonindecomposable}, the arcs corresponding to objects of $\ind{\extn{\sX}}$ are a subset of the closure of $\tX$ under admissible Ptolemy arcs. We now need to show that the inclusion in Corollary~\ref{cor:extensions-vs-Ptolemy} is in fact an equality.

\subsection{The extension closure}

Let $a, b \in \ind{\Tw}$. Combinatorially, Ptolemy arcs of class I and II arise out of the following situations, respectively:
\begin{compactitem}
\item $w \in \bZ\setminus\{0,1\}$ and $\arc{a}$ and $\arc{b}$ are crossing arcs,
\item $w \leq -1$ and $\arc{a}$ and $\arc{b}$ are neighbouring arcs.
\end{compactitem}

We shall show that in each of the two cases above $\tE(\arc{a},\arc{b})$ is precisely the set of all the admissible Ptolemy arcs of the appropriate class associated to $\arc{a}$ and $\arc{b}$. 
First, let us consider the case when $\arc{a}$ and $\arc{b}$ are neighbours. 

\begin{proposition}\label{prop:neighbourotherinclusion}
Let $w \leq -1$ and $a, b \in \ind{\Tw}$ be such that $\arc{a}$ and $\arc{b}$ are neighbouring arcs. Then $\Ext^1_{\Tw} (a,b) \neq 0$ or $\Ext^1_{\Tw} (b,a) \neq 0$, and $\tE(\arc{a},\arc{b})$ is the set of admissible Ptolemy arcs of class II associated to $\arc{a}$ and $\arc{b}$.  
\end{proposition}
\begin{proof}
If $\arc{b}$ is incident with $s(\arc{a})-1$ or $t(\arc{a})-1$, then $\Ext^1_{\Tw} (b,a) \neq 0$, by Proposition \ref{prop:middletermsarcs}. Dually, if $\arc{b}$ is incident with $s(\arc{a}) + 1$ or $t(\arc{a}) +1$, then $\Ext^1_{\Tw} (a,b) \neq 0$. In both cases, the middle term of the extension must be nonzero, since $b \neq \Sigma a, \Sigma^{-1} a$. Therefore, $\tE(\arc{a},\arc{b}) \neq \emptyset$, and so in particular, its cardinality is greater than or equal to one. 

Note that the set of (admissible) Ptolemy arcs of class II associated to $\arc{a}$ and $\arc{b}$ has cardinality one or two. If the cardinality is one, then by Corollary \ref{cor:extensions-vs-Ptolemy}, $\tE(\arc{a},\arc{b})$ must be equal to the set of admissible Ptolemy arcs associated to $\arc{a}$ and $\arc{b}$. Now, suppose there are two admissible Ptolemy arcs associated to $\arc{a}$ and $\arc{b}$. Then these must be of the form $(x,y)$ and $(x-1,y+1)$, with $x \geq y + 3$, for $w = -1$. In this case we have extensions in both directions, giving rise to the Ptolemy arcs $(x,x-1)$ and $(y,y-1)$. Hence, we also have equality in this case. 
\end{proof}

We now turn our attention to the case when $\arc{a}$ and $\arc{b}$ cross.

\begin{proposition}\label{prop:crossingwithextensionsS}
Let $a, b \in \ind{\Tw}$. If $\Ext^1_{\Tw} (b,a) \neq 0$ and $\arc{a}, \arc{b}$ are crossing arcs, then the set of admissible Ptolemy arcs of class I associated to $\arc{a}$ and $\arc{b}$ is contained in $\tE(\arc{a},\arc{b})$. 
\end{proposition}
\begin{proof}
Firstly, let us consider the case when $w \geq 2$. 

\Case{1} $b \in \extfrom{a}$.
We need to check whether $(s(\arc{a}),s(\arc{b}))$ and $(t(\arc{a}),t(\arc{b}))$ are admissible. Note that these two arcs are Ptolemy arcs associated to $\arc{a}$ and $\arc{b}$, but they do not correspond to the middle term of the extension $\tri{a}{e}{b}$.

We have $s(\arc{b}) = s(\arc{a}) +id$, for some $i = 1, \ldots, k$ because $s(\arc{b}) \in V_{\arc{a}}$. Hence, $s(\arc{b}) - s(\arc{a}) = id$, and $t(\arc{b}) - t(\arc{a}) = (i + k^\prime + k)d$.
Therefore, the arcs $(s(\arc{a}),s(\arc{b}))$ and $(t(\arc{a}),t(\arc{b}))$ are admissible if and only if $w=2$. Indeed, $\Ext^1_{\sT_2} (a,b) \simeq D \Ext^1_{\sT_2} (b,a)$ since $\sT_2$ is $2$-CY. Since $a \in \extto{b}$, $(s(\arc{a}),s(\arc{b}))$ and $(t(\arc{a}),t(\arc{b}))$ are the arcs corresponding to the indecomposable summands of the middle term of $\tri{b}{e^\prime}{a}$, by Proposition \ref{prop:middletermsarcs}. 

\Case{2} $b \in \extto{a}$.
We need to check whether $(s(\arc{b}),t(\arc{a}))$ and $(s(\arc{a}),t(\arc{b}))$ are admissible. We have $t(\arc{b}) = s(\arc{a}) + id$, for some $i = 1, \ldots, k$, and so, $t(\arc{b})-s(\arc{a}) = id$ and $t(\arc{a})-s(\arc{b}) = (k+k^\prime-i)d + 2$. Therefore, $(s(\arc{b}),t(\arc{a}))$ and $(s(\arc{a}),t(\arc{b}))$ are admissible if and only if $w = 2$, as in case 1. As we have seen above, when $w = 2$, $\Ext^1_{\sT_2} (a,b) \simeq \Ext^1_{\sT_2} (b,a) \neq 0$, and here we have $a \in \extfrom{b}$. Hence, by Proposition \ref{prop:middletermsarcs}, $(s(\arc{b}),t(\arc{a}))$ and $(s(\arc{a}),t(\arc{b}))$ are the arcs corresponding to the indecomposable summands of the middle term of the triangle $\tri{b}{e^\prime}{a}$.  

Now, let $w \leq -1$. We need to consider the following three cases.

\Case{1} $\arc{b}$ crosses $\arc{a}$, it is to the left of $\arc{a}$, and $s(\arc{b}) \in V_{\arc{a}}$. 
As we have seen in case 1 for $w \geq 2$, we have $s(\arc{b})-s(\arc{a}) = id$ and $t(\arc{b})-t(\arc{a}) = (i+k+k^\prime)d$. However, $d \leq -2$ since $w \leq -1$, and so the arcs $(s(\arc{a}),s(\arc{b}))$ and $(t(\arc{a}),t(\arc{b}))$ are never admissible. 

\Case{2} $\arc{b}$ crosses $\arc{a}$, it is to the right of $\arc{a}$, and $t(\arc{b}) \in V_{\arc{a}}$.
As in case 2 for $w \geq 2$, we have $t(\arc{b})-s(\arc{a}) = id$ and $t(\arc{a})-s(\arc{b}) = (k+k^\prime+ i) d +2$. But again, since $d \leq -2$, there is no $l \geq 1$ such that $t(\arc{b})-s(\arc{a}) = ld+1$ or $t(\arc{a})-s(\arc{b}) = ld+1$. Therefore, $(s(\arc{b}),t(\arc{a}))$ and $(s(\arc{a}),t(\arc{b}))$ are never admissible.

\Case{3} $b = \Sigma a$, i.e. $\arc{b} = (s(\arc{a})-1,t(\arc{a})-1)$, and $a$ does not lie on the mouth if $w = -1$.
It is easy to check that $(s(\arc{a}),t(\arc{b}))$ and $(s(\arc{b}),t(\arc{a}))$ are never admissible. On the other hand, $(s(\arc{a}),s(\arc{b}))$ and $(t(\arc{a}),t(\arc{b}))$ are admissible if and only if $w = -1$. From now on let $w = -1$. We have $\Ext^1_{\sT_{-1}} (a,\Sigma a) \simeq D \Hom_{\sT_{-1}} (a,\tau \SSS a) \neq 0$ if and only if $\tau \SSS a \in \homto{\SSS a}$, which always holds unless $a$ lies on the mouth of the AR quiver. 

Since, by hypothesis, $a$ does not lie on the mouth, we have $\Ext^1_{\sT_{-1}} (a,\Sigma a) \neq 0$ and $\Sigma a \in \extto{a}$. So $(s(\arc{a}),s(\arc{a})-1)$ and $(t(\arc{a}),t(\arc{a})-1)$, which are the only admissible Ptolemy arcs of class I associated to $a$ and $\Sigma a$, are the arcs corresponding to the indecomposable summands of the middle term of the triangle $\Sigma a \rightarrow e \rightarrow a \rightarrow \Sigma^2 a$. 
\end{proof}

In contrast to the case $w = 2$, for $w \in \bZ\setminus\{0,1,2\}$, there are crossing arcs whose corresponding indecomposable objects do not have extensions between them. We must check that these yield no admissible Ptolemy arcs.

\begin{lemma}\label{lemma:extvanishnotadmissiblePtolemyarcs}
Suppose, in addition to $w \in \bZ \setminus \{0, 1\}$, that $w \neq 2$ and let $a, b \in \ind{\Tw}$ be such that $\Ext^1_{\Tw} (a,b) = 0 = \Ext^1_{\Tw}(b,a)$ and $\arc{a}$ and $\arc{b}$ cross. Then none of the associated Ptolemy arcs of class I is admissible. 
\end{lemma}
\begin{proof}
Let $w \geq 3$ and suppose $\arc{a}$ and $\arc{b}$ cross each other in such a way that $s(\arc{a}) < s(\arc{b}) < t(\arc{a}) < t(\arc{b})$. Since there are no extensions between $a$ and $b$, we cannot have $s(\arc{b}) = s(\arc{a}) + id$, for some $i \geq 1$, nor can we have $t(\arc{a}) = s(\arc{b}) + jd$, for some $j \geq 1$. 

Since $\arc{a}$ and $\arc{b}$ are admissible arcs, we have $t(\arc{a}) - s(\arc{a}) = kd +1$ and $t(\arc{b}) - s(\arc{b}) = k^\prime d+1$, for some $k, k^\prime \geq 1$. We also have $s(\arc{b}) = s(\arc{a}) + x$ for some $x \geq 1$, and $t(\arc{a}) = s(\arc{b}) + y$, for some $y \geq 1$. 
  
Since $t(\arc{b})-s(\arc{a}) = t(\arc{b}) -s(\arc{b}) + s(\arc{b}) - s(\arc{a}) = k^\prime d +x+1$, we conclude that the Ptolemy arc $(s(\arc{a}),t(\arc{b}))$ is admissible if and only if $k^\prime d + x+ 1 = ld+1$, for some $l \geq 1$. This is equivalent to $x = (l-k^\prime) d$, and $l - k^\prime \geq 1$ since $s(\arc{a}) < s(\arc{b})$. However, this contradicts the hypothesis, and therefore this Ptolemy arc is not admissible. 

The arc $(s(\arc{b}),t(\arc{a}))$ is admissible if and only if $t(\arc{a}) - s(\arc{b}) = kd+x+1 = ld+1$, for some $l \geq 1$, i.e. $x = (k-l)d$. Since $x, d \geq 1$, we must also have $k-l \geq 1$, but this contradicts the hypothesis. 

Using the fact that $t(\arc{a}) = s(\arc{b}) + y$, for some $y \geq 1$, we can similarly show that $(t(\arc{a}), t(\arc{b}))$ is not admissible.

Finally, $(s(\arc{a}),s(\arc{b}))$ is admissible if and only if  $s(\arc{b})- s(\arc{a}) = x = ld+1$, for some $l \geq 1$. Suppose, for a contradiction, that this holds, i.e. $s(\arc{b}) = s(\arc{a}) + ld +1$, for some $l \geq 1$. Then we have $t(\arc{a}) = s(\arc{b}) + (k-l)d$, with $k-l \geq 1$, a contradiction. 

The proof for the case when $w \leq -1$ is analogous. 
\end{proof}

\begin{corollary}
Let $a, b \in \ind{\sT_w}$ be such that $\arc{a}$ and $\arc{b}$ are crossing arcs. Then $\tE(\arc{a},\arc{b})$ is the set of admissible Ptolemy arcs of class I associated to $\arc{a}$ and $\arc{b}$.
\end{corollary}
\begin{proof}
This follows from Corollary~\ref{cor:extensions-vs-Ptolemy}, Proposition~\ref{prop:crossingwithextensionsS} and Lemma~\ref{lemma:extvanishnotadmissiblePtolemyarcs}.
\end{proof}

Putting these together with Corollary~\ref{cor:extensions-vs-Ptolemy} and Proposition~\ref{prop:neighbourotherinclusion} yields the following corollary, which together with Theorem~\ref{thm:extensionsnonindecomposable} gives Theorem~\ref{thm:extensionclosurewneq01}.

\begin{corollary} \label{cor:Ptolemy-equality}
Let $a,b\in \ind{\Tw}$ and $\arc{a},\arc{b}$ be the corresponding admissible arcs. Then
\[
\tE(\arc{a},\arc{b}) = \{\text{admissible Ptolemy arcs incident with the endpoints of $\arc{a}$ and $\arc{b}$}\}.
\]
\end{corollary}

%==============================================================================
% Section
\section{Torsion pairs and the Ptolemy condition in $\sT_0$} \label{sec:Ptolemy-0}
%==============================================================================

In this section we complete the proof of Theorem~\ref{thm:A}. Throughout this section $w=0$. The AR quiver of $\sT_0$ has only one component since $|d| = 1$. Recall that $\sT_0$ is $0$-CY, i.e. for $a \in \sT_0$ we have $\Sigma a = \tau^{-1} a$, and $\SSS a = a$. In this case, the admissible arcs are pairs $(t,u)$ with $t \geq u$. In particular, for $a \in \ind{\sT_0}$ lying on the mouth of the AR quiver, the corresponding arc $\arc{a}$ is a \emph{loop}, i.e. a pair of integers $(x,x)$.

To classify extension closed subcategories of $\sT_0$ we need to introduce a new class of Ptolemy arcs.

\begin{definition}
Let $\arc{a}$ and $\arc{b}$ be $(-1)$-admissible arcs which are not loops. Note that we can have $\arc{a} = \arc{b}$. We say $\arc{a}$ and $\arc{b}$ are \emph{adjacent} if they are incident with a common vertex $x$. In this case, the \emph{Ptolemy arcs of class III} associated to $\arc{a}$ and $\arc{b}$ are the loop at $x$ and the arc connecting the other two vertices. See Figure \ref{fig:Ptolemy-classIII} for an illustration.
\end{definition}

\begin{figure}[!ht]
\begin{center}
\begin{tikzpicture}[thick,scale=0.75, every node/.style={scale=0.75}]

\draw (3,4) -- (10,4);

%neighbours which are not innerarcs:
\path[out=85,in=95] (4,4) edge (6.5,4);
\path[out=85,in=95] (6.5,4) edge (9,4);

%the Ptolemy arcs of class III:
\path[out=90,in=90,dashed] (4,4) edge (9,4);
\node (vertex) at (6.5,4) {};
\path[dashed] (vertex) edge [loop below] ();

%neighbours which are innerarcs:
\draw (1,1) -- (6,1);

\draw (7,1) -- (12,1);

%neighbour to the left:
\path[out=85,in=95] (1.5,1) edge (5.5,1);
\path[out=85,in=95] (1.5,1) edge (3.5,1);
\path[out=45,in=135,dashed] (3.5,1) edge (5.5,1);
\node (vertex2) at (1.5,1) {};
\path[dashed] (vertex2) edge [loop below] ();

%neighbour to the right:
\path[out=85,in=95] (7.5,1) edge (11.5,1);
\path[out=85,in=95] (9.5,1) edge (11.5,1);
\path[out=45,in=135,dashed] (7.5,1) edge (9.5,1);
\node (vertex3) at (11.5,1) {};
\path[dashed] (vertex3) edge [loop below] ();

\end{tikzpicture}
\end{center}
\caption{Ptolemy arcs of class III.}
\label{fig:Ptolemy-classIII}
\end{figure}
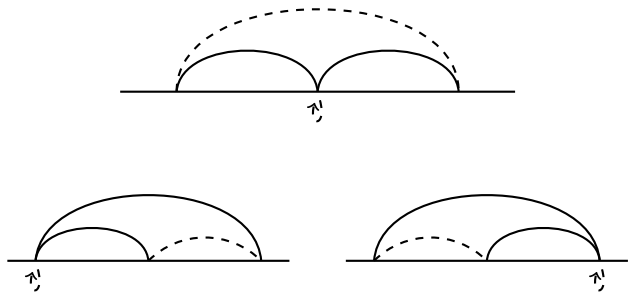

Note that the notion of crossing arcs only makes sense for non-loops. However, we admit loops in the notion of neighbouring arcs. Note also that Ptolemy arcs in $\sT_0$ are always admissible.

The aim of this section is to prove the following theorem.

\begin{theorem}\label{thm:extensionclosureT0}
Let $\sX$ be a full additive subcategory of $\sT_0$ and $\tX$ be the arcs corresponding to the objects of $\ind{\sX}$. Then the objects of $\ind{\extn{\sX}}$ correspond to the arcs of the closure of $\tX$ under Ptolemy arcs of classes I, II, and III.
\end{theorem}

Putting this together with Propositions~\ref{prop:Iyama-Yoshino} and \ref{cor:contravariant-finiteness} completes Theorem~\ref{thm:A}.

\begin{corollary}
Let $\sX$ be a full additive subcategory of $\sT_0$ and $\tX$ be the corresponding set of arcs. Then $(\sX,\sX^{\perp})$ is a torsion pair in $\Tw$ if and only if any left fountain in $\tX$ is also a right fountain and $\tX$ is closed under taking admissible Ptolemy arcs of classes I, II and III.
\end{corollary}

\subsection{Combinatorial description of the Ext-hammocks}
The following propositions are direct consequences of Lemma \ref{partialfountainslemma}(2) and give us a combinatorial description of the arcs $\arc{b}$ for which there is an extension of $b$ by $a$. 

\begin{proposition}\label{lemma:Ext(b,a)w=0}
Let $a, b \in \ind{\sT_0}$, and assume $\arc{a}$ has length greater than or equal to one. We have $\Ext^1_{\sT_0} (b,a) \ne 0$ if and only if $\arc{b}$ satisfies one of the following conditions:
\begin{compactenum}
\item $\arc{b}$ crosses $\arc{a}$,
\item $\arc{b}$ is a neighbour of $\arc{a}$ incident with $t(\arc{a})-1$ or $s(\arc{a})-1$,
\item $s(\arc{b}) = t(\arc{a})$ and $t(\arc{b}) \leq t(\arc{a})-1$,
\item $t(\arc{b}) = t(\arc{a})$ and $s(\arc{b}) \geq s(\arc{a})-1$,
\item $s(\arc{b}) = s(\arc{a})$ and $t(\arc{a})-1 \leq t(\arc{b}) \leq s(\arc{a})-1$.
\end{compactenum} 
\end{proposition}

\begin{proposition}\label{lemma:Ext(b,a)w=0length0}
If $\arc{a}$ is a loop, and $b \in \ind{\sT_0}$, then $\Ext^1_{\sT_0} (b,a) \neq 0$ if and only if $\arc{b}$ satisfies condition (2) or (5) of Proposition \ref{lemma:Ext(b,a)w=0}.  
\end{proposition}

\subsection{The middle terms of extensions correspond to admissible Ptolemy arcs}

Recall the definitions of $\arc{e_1}$ and $\arc{e_2}$ from Definition~\ref{def:es}. The next proposition shows that, like the case $w \in \mathbb{Z}\setminus\{0,1\}$, these arcs, when admissible, correspond to the indecomposable summands of the middle term of the extension. 

\begin{proposition}\label{lemma:middletermsw=0}
Let $a, b \in \ind{\sT_0}$ be such that $\Ext^1_{\sT_0}(b,a) \neq 0$, and let $e_1$ and $e_2$ be as in Theorem \ref{thm:graphical-calculus}.
\begin{compactenum}
\item Suppose $b \in \extfrom{a}$. The arc $\arc{e_1}$ is always admissible. The arc $\arc{e_2}$ is admissible if and only if $s(\arc{b}) \geq t(\arc{a})$.
\item Suppose $b \in \extto{a}$. The arc $\arc{e_1}$ is admissible if and only if $s(\arc{b}) \geq s(\arc{a})$. The arc $\arc{e_2}$ is admissible if and only if $t(\arc{b}) \leq t(\arc{a})$. 
\end{compactenum}
Moreover, $e_i$ is nonzero if and only if $\arc{e_i}$, which is the corresponding arc, is admissible. 
\end{proposition}
\begin{proof}
Since $w = 0$, the arcs $\arc{e_1}$ and $\arc{e_2}$ are admissible if and only if their first component is greater than or equal to the second component. The proof of the last statement is the same as in Proposition \ref{prop:middletermsarcs}.
\end{proof}

Recall the definitions of $E(a,b)$ and $\tE(\arc{a},\arc{b})$ from Definition~\ref{def:E(a,b)}. We end this subsection by checking that $\tE (\arc{a},\arc{b})$ is contained in the set of Ptolemy arcs associated to $\arc{a}$ and $\arc{b}$. 

\begin{remark}\label{rmk:exceptionalcasesw0}
The only cases in Propositions \ref{lemma:Ext(b,a)w=0} and \ref{lemma:Ext(b,a)w=0length0} where $\arc{a}$ and $\arc{b}$ are neither crossing, neighbouring nor adjacent arcs are when:
\begin{compactenum}[(i)]
\item $\arc{a}$ is a loop and $\arc{b} = (t(\arc{a}),t(\arc{a})-1)$, or
\item $\arc{a} = (s(\arc{a}),s(\arc{a})-1)$ and $\arc{b} = (t(\arc{a}),t(\arc{a}))$.
\end{compactenum}
The extensions of $b$ by $a$ are $\trilabels{a}{a}{b}{f}{}{}$ in (i) and $\trilabels{a}{b}{b}{}{g}{}$ in (ii), where $f$ and $g$ are non-isomorphisms. On the other hand, $\Ext^1_{\sT_0} (a,b) = 0$. Therefore $\tE(\arc{a},\arc{b}) = \emptyset$ in these cases.
\end{remark}

\begin{corollary}\label{cor:middletermsarePtolemyw=0}
Let $a, b \in \ind{\sT_0}$ and $\arc{a},\arc{b}$ be the corresponding admissible arcs. Then
\[
\tE(\arc{a},\arc{b}) \subseteq \{\text{admissible Ptolemy arcs incident with the endpoints of $\arc{a}$ and $\arc{b}$}\}.
\]
\end{corollary}

\begin{proof}
The proof is similar to that of Corollary \ref{cor:extensions-vs-Ptolemy}, when $\arc{a}$ and $\arc{b}$ are crossing or neighbouring arcs. When $\arc{a}$ and $\arc{b}$ are adjacent, any arc connecting endpoints of $\arc{a}$ and $\arc{b}$ are Ptolemy arcs of class III, so in particular, $\arc{e_i}$ is a Ptolemy arc. The result then follows by Remark~\ref{rem:multiplesummands}; cf. Remark~\ref{rem:addE(a,b)}.
\end{proof}

\subsection{The extension closure}

Let $a, b \in \ind{\sT_0}$. Combinatorially, Ptolemy arcs of class I, II and III arise out of the following situations, respectively:
\begin{compactitem}
\item $\arc{a}$ and $\arc{b}$ are crossing arcs,
\item $\arc{a}$ and $\arc{b}$ are neighbouring arcs,
\item $\arc{a}$ and $\arc{b}$ are adjacent arcs.
\end{compactitem}

We shall show that in each of the three cases above $\tE(\arc{a},\arc{b})$ is precisely the set of all the admissible Ptolemy arcs of the appropriate class associated to $\arc{a}$ and $\arc{b}$. We consider each situation in turn.

\begin{proposition}
If $\arc{a}$ and $\arc{b}$ cross each other then $\Ext^1_{\sT_0} (b,a) \ne 0$, $\Ext^1_{\sT_0}(a,b) \ne 0$ and $\tE(\arc{a},\arc{b})$ is the set of the Ptolemy arcs of class I associated to $\arc{a}$ and $\arc{b}$. 
\end{proposition}
\begin{proof} 
The fact that there are extensions in both directions is an immediate consequence of Proposition \ref{lemma:Ext(b,a)w=0}(1). 

We can assume that $\arc{b}$ crosses $\arc{a}$ to the right, as the other case is dual. We must show that the four Ptolemy arcs associated to $\arc{a}$ and $\arc{b}$ lie in $\tE(\arc{a},\arc{b})$. 

On one hand, $b \in \extto{a}$ and so the admissible Ptolemy arcs $(s(\arc{b}),s(\arc{a}))$ and $(t(\arc{b}),t(\arc{a}))$ lie in $\tE (\arc{a}, \arc{b})$, by Proposition \ref{lemma:middletermsw=0}(2). On the other hand, $a \in \extfrom{b}$, and so the other two admissible Ptolemy arcs, namely $(s(\arc{b}),t(\arc{a}))$ and $(s(\arc{a}),t(\arc{b}))$, also lie in $\tE(\arc{a},\arc{b})$, by Proposition \ref{lemma:middletermsw=0}(1). 
\end{proof}

\begin{proposition}
If $\arc{a}$ and $\arc{b}$ are neighbouring arcs, then $\Ext^1_{\sT_0}(b,a) \neq 0$ or $\Ext^1_{\sT_0}(a,b) \neq 0$, and $\tE (\arc{a},\arc{b})$ is the set of Ptolemy arcs of class II associated to $\arc{a}$ and $\arc{b}$. 
\end{proposition}
\begin{proof}
The proof is similar to that of Proposition \ref{prop:neighbourotherinclusion}. The only difference is that $b$ can be $\Sigma a$ or $\Sigma^{-1} a$, namely when $\arc{a}$ is a loop. In these cases, the extension has dimension two, and the middle term of the extension whose middle term is nonzero corresponds to the Ptolemy arc of associated to $\arc{a}$ and $\arc{b}$. 
\end{proof}

\begin{proposition}
If $\arc{a}$ and $\arc{b}$ are adjacent arcs, then $\Ext^1_{\sT_0} (b,a) \ne 0$ or $\Ext^1_{\sT_0} (a,b) \ne 0$, and $\tE (\arc{a}, \arc{b})$ contains the set of Ptolemy arcs of class III associated to $\arc{a}$ and $\arc{b}$. 
\end{proposition}
\begin{proof}
Let us fix an arc $\arc{a}$ and check the possibilities for $\arc{b}$.

\Case{1} $\arc{b}$ and $\arc{a}$ are adjacent at $t(\arc{a})$ and $t(\arc{b}) < t(\arc{a})$. We have $\Ext^1_{\sT_0}(b,a) \neq 0$ and $b \in \extfrom{a}$. So, by Proposition \ref{lemma:middletermsw=0}(1), the two Ptolemy arcs associated to $\arc{a}$ and $\arc{b}$ lie in $\tE (\arc{a},\arc{b})$. 

\Case{2} $\arc{b}$ and $\arc{a}$ are adjacent at $t(\arc{a})$ and $s(\arc{b}) \geq s(\arc{a})$. Then $\Ext^1_{\sT_0}(b,a) \neq 0$ and $b \in \extto{a}$. So, by Proposition \ref{lemma:middletermsw=0}(2), the middle term of the extension has two indecomposable summands, which correspond to the Ptolemy arcs associated to $\arc{a}$ and $\arc{b}$.

\Case{3} $\arc{b}$ and $\arc{a}$ are adjacent at $s(\arc{a})$ and $t(\arc{a}) \leq t(\arc{b}) < s(\arc{a})$. In this case we also have $\Ext^1_{\sT_0}(b,a) \neq 0$ and $b \in \extto{a}$. By Proposition \ref{lemma:middletermsw=0}(2), the two Ptolemy arcs of class III associated to $\arc{a}$ and $\arc{b}$ lie in $\tE (\arc{a}, \arc{b})$.

The remaining three cases are dual. 
\end{proof}

Putting these together with Corollary~\ref{cor:middletermsarePtolemyw=0} yields the following corollary, which in turn, together with Theorem \ref{thm:extensionsnonindecomposable}, gives Theorem~\ref{thm:extensionclosureT0}.

\begin{corollary} \label{cor:Ptolemy-equality0}
Let $a,b\in \ind{\sT_0}$ and $\arc{a},\arc{b}$ be the corresponding admissible arcs. Then
\[
\tE(\arc{a},\arc{b}) = \{\text{admissible Ptolemy arcs incident with the endpoints of $\arc{a}$ and $\arc{b}$}\}.
\]
\end{corollary}

%=======================================================================
% Section
\section{Torsion pairs and extensions in $\sT_1$}\label{sec:extensions-in-T1}
%=======================================================================

There is no combinatorial model of $\sT_1$ in terms of admissible arcs of the infinity-gon, and thus no characterisation of torsion pairs in terms of Ptolemy diagrams. However, the classification of torsion pairs in $\sT_1$ is quite simple, see Theorem~\ref{thm:classification-T1} below. 

Recall that the AR quiver of $\sT_1$ consists of $\bZ$ copies of the homogeneous tube below:
\[
\xymatrix{
X_0 \ar@/^/[r] & X_1 \ar@/^/[r] \ar@/^/[l] & X_2 \ar@/^/[r] \ar@/^/[l] & X_3 \ar@/^/[r] \ar@/^/[l] & \cdots . \ar@/^/[l]
}
\]
Let $\cT_0$ be the additive (even abelian) category generated by this tube and $\cT_n \coloneqq \Sigma^n \cT_0$. 

There is a $\bZ$-indexed family of split t-structures in $\sT_1$, $(\sX_n,\sY_n)$ given by 
\[
\sX_n \coloneqq \add{\bigcup_{i\geq n} \cT_i} 
\quad \text{and} \quad
\sY_n \coloneqq \add{\bigcup_{i < n} \cT_i},
\]
whose heart is $\cT_n$, which is a hereditary abelian category. The short exact sequences in $\cT_n$ correspond precisely to distinguished triangles $\tri{a'}{a}{a''}$ with $a',a,a'' \in \cT_n$; see \cite{BBD}.

\subsection{Torsion pairs in $\sT_1$}

The main result of this section is the following.

\begin{theorem} \label{thm:classification-T1}
The only torsion pairs in $\sT_1$ are the (de)suspensions of the standard t-structure, i.e. $(\sX_n,\sY_n)$ for $n\in \bZ$, and the trivial torsion pairs $(\sT_1,0)$ and $(0,\sT_1)$.
\end{theorem}

Before proving Theorem~\ref{thm:classification-T1} we need two preliminary results. The first one describes the Hom- and Ext-hammocks of $\sT_1$.

\begin{proposition}[{\cite[Proposition 3.4]{HJY}}] \label{prop:hom-hammocks-in-T1}
Consider the indecomposable object $X_r$ in $\sT_1$ and let $b \in \ind{\sT_1}$ be any other indecomposable. Then
\[
\ext_{\sT_1}(b,X_r) = \hom_{\sT_1}(X_r,b) =
\left\{
\begin{array}{ll}
\min\{r,s\}+ 1 & \text{if } b= X_s \text{ or } b=\Sigma X_s, \\
0              & \text{otherwise.}
\end{array}
\right.
\]
\end{proposition}

Note that the first equality above follows from the $1$-Calabi-Yau property.

\begin{lemma} \label{lem:split}
All torsion pairs in $\sT_1$ are split.
\end{lemma}

\begin{proof}
Suppose $(\sX,\sY)$ is a non-split torsion pair in $\sT_1$ and that $t\in \ind{\sT_1}$ does not belong to either $\sX$ or $\sY$. Therefore, there is a non-trivial approximation triangle $\tri{x}{t}{y}$ with $x\in \sX$ and $y\in \sY$. The object $t$ lies in some homogeneous tube, $\cT_k$ say, and thus $t= \Sigma^k X_u$ for some $u\geq 0$. By Proposition~\ref{prop:hom-hammocks-in-T1} there are maps to $t$ only from the tubes $\cT_{k-1}$ and $\cT_k$. Suppose $x$ contains a summand $\Sigma^{k-1} X_r \in \cT_{k-1}$. By the properties of homogeneous tubes (see for example \cite[Chapter X]{SS} and combine it with the properties of derived categories of hereditary categories \cite{Happel}), the maps $\Sigma^{k-1}X_r \to t$ factor through $\Sigma^{k-1} X_s$ for all $s > r$. Hence, no finite sum of indecomposable objects from $\cT_{k-1}$ can be a summand of $x$. The only possibility remaining is that $x$ contains a summand from $\cT_k$. But since $\sX$ is extension-closed, we get $\cT_k \subseteq \sX$, whence $t\in \sX$; a contradiction. Thus, any torsion pair $(\sX,\sY)$ is split.
\end{proof}

\begin{proof}[Proof of Theorem~\ref{thm:classification-T1}]
Let $(\sX,\sY)$ be a torsion pair in $\sT_1$ and suppose $\cT_k \subseteq \sX$. By Proposition~\ref{prop:hom-hammocks-in-T1}, $\Hom_{\sT_1}(\cT_k,\cT_{k+1}) \neq 0$, whence $\cT_{k+1} \subseteq \sX$ since $(\sX,\sY)$ is split by Lemma~\ref{lem:split}. Thus, if $\cT_k \subseteq \sX$ then $\cT_j \subseteq \sX$ for each $j\geq k$. There is either a minimal such $k$, in which case $(\sX, \sY)=(\sX_k,\sY_k)$, or there is not, in which case $(\sX,\sY) = (\sT_1,0)$. Arguing from the point of view of $\sY$ gives the other trivial torsion pair $(0,\sT_1)$.
\end{proof}

\begin{remark}
In \cite{HJY} it was shown that $\sT_1$ has only one family of non-trivial (bounded) t-structures, namely the $(\sX_n,\sY_n)$. Here we have shown that this family of bounded t-structures are the only non-trivial torsion pairs in $\sT_1$.
\end{remark}

\subsection{Extensions with indecomposable outer terms in $\sT_1$}

Whilst this is not needed in the classification of torsion pairs in $\sT_1$, for the sake of completeness, we include a brief description.

Without loss of generality we may consider extensions starting at $X_r$ for some $r\geq 0$. Using Proposition~\ref{prop:hom-hammocks-in-T1} and the $1$-Calabi-Yau property, the extensions whose outer terms are indecomposable and first term is $X_r$ have the following form for $s \geq 0$:
\[
\tri{X_r}{E}{X_s} \ \text{and} \ \tri{X_r}{F}{\Sigma X_s}.
\]

 The first extension $\tri{X_r}{E}{X_s}$ has all three objects lying in the heart $\cT_0$. Thus, it is enough to compute the extensions in the abelian category $\cT_0$, which is a special case of \cite[Lemma 5.1]{Baur-Buan-Marsh}.

The middle term of the second extension is isomorphic to $\cone{f} = (\Sigma \ker f)\oplus \coker f$, where $f \colon X_s \to X_r$, where we use the fact that $\cT_0$ is hereditary. We can now 
compute the cones of these morphisms in $\sT_1$ using, for example, \cite[Chapter X]{SS}.

We summarise these considerations below. 

\begin{proposition} \label{prop:graphical-calculus-w=0}
Consider a non-trivial extension $\tri{X_r}{E}{B}$ in $\sT_1$, where $B$ is either $X_s$ or $\Sigma X_s$ for some $s \geq 0$. We interpret $X_{-1}$ as the zero object.
\begin{compactenum}[(i)]
\item If $B=X_s$, write $n=\min\{r,s\}$ and $m=\max\{r,s\}$. Then the $n+1$ extensions are $\tri{X_r}{X_{m+i} \oplus X_{n-i}}{X_s}$ for $1 \leq i \leq n+1$.
\item If $B=\Sigma X_s$ for $s \geq r$ then the $r+1$ extensions are $\tri{X_r}{\Sigma X_{s-r-i} \oplus X_{i-2}}{\Sigma X_s}$ for $1 \leq i \leq r+1$.
\item If $B=\Sigma X_s$ for $s < r$ then the $s+1$ extensions are $\tri{X_r}{X_{r-s-i} \oplus \Sigma X_{i-2}}{\Sigma X_s}$ for $1 \leq i \leq s+1$.
\end{compactenum} 
\end{proposition}

%==============================================================================
%  Section
\section{Torsion pairs in $\sC_w(A_n)$} \label{sec:orbit}
%==============================================================================

Throughout this section $w \leq -1$ and $m=-w+1$ and we shall consider the orbit category 
\[
\Cm \coloneqq \Cm(A_n) = \Db(\kk A_n)/\Sigma^{m}\tau,
\]
where $\tau$ denotes the AR translate of $\Db(\kk A_n)$. For more detailed background on these categories we refer the reader to the papers \cite{CS10,CS11,CS13}. These categories are triangulated by Keller's Theorem \cite{Keller}, and satisfy $\Sigma^w \simeq \SSS$, where $\SSS = \Sigma \tau$ is the Serre functor. In particular, they can be considered to be $w$-CY.

\subsection{The combinatorial model for $\Cm$}

Given two indecomposable objects $a$ and $b$ in $\Tw$, we say that $\arc{a}$ is an \emph{innerarc} of $\arc{b}$ if one has $t(\arc{b}) < t(\arc{a}) < s(\arc{a}) < s(\arc{b})$.

It was shown in \cite[Theorem 5.1, Corollary 6.3]{CS13} that the combinatorial model in $\Tw$ induces a combinatorial model in $\Cm$ as follows:  $\Cm$ is equivalent to the full subcategory $\cC_w$ of $\Tw$ whose set of indecomposable objects correspond to the admissible innerarcs of an admissible arc $\arc{a}$ of length $|(n+1)(-w-1)+1|$. Note that this equivalence is not a triangle equivalence.

We briefly recall the explicit description of the induced combinatorial model. Let $\mcP_{n,m}$ be the regular $N$-gon, where $N = m(n+1)-2$, with vertices numbered clockwise from $1$ to $N$. All operations on vertices of $\mcP_{n,m}$ will be done modulo $N$, with representatives $1, \ldots, N$. An \emph{$m$-diagonal} of $\mcP_{n,m}$ is a diagonal that divides $\mcP_{n,m}$ into two polygons each of whose number of vertices is divisible by $m$.

The AR quiver of $\Cm$ is equivalent to the stable translation quiver $\Gamma (n,m)$ whose vertices are the $m$-diagonals $\mcP_{n,m}$. We denote a vertex of $\Gamma(n,m)$ by $\{i,j\}$, where $i$ and $j$ are vertices of $\mcP_{n,m}$. The arrows of $\Gamma (n,m)$ are obtained in the following way: given two $m$-diagonals $D$ and $D'$ with a vertex $i$ in common, there is an arrow from $D$ to $D'$ in $\Gamma (n,m)$ if and only if $D'$ can be obtained from $D$ by rotating clockwise $m$ steps around $i$.
The translation automorphism $\tau: \Gamma (n,m) \to \Gamma (n,m)$ sends an $m$-diagonal $\{i,j\}$ to $\tau (\{i,j\}):= \{i-m, j-m\}$. Note that $\Sigma \{i,j\} = \{i+1,j+1\}$. 
Figure \ref{fig:orbitcatexample} shows an example of this stable translation quiver.

\begin{figure}[!ht]
\begin{center}
\begin{tikzpicture}[shorten >=1pt,node distance=1.5cm,thick,scale=0.75, every node/.style={scale=0.75}]

\node[fill=black!20!white,draw=none,regular polygon, regular polygon sides=4, minimum size=1.2cm] (1) {};
\node at (1) {$\smxy{\{1,3\}}$};

\node[fill=black!20!white,draw=none,regular polygon, regular polygon sides=4,minimum size=1.2cm] (2) [above right of=1]{};
\node at (2) {$\smxy{\{1,6\}}$};

\node[fill=black!20!white,draw=none,regular polygon, regular polygon sides=4,minimum size=1.2cm] (3) [above right of=2]{};
\node at (3) {$\smxy{\{1,9\}}$};

\node[fill=black!20!white,draw=none,regular polygon, regular polygon sides=4,minimum size=1.2cm] (4) [below right of=2]{};
\node at (4) {$\smxy{\{4,6\}}$};

\node[fill=black!20!white,draw=none,regular polygon, regular polygon sides=4,minimum size=1.2cm] (5) [above right of=4]{};
\node at (5) {$\smxy{\{4,9\}}$};

\node[fill=black!20!white,draw=none,regular polygon, regular polygon sides=4,minimum size=1.2cm] (6) [below right of=5]{};
\node at (6) {$\smxy{\{7,9\}}$};

\node[fill=black!20!white,draw=none,diamond,minimum size=1.2cm] (7) [above right of=5]{};
\node at (7) {$\smxy{\{2,4\}}$};

\node[fill=black!20!white,draw=none,diamond,minimum size=1.2cm] (8) [below right of=7]{};
\node at (8) {$\smxy{\{2,7\}}$};

\node[fill=black!20!white,draw=none,diamond,minimum size=1.2cm] (9) [below right of=8]{};
\node at (9) {$\smxy{\{2,10\}}$};

\node[fill=black!20!white,draw=none,diamond,minimum size=1.2cm] (10) [above right of=8]{};
\node at (10) {$\smxy{\{5,7\}}$};

\node[fill=black!20!white,draw=none,diamond,minimum size=1.2cm] (11) [below right of=10]{};
\node at (11) {$\smxy{\{5,10\}}$};

\node[fill=black!20!white,draw=none,circle,minimum size=1cm] (12) [below right of=11]{};
\node at (12) {$\smxy{\{3,5\}}$};

\node[fill=black!20!white,draw=none,diamond,minimum size=1.2cm] (13) [above right of=11]{};
\node at (13) {$\smxy{\{8,10\}}$};

\node[fill=black!20!white,draw=none,circle,minimum size=1cm] (14) [below right of=13]{};
\node at (14) {$\smxy{\{3,8\}}$};

\node[fill=black!20!white,draw=none,circle,minimum size=1cm] (15) [below right of=14]{};
\node at (15) {$\smxy{\{6,8\}}$};

\node[fill=black!20!white,draw=none,regular polygon, regular polygon sides=4 ,minimum size=1.2cm] (16) [above right of=14]{};
\node at (16) {$\smxy{\{1,3\}}$};

\node[fill=black!20!white,draw=none,regular polygon, regular polygon sides=4 ,minimum size=1.2cm] (17) [below right of=16]{};
\node at (17) {$\smxy{\{1,6\}}$};

\node[fill=black!20!white,draw=none,regular polygon, regular polygon sides=4 ,minimum size=1.2cm] (18) [below right of=17]{};
\node at (18) {$\smxy{\{1,9\}}$};

\path[->] (1) edge (2)
          (2) edge (3)
              edge (4)
          (3) edge (5)
          (4) edge (5)
          (5) edge (6)
              edge (7)
          (6) edge (8)
          (7) edge (8)
          (8) edge (9)
              edge (10)
          (9) edge (11)
          (10) edge (11)
          (11) edge (12)
               edge (13)
          (12) edge (14)
          (13) edge (14)
          (14) edge (15)
               edge (16)
          (15) edge (17) 
          (16) edge (17) 
          (17) edge (18);
\end{tikzpicture}
\end{center}
\caption{The AR quiver of $\sC_2(A_3)$ is equivalent to $\Gamma (3,2)$.}
\label{fig:orbitcatexample}
\end{figure}
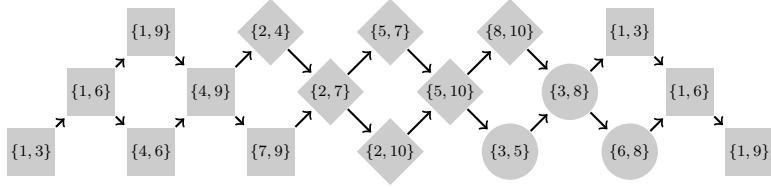

Define Ptolemy diagonals of class I, neighbouring $m$-diagonals and Ptolemy diagonals of class II in the same manner as in $\Tw$. The main result of this section is:

\begin{theorem} \label{thm:orbit-extension-closure}
Let $\sX$ be a full additive subcategory of $\Cm$ and $\tX$ the set of $m$-diagonals corresponding to the objects of $\ind{\sX}$. Then the objects of $\ind{\extn{\sX}}$ correspond to the $m$-diagonals of the closure of $\tX$ under Ptolemy $m$-diagonals of classes I and II.
\end{theorem}

Since $\Cm$ has finitely many indecomposable objects up to isomorphism, any subcategory of $\Cm$ is contravariantly finite. Thus, Theorem~\ref{thm:B} is an immediate corollary of Theorem~\ref{thm:orbit-extension-closure} and Proposition~\ref{prop:Iyama-Yoshino}.

\begin{corollary}
Let $\sX$ be a full additive subcategory of $\Cm$ and $\tX$ the corresponding set of $m$-diagonals. Then $(\sX,\sX^{\perp})$ is a torsion pair in $\Cm$ if and only if $\tX$ is closed under Ptolemy $m$-diagonals of classes I and II. 
\end{corollary}

\begin{remark}\label{rmk:torsionpairsCvsT}
The characterisation of torsion pairs in $\Cm$ does not follow immediately from that in $\Tw$: $\cC_w$ is not a triangulated subcategory of $\Tw$; see \cite[Theorem 5.1]{CS13}. 
In fact, if there is an extension in $\Tw$ between two objects of $\cC_w$ then there is an extension in $\Cm$ between their images, but the converse is not true.
As an example, consider $m = 2, n= 3$ and $\arc{a} = (7,0)$ to be the arc that defines the equivalence. The admissible arcs $(2,1)$ and $(6,5)$ are objects in $\cC_w$ and the corresponding images $\{1,2\}$ and $\{5,6\}$ are $2$-diagonals of a hexagon. It is easy to check that $\Ext^1_{\sT_{-1}} ((1,2),(5,6)) = 0$ but $\Ext^1_{\sC_{2}} (\{1,2\},\{5,6\}) \neq 0$.   
\end{remark}

\subsection{Combinatorial description of the Ext-hammocks}
We require the following notation to describe the Hom- and Ext-hammocks in $\Cm$.

\begin{notation}
Given the vertices $i_1, i_2, \ldots, i_k$ of $\mcP_{n,m}$, we write $C(i_1, i_2, \ldots, i_k)$ to mean that $i_1, i_2, \ldots, i_k, i_1$ follow each other under the clockwise circular order on the boundary of $\mcP_{n,m}$.
\end{notation}

\begin{lemma}
Let $a, b \in \ind \Cm$ and $\arc{a} = \{a_1, a_2\}$, with $a_1 < a_2$, $\arc{b} = \{b_1, b_2\}$ be the corresponding $m$-diagonals of $\mcP_{n,m}$. 
We have $\Hom_{\Cm} (a, b) \ne 0$ if and only if $\arc{b}$ satisfies the following condition:
\[
b_1 = a_2 +im, b_2 = a_1 + jm, \text{ for some } i, j \geq 0, \text{ and } C(a_2,b_1,a_1,b_2).
\]
\end{lemma}
\begin{proof}
Explicit computation using the combinatorial model of $\Cm$.
\end{proof}

Using the Auslander--Reiten formula, we obtain the Ext-hammocks:

\begin{corollary}\label{corollary:exthammock}
We have $\Ext^1_{\Cm} (b,a) \ne 0$ if and only if the following condition holds:
\[
b_1 = a_2 + im, b_2 = a_1 + jm, \text{ for some } i, j \geq 1, \text{ and } C(a_2+m, b_1,a_1+m,b_2).
\]
\end{corollary}

Corollary \ref{corollary:exthammock} can be unpacked into the following more readable statement.

\begin{corollary}\label{cor:exthammockconditions}
Let $a, b \in \ind \Cm$ and $\arc{a} = \{a_1, a_2\}$, with $a_1 < a_2$, be the  $m$-diagonal corresponding to $a$. Then $\Ext^1_{\Cm} (b,a) \ne 0$ if and only if $\arc{b}$ satisfies one of the following:
\begin{compactenum}
\item $\arc{b}$ is a neighbour of $\arc{a}$ incident with $a_2 +1$, 
\item $\arc{b}$ is a neighbour of $\arc{a}$ incident with $a_1 +1$,
\item $\arc{b}$ crosses $\arc{a}$ in such a way that $\arc{b}$ is obtained by adding multiples of $m$ to the endpoints of $\arc{a}$,
\item $\arc{b} = \{a_1+1,a_2+1\}$, i.e. $b = \Sigma a$. 
\end{compactenum} 
\end{corollary}

Figure \ref{fig:ext-shaded} shows where the indecomposable objects corresponding to the arcs that satisfy the conditions in Corollary \ref{cor:exthammockconditions} lie in the AR quiver.

\begin{figure}[!ht]
\begin{center}
\begin{tikzpicture}[thick,scale=0.75, every node/.style={scale=0.75}]

%a:
\fill (-10.5,0.5) circle (2pt);
\node (a) at (-10.5,0.5) {};
\node [left] at (a) {$a$};

%shaded areas
\filldraw[fill=black!20!white] (-10,0.5) -- (-8.5,2) -- (-6,-0.5) -- (-7.5,-2) -- (-10,0.5); %case 3

\filldraw[fill=black!30!white] (-8.5,2) -- (-8,2.5) -- (-5.5,0) -- (-6,-0.5) -- (-8.5,2); %case 2

\filldraw[fill=black!40!white] (-7.5,-2) -- (-6,-0.5) -- (-5.5,-1) -- (-7,-2.5) -- (-7.5,-2); %case 1

\filldraw[fill=black!10!white] (-6,-0.5) -- (-5.5,0) -- (-5,-0.5) -- (-5.5,-1) -- (-6,-0.5); %case 4

%tauinverse of a
\fill[black!50!white] (-9.5,0.5) circle (2pt);
\node (tauinv-a) at (-9.5,0.5) {};
\node [right] at (tauinv-a) {$\tau^{-1} a$};

%hom-hammock of tau-inverse of a:
\draw[dotted] (-9.5,0.5) -- (-8,2) -- (-5.5,-0.5) -- (-7,-2) -- (-9.5,0.5);

%shaded areas notation:
\node (2) at (-6.5,1.5) {};
\node at (2) {$(2)$};
\node (3) at (-9,-1) {};
\node at (3) {$(3)$}; 
\node (4) at (-5,-0.5) {};
\node[right] at (4) {$(4)$};
\node (1) at (-5.75,-2){};
\node[above] at (1) {$(1)$};

%top and bottom lines of the AR-quiver:
\draw[thick] (-11,-2) -- (-7.5,-2);
\draw[dotted] (-7.5,-2) -- (-6.5,-2);
\draw[thick] (-6.5,-2) -- (-5,-2);

\draw[thick] (-11,2) -- (-8.5,2);
\draw[dotted] (-8.5,2) -- (-7.5,2);
\draw[thick] (-7.5,2) -- (-5,2);
\end{tikzpicture}
\end{center}
\caption{$\Ext^1_{\Cm}(-,a) \neq 0$.}
\label{fig:ext-shaded}
\end{figure}
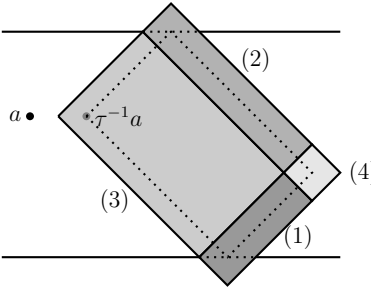

\subsection{Graphical calculus}
To prove Theorem~\ref{thm:orbit-extension-closure} we need a graphical calculus analogous to Theorem~\ref{thm:graphical-calculus}. Before doing this, we briefly recall some useful notation and facts about $\Cm$.

Write $\inj{\kk A_n}$ for the subcategory of $\mod{\kk A_n}$ containing the injective $\kk A_n$-modules. We have the following fundamental domain for $\Cm$ in $\Db(\kk A_n)$:
\[
\sF = \ind{\bigcup_{i=1}^{m} \Sigma^{i-1} \mod{\kk A_n} \cup \Sigma^m (\mod{\kk A_n} \setminus \inj{\kk A_n})}, 
\]
From now on, we identify objects in $\ind{\Cm}$ with their representatives in $\sF$. Given $X \in \ind{\Db(\kk A_n)}$, we denote by $\dd{X}$ the degree of $X$, i.e. the integer such that $X = \Sigma^{\dd{X}} \overline{X}$ for  $\overline{X}\in \ind{\mod{\kk A_n}}$. 

\begin{lemma} Let $A, B \in \sF$, $P$ be a projective $\kk A_n$-module and $C \in \ind{\Db (\kk A_n)}$.
\begin{compactenum}
\item $\Hom_{\Db(\kk A_n)} (A, \tau^k \Sigma^{km} B) = 0$, for every $k \neq 0, 1$.
\item $\Hom_{\Db(\kk A_n)} (A, \tau^k \Sigma^{km} B) \neq  0$ for at most one value of $k$.
\item If $\Hom_{\Db(\kk A_n)} (P,C) \neq 0$, then $C \in \ind{\mod{\kk A_n}}$.
\end{compactenum}
\end{lemma}
\begin{proof}
Statements (1) and (2) are proved for $m=1$ in \cite[Proposition 2.1]{CS10}; the proof for $m \geq 2$ is similar. Statement (3) is trivial using projectivity and the hereditary property.
\end{proof}

For $a \in \ind{\Cm}$ the \emph{starting} and \emph{ending frames} of $a$ are:
\begin{align*}
F_s(a) & \coloneqq \{b \in \ind \Cm \mid \Hom_{\Cm} (a,b) \neq 0, \Ext^1_{\Cm} (b,a) = 0 \}; \\
F_e (a) & \coloneqq \{b \in \ind \Cm \mid \Hom_{\Cm} (b,a) \neq 0, \Ext^1_{\Cm} (a,b) = 0 \},
\end{align*}
cf. Lemma~\ref{lem:homext-vanishing}(ii), $\rayfrom{a} \cup \corayfrom{a}$ and $\rayto{a} \cup \corayto{a}$ in Theorem~\ref{thm:graphical-calculus}. 

We are now ready to state the graphical calculus result. For $w=-1$ this was done in \cite[Chapter 5]{CS-thesis}.

\begin{proposition} \label{prop:graphicalcalculusorbitcat}
Let $a, b \in \ind{\Cm}$ be such that $\Ext^1_{\Cm}(b,a) \ne 0$. Then $\ext^1_{\Cm} (b,a) = 1$ and the unique (up to equivalence) non-split extension $\tri{a}{e}{b}$ has middle term $e$ whose summands are given by 
$F_s(a) \cap F_e (b)$. If this intersection is empty, then we interpret $e$ to be the zero object.
\end{proposition}

\begin{proof}
We can assume, without loss of generality, that the quiver of type $A_n$ has the linear orientation $n \too n-1 \too \cdots \too 1$. 
Recall that we see $a$ and $b$ as sitting inside the fundamental domain $\sF$. Suppose $\Ext^1_{\Cm} (b,a) \neq 0$. By taking a suitable AR translate of $a$ and $b$, we may assume that $a$ is a projective $\kk A_n$-module.

\Case{1} Assume $b$ is a non-projective $\kk A_n$-module. 
Then $\tau b \in \mod{\kk A_n} \subseteq \sF$ and so 
\[
\Ext^1_{\Cm} (b,a) \simeq D \Hom_{\Cm} (a, \tau b) = D \Hom_{\Db(\kk A_n)} (a, \tau b) \oplus D \Hom_{\Db(\kk A_n)} (a, \tau^2 \Sigma^m b).
\] 
Since $\tau b$ is a module, $\dd{\tau^2 \Sigma^m b} = m-1$ or $m$, so $\dd{\tau^2 \Sigma^m b} \geq 1$. Hence, the second summand is zero, since $a$ is a projective module. Therefore $\Ext^1_{\Cm} (b,a) \simeq \Hom_{\kk A_n} (a, \tau b)$. Hence, $\ext^1_{\Cm} (b,a) = 1$, since Hom spaces in type $A$ are either zero or one dimensional. The result then follows from \cite[Corollary 8.5]{BMRRT}.

\Case{2} Assume $b$ is a projective indecomposable module.
Then $\tau b = \Sigma^{-1} I$ for some indecomposable injective $I$.
In $\Cm$ we have $\tau b \simeq \Sigma^{m-1} \tau I$, which lies in $\sF$. Hence,
\[
\Ext^1_{\Cm} (b,a) \simeq D \Hom_{\Cm} (a, \tau b) = D \Hom_{\Db(\kk A_n)} (a, \Sigma^m \tau I) \oplus D \Hom_{\Db(\kk A_n)} (a, \tau^2 \Sigma^{2m-1} I).
\] 
The second summand is always zero since $\dd{\tau^2 \Sigma^{2m-1} I} \geq 1$ and  $a$ is projective. We have $\dd{\tau \Sigma^{m-1} I} = m-2$ or $m-1$. For any $m \geq 2$, we have $m-1 \geq 1$. Thus, if $\dd{\tau \Sigma^{m-1} I} = m-1$, the first summand is also zero giving a contradiction. Analogously if $\dd{\tau \Sigma^{m-1} I} = m-2$ and $m > 2$. 
To avoid a contradiction, we must have  $m=2$ and $\dd{\tau \Sigma I} = 0$ and $I \cong P(n) \cong I(1)$ the indecomposable projective-injective, whence $a \cong P(1)$. Now, 
\[
0 \neq \Ext^1_{\sC_{2}} (b,a) \simeq D \Hom_{\Db(A_n)} (a, \Sigma \tau P(n)) = D \Hom_{\kk A_n} (a,I(n)),
\]
where one-dimensionality follows from being in type $A$.
Since $I(n)=S(n)$, i.e. the simple at vertex $n$, and $a$ is projective, we get $a = P(n)$. In $\sC_{2}$ we have $\Sigma P(n) \cong P(1)$, whence the non-split triangle is $\tri{P(n)}{0}{P(1)}$. However, $F_s(a) \cap F_e(b) = \emptyset$, giving the claim in this case .

\Case{3} $b = \Sigma^i \overline{B}$, for some indecomposable non-injective module $\overline{B}$ and $0 < i < m$. If $i = 1$, assume also that $\overline{B}$ is non-projective. Then $\tau b$ lies in $\sF$ and $\dd{\tau b} \geq 1$. Therefore,
\[
\Ext^1_{\Cm} (b,a) \simeq D \Hom_{\Cm} (a, \tau b) = D \Hom_{\Db(\kk A_n)} (a,\tau b) \oplus D \Hom_{\Db(\kk A_n)} (a, \Sigma^m \tau^2 b).
\]
Since $\dd{\Sigma^m \tau^2 b} \geq 2$ and $\dd{\tau b} \geq 1$, both summands are zero; a contradiction, so this case provides no extensions.

\Case{4} $b = \Sigma \overline{B}$, where $\overline{B}$ is projective.
Then $\tau b$ is an injective module, and thus lies in $\sF$. Hence
\[
\Ext^1_{\Cm} (b,a) \simeq D \Hom_{\Cm} (a, \tau b) = \Hom_{\Db(\kk A_n)} (a, \tau b) \oplus D \Hom_{\Db(\kk A_n)} (a, \Sigma^m \tau^2 b).
\]
The second summand is zero because $\dd{\Sigma^m \tau^2 b} \geq 1$. Therefore $\ext^1_{\Cm} (b,a)=1$, since $\Ext^1_{\Cm}(b,a) \neq 0$ by assumption. We construct a non-split triangle with first term $a$ and last term $b$. Let $a = P(i)$ and $b = \Sigma P(j)$.

For $i < j$ we have $\Hom_{\Db(\kk A_n)} (a, \tau b)  = \Hom_{\Db(\kk A_n)} (P(i),I(j))=0$, giving a contradiction. Thus $i \geq j$.
For $i > j$, there is a short exact sequence $0 \to P(j) \to P(i) \to E \to 0$, which induces a triangle $\tri{P(j)}{P(i)}{E}$. Shifting this triangle gives the desired triangle. It is easy to check that the cokernel $E$ is the unique object in $F_{s}(P(i))\cap F_e(\Sigma P(j))$.
For $i = j$, we get the standard triangle $\tri{P(i)}{0}{\Sigma P(i)}$ in which the last map is an isomorphism. Moreover, $F_s(P(i)) \cap F_e(\Sigma P(i)) = \emptyset$, corresponding to the zero middle term.
\end{proof}

Using Proposition~\ref{prop:graphicalcalculusorbitcat}, one can check that the $m$-diagonals corresponding to the middle term are as in Figure \ref{fig:middle-term}.

\begin{figure}
\begin{center}
\begin{tikzpicture}[thick,scale=0.6, every node/.style={scale=0.6}]

%%%%%%%%%%%%%%case 1%%%%%%%%%%%%%%%%%%%%%%%%%%%%%%%%%%%%%%%%%%%%%%%%
%a:
\fill (-10.5,0.5) circle (2pt);
\node (a) at (-10.5,0.5) {};
\node [left] at (a) {$a$};

\filldraw[fill=black!40!white] (-7.5,-2) -- (-6,-0.5) -- (-5.5,-1) -- (-7,-2.5) -- (-7.5,-2);

%b:
\fill (-6.5,-1.5) circle (2pt);
\node (b) at (-6.5,-1.5) {};
\node [right] at (b) {$b$};

%hom-hammock of tau-inverse of a:
\draw[dotted] (-9.5,0.5) -- (-8,2) -- (-5.5,-0.5) -- (-7,-2) -- (-9.5,0.5);

%section of a:
\draw (-10.5,0.5) -- (-8,-2);

\draw (-10.5,0.5) -- (-9.5,1.5);
\draw[dashed] (-9.5,1.5) -- (-9,2);

%section of b:
\draw (-7,-2) -- (-6.5,-1.5);

\draw (-9.5,1.5) -- (-6.5,-1.5);
\draw[dashed] (-10,2) -- (-9.5,1.5);

%middle terms: 
\node (e) at (-9.5,1.5) [circle,draw] {};
\node [below] at (e.south) {$e$};

%top and bottom lines:
\draw[thick] (-11,-2) -- (-7.5,-2);
\draw[dotted] (-7.5,-2) -- (-6.5,-2);
\draw[thick] (-6.5,-2) -- (-5,-2);

\draw[thick] (-11,2) -- (-5,2);

%%%%%%%%%%%%%%%%%%end of case 1%%%%%%%%%%%%%%%%%%%%%%%%%%%%%%%%%%%%%%

%%%%%%%%%%%%%%%%circle case 1%%%%%%%%%%%%%%%%%%%%%%%%%%%%%%%%%%%%
\filldraw[fill=black!10!white] (0,0) circle (2cm);

\foreach \angle in {0,20,40,60,80,100,120,140,160,180,200,220,240,260,280,300,320,340}
{\draw (\angle:1.95cm) -- (\angle:2cm);}

\draw (200:2cm) -- (340:2cm)  node [pos=0.5,fill=black!10!white]{$a$};
\draw (180:2cm) -- (80:2cm)  node [pos=0.5,fill=black!10!white]{$b$};
\draw[dashed] (80:2cm) -- (340:2cm)  node [pos=0.5,fill=black!10!white]{$e$};

%%%%%%%%%%%%%%%%end of circle case 1%%%%%%%%%%%%%%%%%%%%%%%%%%%%%%%

\end{tikzpicture}
\end{center}

\vspace*{0.5cm}

\begin{center}
\begin{tikzpicture}[thick,scale=0.6, every node/.style={scale=0.6}]

%%%%%%%%%%%%%%case 2%%%%%%%%%%%%%%%%%%%%%%%%%%%%%%%%%%%%%%%%%%%%%%%%
%a:
\fill (-10.5,0.5) circle (2pt);
\node (a) at (-10.5,0.5) {};
\node [left] at (a) {$a$};

\filldraw[fill=black!30!white] (-8.5,2) -- (-8,2.5) -- (-5.5,0) -- (-6,-0.5) -- (-8.5,2);

%b:
\fill (-7,1) circle (2pt);
\node (b) at (-7,1) {};
\node [right] at (b) {$b$};

%hom-hammock of tau-inverse of a:
\draw[dotted] (-9.5,0.5) -- (-8,2) -- (-5.5,-0.5) -- (-7,-2) -- (-9.5,0.5);

%section of a:
\draw (-10.5,0.5) -- (-9,2);

\draw (-10.5,0.5) -- (-9,-1);
\draw[dashed] (-9,-1) -- (-8,-2);

%section of b:
\draw (-8,2) -- (-7,1);

\draw (-9,-1) -- (-7,1);
\draw[dashed] (-10,-2) -- (-9,-1);

%middle terms:
\node (e) at (-9,-1) [circle,draw] {};
\node [above] at (e.north) {$e$};

%top and bottom lines:
\draw[thick] (-11,2) -- (-8.5,2);
\draw[dotted] (-8.5,2) -- (-7.5,2);
\draw[thick] (-7.5,2) -- (-5,2);

\draw[thick] (-11,-2) -- (-5,-2);

%%%%%%%%%%%%%%%%%%end of case 2%%%%%%%%%%%%%%%%%%%%%%%%%%%%%%%%%%%%%%

%%%%%%%%%%%%%%%%circle case 2%%%%%%%%%%%%%%%%%%%%%%%%%%%%%%%%%%%%
\filldraw[fill=black!10!white] (0,0) circle (2cm);

\foreach \angle in {0,20,40,60,80,100,120,140,160,180,200,220,240,260,280,300,320,340}
{\draw (\angle:1.95cm) -- (\angle:2cm);}

\draw (20:2cm) -- (160:2cm) node [pos=0.5,fill=black!10!white]{$a$};
\draw (260:2cm) -- (0:2cm) node [pos=0.5,fill=black!10!white]{$b$};
\draw[dashed] (160:2cm) -- (260:2cm) node [pos=0.5,fill=black!10!white]{$e$};

%%%%%%%%%%%%%%%%end of circle case 2%%%%%%%%%%%%%%%%%%%%%%%%%%%%%%%

\end{tikzpicture}
\end{center}

\vspace*{1cm}

\begin{center}
\begin{tikzpicture}[thick,scale=0.6, every node/.style={scale=0.6}]

%%%%%%%%%%%%%%case 3%%%%%%%%%%%%%%%%%%%%%%%%%%%%%%%%%%%%%%%%%%%%%%%%
%a:
\fill (-10.5,0.5) circle (2pt);
\node (a) at (-10.5,0.5) {};
\node [left] at (a) {$a$};

%boxes:
\filldraw[fill=black!20!white] (-10,0.5) -- (-8.5,2) -- (-6,-0.5) -- (-7.5,-2) -- (-10,0.5);

%b:
\fill (-7.5,-0.5) circle (2pt);
\node (b) at (-7.5,-0.5) {};
\node [right] at (b) {$b$};

%hom-hammock of tau-inverse of a:
\draw[dotted] (-9.5,0.5) -- (-8,2) -- (-5.5,-0.5) -- (-7,-2) -- (-9.5,0.5);

%section of a:
\draw (-10.5,0.5) -- (-8.5,-1.5);
\draw[dashed] (-8.5,-1.5) -- (-8,-2);

\draw (-10.5,0.5) -- (-9.5,1.5);
\draw[dashed] (-9.5,1.5) -- (-9,2);

%section of b:
\draw (-9.5,1.5) -- (-7.5,-0.5);
\draw[dashed] (-10,2) -- (-9.5,1.5);

\draw (-8.5,-1.5) -- (-7.5,-0.5);
\draw[dashed] (-9,-2) -- (-8.5,-1.5);

%middle terms:
\node (e1) at (-9.5,1.5) [circle,draw] {};
\node [left] at (e1.west) {$e_1$};

\node (e2) at (-8.5,-1.5) [circle,draw] {};
\node [left] at (e2.west) {$e_2$};

%top and bottom lines:
\draw[thick] (-11,2) -- (-5,2);

\draw[thick] (-11,-2) -- (-5,-2);

%%%%%%%%%%%%%%%%%%end of case 3%%%%%%%%%%%%%%%%%%%%%%%%%%%%%%%%%%%%%%

%%%%%%%%%%%%%%%%circle case 3%%%%%%%%%%%%%%%%%%%%%%%%%%%%%%%%%%%%
\filldraw[fill=black!10!white] (0,0) circle (2cm);

\foreach \angle in {0,20,40,60,80,100,120,140,160,180,200,220,240,260,280,300,320,340}
{\draw (\angle:1.95cm) -- (\angle:2cm);}

\draw (180:2cm) -- (0:2cm) node [pos=0.5,fill=black!10!white]{$a$};
\draw (140:2cm) -- (280:2cm) node [pos=0.6,fill=black!10!white]{$b$};
\draw[dashed] (140:2cm) -- (0:2cm) node [pos=0.5,fill=black!10!white]{$e_1$};
\draw[dashed] (180:2cm) -- (280:2cm) node [pos=0.3,fill=black!10!white]{$e_2$};

\draw[<-] (140:2.2cm) arc (140:180:2cm);

\draw[<-] (280:2.2cm) arc (280:360:2.2cm);

%%%%%%%%%%%%%%%%end of circle case 3%%%%%%%%%%%%%%%%%%%%%%%%%%%%%%%

\end{tikzpicture}
\end{center}

\vspace*{1cm}

\begin{center}
\begin{tikzpicture}[thick,scale=0.6, every node/.style={scale=0.6}]

%%%%%%%%%%%%%%case 4%%%%%%%%%%%%%%%%%%%%%%%%%%%%%%%%%%%%%%%%%%%%%%%%
%a:
\fill (-10.5,0.5) circle (2pt);
\node (a) at (-10.5,0.5) {};
\node [left] at (a) {$a$};

%boxes:
\filldraw[fill=black!10!white] (-6,-0.5) -- (-5.5,0) -- (-5,-0.5) -- (-5.5,-1) -- (-6,-0.5);

%b:
\fill (-5.5,-0.5) circle (2pt);
\node (b) at (-5.5,-0.5) {};
\node [right] at (b) {$b$};

%hom-hammock of tau-inverse of a:
\draw[dotted] (-9.5,0.5) -- (-8,2) -- (-5.5,-0.5) -- (-7,-2) -- (-9.5,0.5);

%section of a:
\draw (-10.5,0.5) -- (-9,2);

\draw (-10.5,0.5) -- (-8,-2);

%section of b:
\draw (-8,2) -- (-5.5,-0.5);

\draw (-7,-2) -- (-5.5,-0.5);

%top and bottom lines:
\draw[thick] (-11,2) -- (-5,2);

\draw[thick] (-11,-2) -- (-5,-2);

%%%%%%%%%%%%%%%%%%end of case 4%%%%%%%%%%%%%%%%%%%%%%%%%%%%%%%%%%%%%%

%%%%%%%%%%%%%%%%circle case 4%%%%%%%%%%%%%%%%%%%%%%%%%%%%%%%%%%%%
\filldraw[fill=black!10!white] (0,0) circle (2cm);

\foreach \angle in {0,20,40,60,80,100,120,140,160,180,200,220,240,260,280,300,320,340}
{\draw (\angle:1.95cm) -- (\angle:2cm);}

\draw (180:2cm) -- (0:2cm) node [pos=0.75,fill=black!10!white]{$a$};
\draw (160:2cm) -- (340:2cm) node [pos=0.25,fill=black!10!white]{$b$};

%%%%%%%%%%%%%%%%end of circle case 4%%%%%%%%%%%%%%%%%%%%%%%%%%%%%%%

\end{tikzpicture}
\end{center}
\caption{The middle term of the extension of $b$ by $a$. The arrows in the third case mean that the distance between the corresponding endpoints is a multiple of $m$.}
\label{fig:middle-term}
\end{figure}
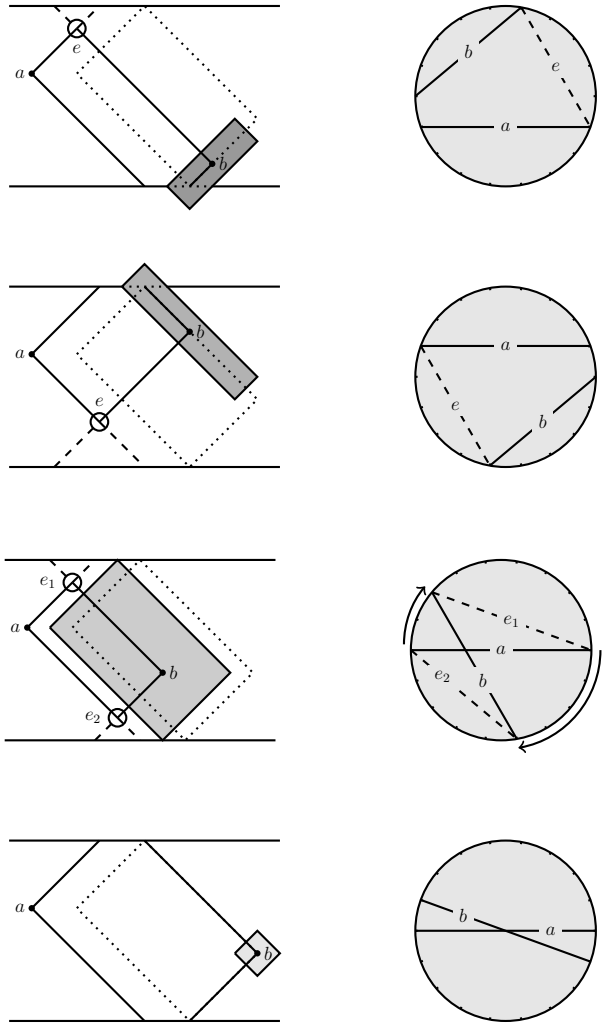

The arguments of Section~\ref{sec:non-indecouterterms} can be applied to $\Cm$ to give the following proposition, which makes it sufficient to consider only extensions between indecomposable objects.

\begin{proposition}\label{prop:extensionsnonindecomposableCm}
Let $\{a_i\}_{i=1}^n$ and $\{b_j \}_{j=1}^m$ be sets of (not necessarily pairwise non-isomorphic) indecomposable objects of $\Cm$. Any extension of the form
\[
\tri{\bigoplus_{i=1}^n a_i}{e}{\bigoplus_{j=1}^m b_j}\]
can be computed iteratively from extensions whose outer terms are indecomposable and built from $\{a_i \}_{i=1}^n$ and $\{b_j \}_{j=1}^m$. 
\end{proposition}

\subsection{Extension closure}
Let $\arc{a}$ and $\arc{b}$ be $m$-diagonals and recall the notion of $\tE(\arc{a},\arc{b})$ from Section \ref{sec:Ptolemy}. 

\begin{remark}
Let $m = 2$, $\arc{a} = \{i,i+1\}$ and $b = \Sigma a$. Note that $\arc{a}$ and $\arc{b}$ are incident with the vertex $i+1$, and this is the only case when extensions occur between noncrossing and non-neighbouring $m$-diagonals. In this case, the middle term of the extension of $b$ by $a$ is zero, and $\Ext^1_{\sC_{2}} (a,b) \simeq D \Hom_{\sC_{2}} (a, \tau^{-1} a) = 0$. Therefore, $\tE(\arc{a},\arc{b}) = \emptyset$.
\end{remark}

Figure \ref{fig:middle-term} (together with Proposition~\ref{prop:extensionsnonindecomposableCm}) shows us that $\tE(\arc{a},\arc{b})$ is contained in the set of Ptolemy $m$-diagonals associated to $\arc{a}$ and $\arc{b}$. Our aim is to prove that these two sets are in fact equal, giving Theorem~\ref{thm:orbit-extension-closure}. This follows from the following three propositions and Proposition~\ref{prop:extensionsnonindecomposableCm}.

\begin{proposition}
Let $a, b \in \ind{\Cm}$ be such that $\arc{a}$ and $\arc{b}$ are neighbouring $m$-diagonals. Then $\tE(\arc{a},\arc{b})$ contains all the Ptolemy $m$-diagonals of class II associated to $\arc{a}$ and $\arc{b}$. 
\end{proposition}
\begin{proof}
Similar to the proof of Proposition \ref{prop:neighbourotherinclusion}.
\end{proof}

\begin{proposition}
Let $a, b \in \ind{\Cm}$ be such that $\arc{a}$ and $\arc{b}$ are crossing $m$-diagonals and $\Ext^1_{\Cm} (b,a) \neq 0$. Then $\tE(\arc{a},\arc{b})$ contains all the Ptolemy $m$-diagonals of class I associated to $\arc{a}$ and $\arc{b}$. 
\end{proposition}
\begin{proof}
It follows immediately from the fact that the Ptolemy diagonals of class I associated to $\arc{a}$ and $\arc{b}$ other than $\arc{e_1}$ and $\arc{e_2}$ (see Figure \ref{fig:middle-term}) are not $m$-diagonals.
\end{proof}

\begin{proposition}
Let $a, b \in \ind{\Cm}$ be such that $\arc{a}$ and $\arc{b}$ are crossing $m$-diagonals and $\Ext^1_{\Cm} (a,b) = 0 = \Ext^1_{\Cm} (b,a)$. Then none of the corresponding Ptolemy diagonals is an $m$-diagonal.
\end{proposition}
\begin{proof}
By Remark~\ref{rmk:torsionpairsCvsT}, there is no extension between the corresponding $d=(w-1)$-admissible arcs in $\Tw$. Hence, by Lemma~\ref{lemma:extvanishnotadmissiblePtolemyarcs} the corresponding Ptolemy diagonals of class I are not $m$-diagonals.
\end{proof}

%==========================================================================
% The bibliography

%================================================================

\end{document}